\documentclass[11pt]{article}

\usepackage{amsmath}
\usepackage{amssymb}
\usepackage{amsthm}

\usepackage[utf8]{inputenc}
\usepackage[T1]{fontenc}

\usepackage{graphicx} 
\usepackage{tikz}

\usepackage[hidelinks]{hyperref}
\usepackage{cleveref} 
\usepackage{enumitem}
	\setlist[enumerate,itemize]{leftmargin=*}

\usepackage{cite}

\numberwithin{equation}{section}

\theoremstyle{plain}
\newtheorem{theorem}{Theorem}
\newtheorem{proposition}[theorem]{Proposition}
\newtheorem{corollary}[theorem]{Corollary}
\newtheorem{lemma}[theorem]{Lemma}

\theoremstyle{definition}
\newtheorem{definition}[theorem]{Definition}
\newtheorem{remark}[theorem]{Remark}

\numberwithin{theorem}{section}

\usepackage{pgfplots}
\pgfplotsset{compat=1.7}

\def\beq{\begin{equation}}
\def\beql#1{\beq\label{#1}}
\def\eeq{\end{equation}}
\def\beqa{\beq\begin{aligned}}
\def\beqal#1{\beql{#1}\begin{aligned}}
\def\eeqa{\end{aligned}\eeq}
\def\bseq{\begin{subequations}}
\def\bseql#1{\begin{subequations}\label{#1}}
\def\eseq{\end{subequations}}

\def\rb#1{\left(#1\right)}

\def\bigrb#1{\big(#1\big)}

\def\sqb#1{\left[#1\right]}

\def\curb#1{\left\{#1\right\}}
\def\bigcurb#1{\big\{#1\big\}}
\def\Bigcurb#1{\Big\{#1\Big\}}

\def\bigabs#1{\big|#1\big|}

\def\norm#1{\left\|#1\right\|}
\def\bignorm#1{\big\|#1\big\|}

\def\biggnorm#1{\bigg\|#1\bigg\|}
\def\angb#1{\left\langle #1 \right\rangle}
\def\lrang#1{\langle #1 \rangle}

\def\veps{\varepsilon}

\def\G{\Gamma}

\def\lam{\lambda}
\def\Lam{\Lambda}
\def\Om{\Omega}
\def\om{\omega}
\def\Del{\Delta}
\def\del{\delta}
\def\sig{\sigma}
\def\Sig{\Sigma}

\def\ka{\kappa}
\def\vphi{\varphi}

\def\wt{\widetilde}
\def\ol{\overline}

\def\R{\mathbb{R}}

\def\cB{\mathcal{B}}
\def\cC{\mathcal{C}}

\def\cH{\mathcal{H}}
\def\cI{\mathcal{I}}
\def\cJ{\mathcal{J}}
\def\cK{\mathcal{K}}
\def\cL{\mathcal{L}}
\def\cM{\mathcal{M}}

\def\cO{\mathcal{O}}
\def\cP{\mathcal{P}}

\def\cT{\mathcal{T}}

\def\cV{\mathcal{V}}

\def\sbt{\subset}

\def\p{\partial}

\def\what{\widehat}

\def\dvg{\operatorname{\nabla\cdot}}

\def\Span{\operatorname{span}}

\graphicspath{{./figures/}}
\usepackage{caption}
\usepackage{subcaption}
\usepackage{pgfplots}
\usetikzlibrary{plotmarks}
\usepackage{algorithm}
\usepackage[noend]{algorithmic}

\newcommand{\uCoef}{\alpha}
\newcommand{\dist}{\operatorname{dist}}
\newcommand{\supp}{\operatorname{supp}}

\newcommand{\interfaces}{\cM}

\begin{document}

\title{Error Estimates for Adaptive Spectral Decompositions}

\author{%
	Daniel H.\ Baffet\,%
		\footnotemark[1]\, \footnotemark[2]\qquad
	Yannik G.\ Gleichmann\,%
		\footnotemark[1]\, \footnotemark[3]\qquad
	Marcus J.\ Grote\,%
		\footnotemark[1]\, \footnotemark[4]%
}

{%
\renewcommand{\thefootnote}{\fnsymbol{footnote}}
\footnotetext[1]{Department of Mathematics and Computer Science,
		University of Basel, Basel, Switzerland}
\footnotetext[2]{daniel.baffet@unibas.ch}
\footnotetext[3]{yannik.gleichmann@unibas.ch}
\footnotetext[4]{marcus.grote@unibas.ch}
}%

\date{\today}
\maketitle

\begin{abstract}
Adaptive spectral (AS) decompositions associated with a piecewise constant function, $u$, yield small subspaces where the characteristic functions comprising $u$ are well approximated.
When combined with Newton-like optimization methods for the solution of inverse medium problems, AS decompositions have proved remarkably efficient in providing at each nonlinear iteration a low-dimensional search space.
Here, we derive $L^2$-error estimates for the AS decomposition of $u$, truncated after $K$ terms, when $u$ is piecewise constant and consists of $K$ characteristic functions over Lipschitz domains and a background.
Our estimates apply both to the continuous and the discrete Galerkin finite element setting.
Numerical examples illustrate the accuracy of the AS decomposition for media that either do, or do not, satisfy the assumptions of the theory.


\end{abstract}

\section{Introduction}
Adaptive spectral (AS) decompositions have been proposed as low-dimensional search spaces during the iterative solution of inverse medium problems \cite{BO2010, BK2013, GKN2017,GN2019,Baffet_2021}.
For piecewise constant media, in particular, AS decompositions have proved remarkably efficient and accurate.
So far, however, their remarkable approximation properties are only supported by numerical evidence.
Here, starting from \cite{Baffet_2021}, we derive $L^2$-error estimates for AS approximations of piecewise constant functions.

In \cite{BO2010}, 
De Buhan and Osses proposed to restrict the search space of an inverse medium problem to the span of a small basis of eigenfunctions of a judicious elliptic operator, repeatedly adapted during the nonlinear iteration.
Their adaptive inversion approach relies on a decomposition
\begin{equation}\label{eq:vExpansion}
	v=\sum_{k=1}^\infty \beta_k\vphi_k ,
\end{equation}
for $v\in W^{1,\infty}_0(\Om)$, with $\Om\sbt\R^d$.
Here each $\vphi_k$ is an eigenfunction of a $v$-dependent, linear, symmetric, and elliptic operator $L_\veps[v]$, i.e.,
\begin{equation}\label{eq:eigenValProb}
	L_\veps[v]\vphi_k =\lam_k \vphi_k
	\quad \text{in $\Omega$,}
	\qquad
	\vphi_k =0
        \quad \text{on $\partial\Omega$,}
\end{equation}
for an eigenvalue $\lam_k\in\R$.
In the sequel we shall in fact apply the AS decomposition to more general
functions in $W^{1,\infty}(\Om)$ by extending their boundary data appropriately into the interior of $\Om$;
here, for simplicity, we suppose $v\in W^{1,\infty}_0(\Om)$.

Clearly, the choice of $L_\veps[v]$ is crucial for obtaining an efficient approximation of $v$ with as few basis functions as possible. Typically, we use
\begin{equation}\label{IntroEq:linear_op}
	L_\veps[v]w=-\dvg\rb{\mu_{\veps}[v]\nabla w} ,
	\qquad
	\mu_{\veps}[v](x)=\frac{1}{\sqrt{|\nabla v(x)|^2+\veps^2}} \, ,
\end{equation}
where $\veps>0$ is a small parameter to avoid division by zero,
but other forms have also been used in the past and are treated by our analysis.

Note that we cannot apply the above AS decomposition directly to piecewise constant $u$, because
$\mu_\veps[u]$ is not in $L^\infty$ and thus $L_\veps[u]$ not well-defined.
Nevertheless, we may still decompose $u$ at the cost of an additional step.
We first approximate $u$ by a more regular approximation, which we denote generically by $u_\del$, where $\del>0$ is a parameter that controls the error and is proportional to the width of the support of $\nabla u_\del$ near the jump discontinuities of $u$.
Then we may expand $u$ (or $u_\del$) in the spectral basis of $L_\veps[u_\del]$ and obtain an approximate decomposition of $u$ (or $u_\del$) by truncating the expansion.
Typically, $u_\del$ corresponds to the standard, continuous, piecewise polynomial FE interpolant of $u$ 
on a regular triangulation with mesh size $\del=h$. Then, the eigenfunctions $\vphi_k$ may correspond either to the (true continuous) eigenfunctions of $L_\veps[u_\del]$ or to their (discrete approximate) Galerkin FE counterparts, 
as our analysis encompasses both the continuous and the discrete setting.

Insight about the AS decomposition approach may be obtained from its connection to the total variation (TV) functional, which is commonly used for image denoising while preserving edges.
In fact, $L_\veps[v]v$, with $L_\veps[v]$ given by \eqref{IntroEq:linear_op}, is the Fréchet derivative of the penalized TV functional -- see \cite[Remark 1]{GKN2017}.
The eigenvalue problem for $L_\veps[v]$ also bears a striking resemblance to nonlinear eigenvalue problems for the TV functional, which have been studied in the more general context of 1-homogeneous functionals for image processing -- see \cite{BGM2016,BCEGM2016,BCN2002} and the references therein.

The AS decomposition has been used as follows in various iterative Newton-like algorithms for the solution of inverse medium problems \cite{BK2013,GKN2017,GN2019}:
Given an approximation of the medium, $u^{(m-1)}$, from the previous iteration, the approximation $u^{(m)}$ at the current iteration is set as the minimizer of the misfit in the space $\Span(\vphi_k)_{k=1}^K$, where $\vphi_k$, $k=1,\ldots,K$, satisfy \eqref{eq:eigenValProb} with $v=u^{(m-1)}$. As the approximation $u^{(m)}$ changes from one iteration to the next, so does the search space used for the subsequent minimization.

By combining the adaptive inversion process with the TRAC (time reversed absorbing condition) approach, de Buhan and Kray~\cite{BK2013} developed an effective solution strategy for time-dependent inverse scattering problems.
In~\cite{GKN2017}, Grote, Kray and Nahum proposed the AEI (adaptive eigenspace inversion) algorithm for inverse scattering problems in the frequency domain.
In \cite{GN2019}, the AEI algorithm was extended to multi-parameter inverse medium problems.
Recently, it was extended to electromagnetic inverse scattering problems at fixed frequency \cite{BD2017} and also to time-dependent inverse scattering problems when the illuminating source is unknown~\cite{GGNA2019}.
In \cite{10.1093/gji/ggaa009}, AS decompositions were used for solving 2-D and 3-D seismic inverse problems for the Helmholtz equation.
First theoretical estimates for AS decompositions together with an approach for adapting the dimension of the search space were derived in \cite{Baffet_2021}.

When $u$ consists of a sum of $K$ characteristic functions $\chi_{A^k}$ of sets $A^k$, each compactly contained in $\Om$, the expansion \eqref{eq:vExpansion} in the spectral basis of $L_\veps[u_\del]$ truncated after $K$ terms has proved remarkably accurate, as it essentially recovers $u$ and in fact decomposes $u$ into the characteristic functions comprising it.
In \cite{Baffet_2021}, it is shown that the gradients of the first $K$ eigenfunctions of $L_\veps[u_\del]$ are small away from the discontinuities of $u$.
Thus, in regions where $u$ is constant, $\vphi_1,\ldots,\vphi_K$ are also nearly constant and we expect that in their span, $\Phi_K^{\veps,\del}=\Span\{\vphi_k\}_{k=1}^K$, $u$ together with each of the characteristic functions comprising it can be well approximated.
Here, our goal is to rigorously prove this proposition, even in the more general situation where $u$ is not necessarily constant near the boundary $\p\Om$.

Starting from \cite{Baffet_2021}, we derive $L^2$ error estimates for the projection of any $v\in u+\Span\{\chi_{A^k}\}$ onto the appropriate affine subspace.
In our main result, given by Theorem \ref{thm:main}, we prove that the $L^2$-projection error of $v$ (in particular of $u$ itself) is bounded by $\cO(\sqrt{\veps+\del})$.
Similarly, we show that any of the $K$ characteristic functions $\chi_{A^k}$ is approximated by its $L^2$-projection on $\Phi_K^{\veps,\del}$ up to $\cO(\sqrt{\veps+\del})$.
Our analysis treats both continuous AS formulations and their discrete Galerkin approximations.
The proof requires a technical result about the level sets of distance functions 
for Lipschitz domains, which is provided in Appendix \ref{sec:level_sets}.
In Corollary \ref{cor:main}, we particularize our estimates for two standard methods for obtaining $u_\del$.
In particular, our results apply when $u_\del$ is a continuous, piecewise polynomial interpolant of $u$ in a FE space $V_h$ with mesh size $h=\del$, and the eigenfunctions $\vphi_k$ are computed numerically by a Galerkin FE approximation in the same subspace.

The remainder of the paper is organized as follows.
In Section \ref{sec:preliminaries}, we describe the class of piecewise constant functions considered, provide definitions and introduce notation.
Section \ref{sec:analysis} contains the analysis and the main results of the paper.
Finally, we present in Section \ref{sec:num_res} various numerical examples which illustrate the accuracy of the AS decomposition for functions that either do, or do not, satisfy the assumptions of our theory.
There we also illustrate the usefulness of the AS decomposition for the solution of a standard linear inverse problem from image deconvolution.

\section{Notation and definitions}\label{sec:preliminaries}

The adaptive spectral (AS) decomposition \eqref{eq:vExpansion} of a function $v$ is based on the spectral decomposition of the $v$-dependent elliptic operator $L_\veps[v]$ given by
\begin{equation}\label{eq:linear_op}
	L_\veps[v]w=-\dvg\rb{\mu_{\veps}[v]\nabla w} .
\end{equation}
Typically, the weight function $\mu_\veps[v]$ has the form of either
\begin{equation}\label{eq:muq}
	\mu_\veps[v](x) =\ \frac{1}{(|\nabla v(x)|^q+\veps^q)^{1/q}} \, ,
\end{equation}
for some $q\in[1,\infty)$, or
\begin{equation}\label{eq:muinf}
	\mu_\veps[v](x) =\ \frac{1}{\max\{|\nabla v(x)|,\, \veps \}} \, .
\end{equation}
For the analysis below, however, we allow for more general $\mu_\veps[v]$.

Here our goal is to study the application of AS decompositions to regular approximations of piecewise constant functions.
Indeed, for a piecewise constant function $u$, $\mu_\veps[u]$ is not in $L^\infty$, and so $L_\veps[u]$ is not well defined.
Nevertheless, we may still decompose $u$ at the cost of an additional step.
We first approximate $u$ by a more regular approximation, which we denote generically by $u_\del$, where $\del$ is a parameter that controls the error and is proportional to the width of the support of $\nabla u_\del$ near the jump discontinuities of $u$.
Then we may expand $u$ (or $u_\del$) in the spectral basis of $L_\veps[u_\del]$, be it finite- or infinite-dimensional, and obtain an approximation by truncating the expansion.
One important example of a method for obtaining $u_\del$ is the standard, continuous, piecewise polynomial interpolant of $u$ in an $H^1$-conforming finite element (FE) space with underlying mesh size $\delta=h$.

To include FE approximations in the analysis, we formulate boundary value problems in closed subspaces $\cV^\del\sbt H^1(\Om)$ and $\cV^\del_0=\cV^\del \cap H^1_0(\Om)$.
Hence, in the continuous setting $\cV^\del = H^1(\Om)$, independently of $\delta$, whereas in the discrete FE setting $\cV^\del \subsetneq H^1(\Om)$ corresponds to the finite-dimensional FE space with underlying mesh size $\delta=h$.
As a consequence, all our results below are valid both for the continuous and the discrete setting and, in particular, for $H^1$-conforming FE approximations.
We let $\angb{\cdot,\cdot}$ and $\|\cdot\|_{L^2(\Om)}$ denote the standard inner product and norm of $L^2(\Om)$, and $|\cdot|$ denote the $\ell^2$-norm.
We use $C$, $C_1,C_2$, etc.\ to denote generic constants which may depend on $u$, but are independent of $\del$ and $\veps$; their values may also vary depending on the context.
We sometimes use the term ``medium'' to refer to functions on the domain of interest $\Om\sbt\R^d$.

In the remainder of this section, we introduce notation, assumptions and definitions needed for our approximation theory in Section \ref{sec:analysis}.
Section \ref{sec:medium} precisely defines the class of piecewise constant functions $u$ to be decomposed.
In Section \ref{sec:reg_approx}, we introduce admissible approximation methods for obtaining $u_\del$ and provide examples of two standard methods which are admissible.
In Section \ref{sec:weight_func}, we state our assumptions on the medium-dependent weight function $\mu_\veps[\cdot]$, and in Section \ref{sec:BVP} we state the boundary-value problems defining the spectral basis of $L_\veps[u_\del]$ and the $L_\veps[u_\del]$-lifting, $\vphi_0$, of the boundary data of $u_\del$ into~$\Om$.

\begin{figure}[t]
\centering
	\includegraphics[width=0.45\linewidth]{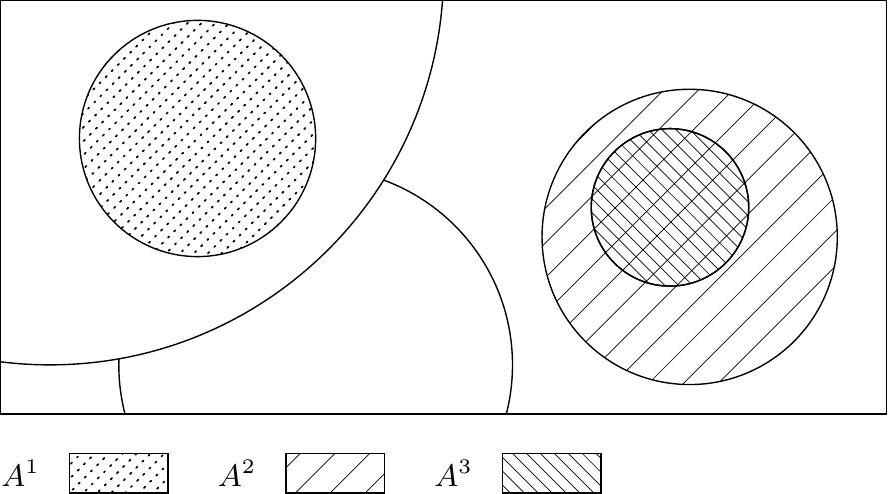} \quad
	\includegraphics[width=0.45\linewidth]{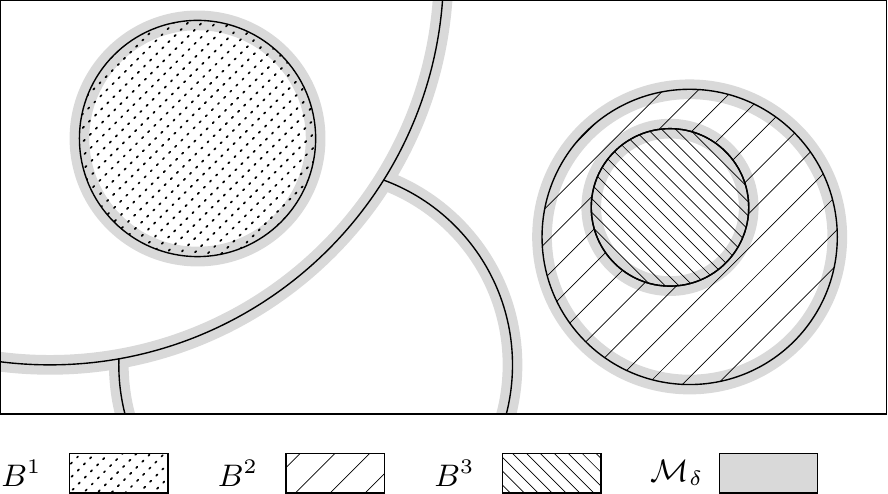}
	\caption{Typical configuration in two dimensions.
		In this example $K=3$ and $M=4$.
		The frame on the left shows the sets $A^1$, $A^2$ and $A^3$, and the frame on the right shows $B^1=A^1$, $B^2=A^2\setminus A^3$, and $B^3=A^3$.}
    \label{fig:2D-Illustration}
\end{figure}

\subsection{Piecewise constant medium}
\label{sec:medium}

Consider $u:\Om\to\R$ piecewise constant, where $\Om\sbt\R^d$, with $d\ge 2$, is a bounded Lipschitz domain.
We assume $u$ has the form
\begin{equation}\label{eq:u_decomp}
	u(x) = u^0(x)+\wt{u}(x) ,
	\qquad
	x\in\Om ,
\end{equation}
where the background $u^0$ and the interior inclusions $\wt u$ are given by
\begin{equation}\label{eq:u0_wtu}
	u^0 = \sum_{m=1}^{M} \om_{m}\chi_{\Om^{m}}\, , \quad \om_{m}\in\R ,
	\qquad
	\wt{u} = \sum_{k=1}^{K} \uCoef_{k}\chi_{A^k}\, , \quad \uCoef_{k}\in\R\setminus\{0\} ,
\end{equation}
with $\chi_{A}$ denoting the characteristic function of a set $A\sbt\R^d$.
In the decomposition \eqref{eq:u_decomp} we distinguish the sets $\Om^m$ connected to the boundary $\p\Om$ from those that are not.
We suppose the sets $\Om^{1},\ldots,\Om^{M}$ characterizing the background $u^0$ are disjoint Lipschitz domains covering $\Om$,
\[
	\ol{\Om} = \bigcup_{m=1}^M \ol{\Om^{m}} ,
\]
and for each $m$, $\p\Om^m\cap\p\Om$ is open in (the relative topology of) $\p\Om$, i.e,
\[
	\Om^m = \Om \cap \wt\Om^m,
	\qquad
	\p\Om \cap \wt\Om^m \ne \emptyset,
\]
for some bounded disjoint Lipschitz domain $\wt\Om^m\sbt\R^d$.
Moreover, we suppose $A^{1},\ldots,A^{K}$ are Lipschitz domains with mutually disjoint boundaries such that for each $k$, the boundary $\p A^{k}$ of $A^{k}$ is connected, and $A^{k}\sbt\sbt \Om^{m}$ for some $m=1,\ldots,M$.
Hence $\Om$ is partitioned into finitely many subdomains $\Om^{m}$ adjacent to its boundary $\p\Om$, while each $\Om^{m}$ may contain one or several inclusions $A^k$ isolated from $\p\Om$;
Figure \ref{fig:2D-Illustration} illustrates a possible configuration in two dimensions.

Note that $u$ given by \eqref{eq:u_decomp} is defined only a.e.\ in $\Om$.
This will be significant only in Section \ref{sec:reg_approx}, where we discuss admissible approximations of $u$; in the rest of the paper this will not cause ambiguity since there we always consider $u$ as an element of $L^2(\Omega)$.

\subsection{Admissible approximation}
\label{sec:reg_approx}

To employ the estimates derived in \cite{Baffet_2021}, we assume $u_\del$ is obtained by an \emph{admissible} method, i.e., by a method satisfying the following.

\begin{definition}\label{def:admiss_approx}
Consider a family of transformations $\cI_\del:L^2(\Om)\to \cV^\del\sbt H^1(\Omega)$, with $\del>0$ in some set of indices.
We say that $\{\cI_\del\}_{\del}$ is an \emph{admissible method}, if for every Lipschitz domain $A\subset \Omega$, the following conditions are satisfied:
\begin{enumerate}
\item
\begin{equation}\label{eq:L2_conv}
	\lim_{\del\to0} \|\cI_\del\chi_{A}-\chi_{A}\|_{L^2(\Om)}=0 .
\end{equation}

\item
\begin{equation}\label{eq:approx_reg}
	\nabla(\cI_\del\chi_{A})\in L^\infty(\Om),
	\quad
	\supp\!\bigrb{\nabla(\cI_\del\chi_{A})} \sbt \ol{U_\del} ,
\end{equation}
where
\[
	U_\del = \bigcurb{x\in\Om\, |\ \dist(x, \p\Om^m\cap\Om)<\del} ,
\]
with $\dist(x,W)$ denoting the distance of $x\in\R^d$ to the set $W\sbt\R^d$.

\item
There exists a constant $C$, such that for every $\del>0$ sufficiently small,
\begin{equation}\label{eq:chi_del_Linf_est}
	\del \|\nabla \cI_\del\chi_A \|_{L^\infty(\Om)}
	\ \le\ C .
\end{equation}

\item
If $\G\sbt\p\Om\setminus\p U_\del$ with positive $(d-1)$-dimensional Hausdorff measure, $\cH^{d-1}(\G)>0$, then the trace of $\chi_A$ on $\G$ coincides with that of $\cI_\del \chi_A$.
\end{enumerate}
\end{definition}

Hence, for convenience, we shall say that $u_\del$ obtained by an admissible method is an admissible approximation of $u$.
By Definition \ref{def:admiss_approx}, we have
\begin{equation}\label{eq:u_del_decomp}
	u_\del \ = \ u^0_\del+\wt{u}_\del ,
	\qquad
	u^0_\del = \cI_\del u^0 \in\cV^\del,
	\quad
	\wt{u}_\del = \cI_\del \wt{u} \in\cV^\del_0 .
\end{equation}
In addition, by \eqref{eq:approx_reg}, we have $\nabla u_\del=0$ in the open complement
\begin{equation}\label{eq:setD}
	D_\del = \Om\setminus \ol{\interfaces_\del} ,
\end{equation}
of the $\del$-wide neighborhood $\interfaces_\del$ of all interfaces,
\begin{equation}\label{eq:setM_del}
	\interfaces_\del = \bigcup_{k=1}^K \bigcurb{x\in\Om :\ \dist(x,\p A^{k})<\del}
		\cup \bigcup_{m=1}^M \bigcurb{x\in\Om :\ \dist(x,\p \Om^{m}\cap\Om)<\del} .
\end{equation}
By \eqref{eq:chi_del_Linf_est}, there exists a constant $C$ (which depends on $u$), such that for every $\del>0$ sufficiently small, $u_\del$ satisfies
\begin{equation}\label{eq:u_del_Linf_est}
	\del \|\nabla u_\del \|_{L^\infty(\Om)} \le C .
\end{equation}

Next we provide two examples \cite[Corollary 6]{Baffet_2021} of standard methods which are admissible.

\begin{figure}[t]
\centering
\includegraphics[width=0.45\linewidth]{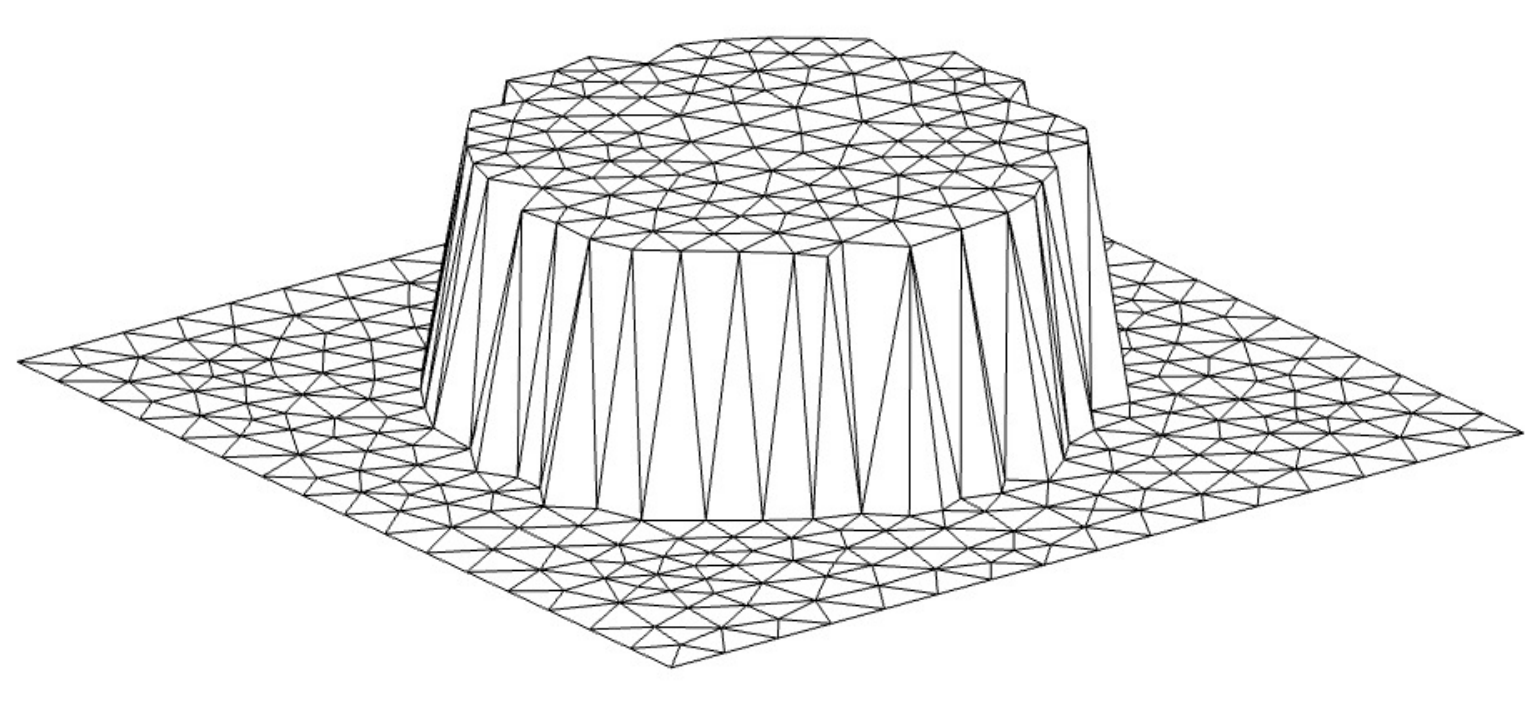}
\hfil\hfil
\includegraphics[width=0.45\linewidth]{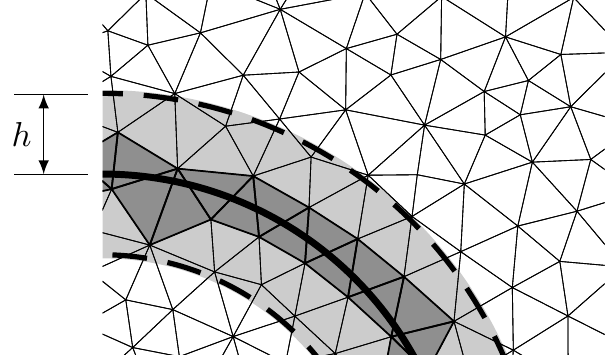}
\caption{The continuous, piecewise linear ${P}^1$-FE interpolant $u_h$ of the characteristic function $u$ for a disk.}
\label{fig:FE_examp}
\end{figure}
\begin{proposition}\label{prop:FE}
Let $u$ be extended to $\ol\Om$ either by assigning to $u$ at any $x$ on the interfaces in $\ol\Om$ one of its values in a neighboring domain $A^k$ or $\Om^m$, or by replacing $A^k$ and $\Om^m$ in \eqref{eq:u0_wtu} by $\ol{A^k}$ and $\ol{\Om^m}$.
For $h>0$, let $V_h$ denote an $H^1$-conforming $\cP^r$-FE space associated with a simplicial mesh $\cT_h$ with mesh size $h$.
If the family of meshes $\{\cT_h\}_{h}$ is regular and quasi-uniform (see, e.g., \cite{Q2008}), then the FE-interpolant $u_h$ of $u$ in $V_h$ is admissible.
\end{proposition}%

\begin{proof}
To prove the proposition, we have to verify the conditions of Definition \ref{def:admiss_approx}.
Most of the conditions, i.e., linearity, and conditions 1,2 and 4 are clearly satisfied.
The only condition which requires careful attention is 3, i.e., \eqref{eq:chi_del_Linf_est}.
The argument of the proof is similar to that of standard inverse inequalities.

Let $u_h$ be the FE-interpolant of the characteristic function $u=\chi_A$ of some set $A\sbt\Om$.
In every element $\cK\in\cT_h$, $u_h$ is the unique polynomial in $\cP^r$ which interpolates the values of $u$ ($0$ or $1$) at the nodes of $\cK$.
By transforming $\cK$ to the (mesh independent) reference element $\hat\cK$, we have
\[
	\nabla u_h(x) = J_\cK^{-T}\nabla P_\cK(F_\cK^{-1}(x))
	\qquad
	x\in\cK ,
\]
where $F_\cK^{-1}:\cK\to\hat\cK$ is the inverse of the affine mapping $F_\cK$ which transforms $\hat\cK$ to $\cK$, $J_\cK\in\R^{d\times d}$ is the Jacobian matrix of $F_\cK$ and $P_\cK\in\cP^r$ is a polynomial taking values of either $0$ or $1$ at the nodes of $\hat\cK$.
Since there is only a finite number of polynomials in $\cP^r$ whose image on the nodes of $\hat\cK$ is a subset of $\{0,1\}$, we can estimate all their gradients in $\hat\cK$ by a single constant independently of the element $\cK$ and mesh size $h$.
In addition, by \cite[Lemma 4.3]{Q2008} and the assumption that the family of meshes $\{\cT_h\}$ is regular and quasi-uniform, we have $\bigabs{J_{\cK}^{-1}} \le C h^{-1}$, where $|\cdot|$ denotes the matrix norm induced on $\R^{d\times d}$ by the $\ell^2$ norm of $\R^d$.
Note that while Lemma 4.3 in \cite{Q2008} is stated and proved in 2-D, its proof extends easily to any dimension.
Hence we obtain \eqref{eq:chi_del_Linf_est} with $\del=h$ which yields that interpolation in $V_h$ is an admissible approximation.
\end{proof}

The main effort in the proof of Proposition \ref{prop:FE} is to show \eqref{eq:chi_del_Linf_est} with $u=\chi_A$ a characteristic function.
Figure \ref{fig:FE_examp} illustrates this situation for the standard interpolant $u_h$ in a $\cP^1$-FE space of the characteristic function $u$ for a disk $A$.
The right frame shows a part of the mesh where the solid black line marks the discontinuity of $u$ along $\p A$.
Outside the dark gray elements, $u$ is constant and therefore so is $u_h$.
In particular, $\nabla u_h=0$ outside the neighborhood of width $\del=h$ (light gray) around $\p A$.

\begin{proposition}\label{prop:conv}
If $u$ is extended to a.e.\ $x\in\R^d$ by
\begin{equation}
	u =\sum_{m=1}^{M} \om_{m}\chi_{\wt{\Om}^{m}}
		+ \sum_{k=1}^{K} \uCoef_{k}\chi_{A^k}
\end{equation}
(compare with \eqref{eq:u_decomp}, \eqref{eq:u0_wtu}), and $u_\del$ is the convolution
\begin{equation}
	u_\del(x) = \zeta_\del* u = \int_{\R^d} \zeta_\del(x-y) u(y)\, d y ,
	\qquad
	\zeta_\del(x)=\del^{-d}\zeta(x/\del)
\end{equation}
with $\zeta$ a standard mollifier (e.g., \cite{EG1992}), then $u_\del$ is admissible.
\end{proposition}

\begin{proof}
See Corollary 6 of \cite{Baffet_2021}.
\end{proof}

For the analysis below it is convenient to partition $D_\del$, given by \eqref{eq:setD}, into its connected components.
Hence, we let the sets $A^1,\ldots,A^K$ be indexed so that if $i>k$, then either $A^i \sbt A^k$ or $A^i \cap A^k=\emptyset$, and let $B^k_\del$ be the connected components of $D_\del$,
\begin{equation}
	B^k_\del = B^k \cap D_\del ,
	\qquad
	B^k = A^k\setminus \bigcup_{i> k} \ol{A^i} ,
	\qquad
	k=1,\ldots,K ;
\end{equation}
see Figure \ref{fig:2D-Illustration}.
Similarly, we define outside the inclusions
\begin{equation}\label{eq:Em_def}
	E^m_\del = E^m \cap D_\del ,
	\qquad
	E^m = \Om^m\setminus \bigcup_{k=1}^K \ol{A^k} ,
	\qquad
	m=1,\ldots,M .
\end{equation}
Here, we assume $\del>0$ sufficiently small so that $B^k_\del$ and $E^m_\del$ are indeed connected and that the $(d-1)$-dimensional Hausdorff measure of $\p E^m_\del\cap\p\Om$ is positive.
Thus, for each $k$ and $\del>0$ small, $B^k$ and $B^k_\del$ are open and connected, and $D_\del$ is given by the disjoint union
\[
	D_\del = E_\del \cup \bigcup_{k=1}^K B^k_\del ,
\]
where $E_\del$ denotes the ``$\del$-exterior'',
\begin{equation}
	E_\del = \bigcup_{m=1}^M E^m_\del .
\end{equation}
Now we may deduce from condition 4 of Definition \ref{def:admiss_approx} that
\begin{equation}\label{eq:u_u_del_E}
	u=u_\del=u^0_\del ,
	\qquad
	\wt u=\wt u_\del=0
	\qquad\text{a.e.\ in $E_\del$.}
\end{equation}

Since we have a finite number of Lipschitz domains, $B^k$ ($k=1,\ldots,K$) and $E^m$ ($m=1,\ldots,M$), we may find a single constant $\Lam>0$ sufficiently large so that each of the sets, near its boundary, locally coincides with the epigraph of a $\Lam$-Lipschitz function.
While the optimal Lipschitz constant for a domain may depend on the scale of the open sets used for covering its boundary, when reducing the scale, the optimal constant cannot increase.
Therefore, if for some scale the Lipschitz constant $\Lam$ is suitable for a domain, for simplicity, we shall say that it is a $\Lambda$-Lipschitz domain.

By Theorem \ref{thm:dist_set}, for every sufficiently small $\del$, each $B^k_\del$ is also a $\Lam$-Lipschitz domain.
Note, however, that since a portion of the boundary of $E^m_\del$ coincides with the boundary of $E^m$ for every $\del$, it does not have the form assumed in Theorem \ref{thm:dist_set}.
As a result, we cannot rely on the same theorem to deduce that $E^m_\del$ is a $\Lam$-Lipschitz domain.
Nevertheless, outside a neighborhood of $\p\Om\cap\p E^m$, the boundary of $E^m_\del$ is a $\Lam$-Lipschitz surface with $E^m_\del$ lying to one of its sides, by Theorem \ref{thm:dist_graph}.
It is therefore possible to modify the definition of $\interfaces_\del$ so that for every $\del$ sufficiently small, $E^m_\del$ given by \eqref{eq:Em_def}, is a $\wt\Lam$-Lipschitz domain, for some $\wt\Lam$ independent of $\del$.
Here, for simplicity, we assume the latter to be true and denote the uniform constant $\max(\Lam,\wt\Lam)$ again by $\Lam$.

\subsection{Medium dependent weight function}
\label{sec:weight_func}

For $\veps>0$ and $v\in H^1(\Om)$, with $\nabla v\in L^{\infty}(\Om)$, we assume the $v$-dependent weight function $\mu_\veps[v]$ has the form
\begin{equation}\label{eq:mu_hatmu}
	\mu_\veps[v](x)=\hat\mu_\veps(|\nabla v(x)|) ,
	\qquad
	x\in\Om ,
\end{equation}
where $\hat\mu_\veps:[0,\infty)\to\R$ is a non-increasing function that satisfies
\begin{equation}\label{eq:mu_cond}
	\hat\mu_\veps(0)=\veps^{-1} ,
	\qquad
	0<\hat\mu_\veps(t) ,
	\quad
	t \hat\mu_\veps(t)\le 1 , \quad t\ge 0 ,
\end{equation}
and
\begin{equation}\label{eq:hat_mu_est}
	\text{$\exists C>0$, s.t.\ for every sufficiently large $t$, $C \le t\hat\mu_\veps(t)$.}
\end{equation}
In particular, for $\hat\mu_\veps(t)=1/(t^q+\veps^q)^{1/q}$ and $\hat\mu_\veps(t)=1/\max(t,\veps)$, as in \eqref{eq:muq} and \eqref{eq:muinf}, respectively, \eqref{eq:mu_cond}-\eqref{eq:hat_mu_est} hold for any $C<1$.
From \eqref{eq:mu_cond}, we immediately conclude that
\begin{equation}\label{eq:REstMu}
	\mu_\veps[v](x)|\nabla v(x)| \le 1 ,
	\qquad
	\text{a.e. $x\in\Om$,}
\end{equation}
and
\begin{equation}\label{eq:LEstMu}
	0 < \hat\mu_\veps(\|\nabla v\|_{L^\infty(\Om)}) \le \mu_\veps[v](x)
	\qquad
	\text{a.e. $x\in\Om$.}
\end{equation}

\subsection{Boundary value problems}
\label{sec:BVP}

Let $\cV^\del$ be a closed subspace of $H^1(\Om)$, possibly equal to $H^1(\Om)$, and $\cV^\del_0= \cV^\del\cap H^1_0(\Om)$.
For sufficiently small and fixed $\del,\veps>0$, the operator $L_\veps[u_\del]$ in \eqref{eq:linear_op} is uniformly elliptic in $\Om$ \cite{Baffet_2021}.
Thus, it admits in $\cV^\del_0$ a (possibly finite) non-decreasing sequence $\{\lam_k\}_{k\ge 1}$ of positive eigenvalues with each repeated according to its multiplicity with corresponding eigenfunctions $\{\vphi_k\}_{k\ge 1}$ which form an $L^2$-orthonormal basis of $\cV^\del_0$.
In addition, we denote by $\vphi_0\in\cV^\del$ the $L_\veps[u_\del]$-lifting of the boundary data of $u_\del$ into $\Om$.
More precisely, we let $\vphi_0\in\cV^\del$ satisfy
\begin{equation}\label{eq:phi0BVP_analysis}
	L_\veps[u_\del]\vphi_0 =0
	\quad \text{in $\Omega$,}
	\qquad
	\vphi_0 = u_\del
	\quad \text{on $\partial\Omega$}
\end{equation}
in $\cV^\del_0$, and for $k\ge1$ we let $\vphi_k\in\cV^\del_0$, $\vphi_k\ne0$ satisfy
\begin{equation}\label{eq:eigenValProb_analysis}
	L_\veps[u_\del]\vphi_k =\lam_k \vphi_k
	\quad \text{in $\Omega$,}
	\qquad
	\vphi_k =0
        \quad \text{on $\partial\Omega$,}
\end{equation}
in $\cV^\del_0$.
Clearly both \eqref{eq:phi0BVP_analysis} and \eqref{eq:eigenValProb_analysis} should be understood in a weak sense with respect to the bilinear form
\begin{equation}\label{eq:eigenValProb_bilinearform}
	B_{\veps,\del}[w,v] = \angb{\mu_{\veps}[u_\del] \nabla w, \nabla v} .
\end{equation}

For instance, if $\cV^\del$ is a (finite-dimensional, $H^1$-conforming) FE space, the eigenvalue problem \eqref{eq:eigenValProb_analysis} is understood as the Galerkin FE formulation: find $\vphi_k\in\cV_0^{\del}$ and $\lam_k\in\R$ such that
\begin{equation}
	B_{\veps,\del}[\vphi_k, \vphi] = \lam_k\angb{\vphi_k,\vphi}
	\qquad
	\forall\, \vphi\in\cV_0^{\del} .
\end{equation}
Thus, the framework above treats both continuous and discrete formulations.

\begin{remark}
Note that $\vphi_k$ ($k\ge0$) and $\lam_k$ ($k\ge1$) always depend on $\veps$ and $u_\del$, and thus on $u$ and $\del$, regardless of any particular finite- or infinite-dimensional choice for $\cV^\del$.
For simplicity of notation, we do not indicate this dependency explicitly.
\end{remark}

\section{Error estimates}\label{sec:analysis}

Given a piecewise constant $u$, we shall now derive our estimates for the AS decomposition of $u_\del$ based on the assumptions and definitions introduced in Section \ref{sec:preliminaries}.
Since $u_\del$ is an admissible approximation of $u$, as defined in Section \ref{sec:reg_approx}, for every $\veps>0$ and every sufficiently small $\del>0$, we have \cite{Baffet_2021}
\begin{equation}\label{eq:est0}
	B_{\veps,\del}[v,v] \le C ,
	\qquad
	v\in\curb{u_\del,\, u^0_\del,\, \wt u_\del,\, \vphi_0,\ldots,\, \vphi_K} .
\end{equation}
Here, and in the rest of the paper, the constants $C,C_1,C_2,\ldots$ may depend on $u$ (i.e., on its values and on the sets $B^k$ and $\Om^m$), but not on $\veps,\del$.
As a consequence of \eqref{eq:est0}, the gradients of $\vphi_k$, with $k=0,\ldots,K$, are small in $D_\del$ \cite[Theorem 5]{Baffet_2021}.
Heuristically, this implies that each $\vphi_k$ is almost constant in regions where $u$ is constant and thus we expect that $u$ be well approximated in $\vphi_0+\Phi_K^{\veps,\del}$, where
\begin{equation}\label{eq:Phi_K_def}
	\Phi_K^{\veps,\del} = \Span\{\vphi_k\}_{k=1}^K .
\end{equation}
Here, our goal is to rigorously prove this proposition.

More precisely, let $\Pi_K^{\veps}[u_\del]$ denote the standard orthogonal projection on $\Phi_K^{\veps,\del}$:
\begin{equation}\label{eq:orthogonal-projector-PK}
	\Pi_K^{\veps}[u_\del]:L^2(\Om)\to\Phi_K^{\veps,\del} ,
	\qquad
	\angb{v-\Pi_K^{\veps}[u_\del]v,\vphi}=0 ,
	\quad
	\forall \, \vphi \in \Phi_K^{\veps,\del} ,
\end{equation}
and let $X_K$ be given by
\begin{equation}\label{eq:X_K_def}
	X_K =\Span\{\chi_{A^k}\}_{k=1}^K = \Span\{\chi_{B^k}\}_{k=1}^K .
\end{equation}
We shall show that every function $v\in u+X_K$ is well approximated in $\vphi_0+\Phi_K^{\veps,\del}$ by its $L^2$-orthogonal projection
\begin{equation}\label{eq:best_L2_affine_QK}
	Q_K^{\veps}[u_\del](v)= \vphi_0+\Pi_K^{\veps}[u_\del](v-\vphi_0) .
\end{equation}
Similarly, we shall show that every $v\in X_K$ is well approximated by its orthogonal projection $\Pi_K^{\veps}[u_\del]v$ on $\Phi_K^{\veps,\del}$.
The main result, given by Theorem \ref{thm:main}, provides estimates of the $L^2$ errors in terms of $\veps$ and~$\del$.


\subsection{Preliminary results}

From \eqref{eq:LEstMu} with $v=u_\del$, the monotonicity of $\hat\mu$, \eqref{eq:u_del_Linf_est} and \eqref{eq:hat_mu_est} we get
\begin{equation}\label{eq:LEstMu_udel}
	0<C\del \le \mu_\veps[u_\del](x)
	\qquad\text{a.e.\ $x\in\Om$}
\end{equation}
for every sufficiently small $\del$, where the constant $C$ may depend on $u$, but is independent of $\del$ and $\veps$.
Since $\nabla u_\del$ vanishes in $D_\del$ by \eqref{eq:approx_reg}, assumptions \eqref{eq:mu_hatmu} and \eqref{eq:mu_cond} on $\hat\mu_\veps$ yield
\begin{equation}\label{eq:mu_eps_D_del}
	\mu_\veps[u_\del](x) = \veps^{-1}
	\qquad\text{a.e.\ $x\in D_\del$.}
\end{equation}
Together with the definition of $B_{\veps,\del}[\cdot,\cdot]$ in \eqref{eq:eigenValProb_bilinearform}, and \eqref{eq:LEstMu_udel} we obtain
\begin{equation}\label{eq:est0.5}
	\veps^{-1}\|\nabla v\|_{L^2(D_\del)}^2 +C_1\del\|\nabla v\|_{L^2(\interfaces_\del)}^2
		\le B_{\veps,\del}[v,v]
\end{equation}
for every $\del>0$ sufficiently small and every $v\in H^1(\Om)$.
By substituting $v=\vphi_k$ in the above and using \eqref{eq:est0} we get
\begin{equation}\label{eq:est1}
	\veps^{-1}\|\nabla\vphi_k\|_{L^2(D_\del)}^2 +C_1\del\|\nabla\vphi_k\|_{L^2(\interfaces_\del)}^2
		\le B_{\veps,\del}[\vphi_k,\vphi_k] \le C .
\end{equation}

Next we employ \eqref{eq:est1} and Poincaré-type inequalities to obtain $L^2$ estimates for $\vphi_k$ in $D_\del$.
To do that we require inequalities with constants independent of $\del$ for the connected components of $D_\del$.
We use Theorems 1 and 2 of \cite{doi:10.1080/03605300600910241} which yield the following:
Let $p\ge 1$ and $\Lam>0$.
There exists a constant $C>0$ such that for every $\Lam$-Lipschitz domain $W\sbt\Om$ and $v\in W^{1,p}(W)$,
\begin{equation}\label{eq:Poin_ineq_avg}
	\|v-\lrang{v}_{W}\|_{L^p(W)} \le C\|\nabla v \|_{L^p(W)} ,
	\qquad
	\forall v\in W^{1,p}(W) ,
\end{equation}
where $\lrang{f}_{W}$ denotes the average of $f$ over $W$,
\begin{equation}
	\lrang{f}_{W} = \frac{1}{\cL(W)} \int_W f(x) dx ,
\end{equation}
with $\cL(W)$ the Lebesgue measure of $W$.
Moreover, if $\G\sbt\ol{\Om}$ has positive $(d-1)$-dimensional Hausdorff measure, then for every $\Lam$-Lipschitz domain $W\sbt\Om$, with $\G\sbt\p W$, and $v\in W^{1,p}(W)$ satisfying $v=0$ on $\G$,
\begin{equation}\label{eq:Poin_ineq0}
	\|v \|_{L^p(W)} \le C\|\nabla v \|_{L^p(W)} .
\end{equation}

\begin{corollary}\label{cor:prelim_est}
There exists a constant $C>0$ such that for every $\veps>0$, $\del>0$ sufficiently small and $1\le j\le K$,
\begin{equation}\label{eq:est_phi_0}
	\|\vphi_0-u^0\|_{L^2(E_\del)}^2 \le C\veps ,
	\qquad
	\|\vphi_0-\lrang{\vphi_0}_{B^j_\del}\|_{L^2(B^j_\del)}^2 \le C\veps
\end{equation}
and
\begin{equation}\label{eq:est_phi_k}
	\|\vphi_k\|_{L^2(E_\del)}^2 \le C\veps ,
	\qquad
	\|\vphi_k-\lrang{\vphi_k}_{B^j_\del}\|_{L^2(B^j_\del)}^2 \le C\veps ,
	\qquad
	k=1,\ldots,K .
\end{equation}
\end{corollary}

\begin{proof}
We show \eqref{eq:est_phi_0}; the proof of \eqref{eq:est_phi_k} is similar.
Fix $1\le m\le M$.
Then, for every sufficiently small $\del$, we have $\eta=\vphi_0-u^0 \in H^1(E^m_\del)$ with $\eta=0$ on
\[
	\G^m = \p\Om\cap \p E^m_\del ,
\]
by \eqref{eq:u_u_del_E}.
As $\G^m$ contains an open set in the topology of $\p\Om$, its $(d-1)$-dimensional Hausdorff measure is positive.
Since $E^m_\del$ is $\Lam$-Lipschitz, with $\Lam$ independent of $\del$, by Poincaré \eqref{eq:Poin_ineq0}, there exists $C_1>0$ such that
\begin{equation}
	\|\eta\|_{L^2(E^m_\del)} \le C_1 \| \nabla \eta\|_{L^2(E^m_\del)} .
\end{equation}
Now, we use the above combined with \eqref{eq:est1} and $\nabla u^0=0$ in $E^m_\del$, to obtain
\begin{equation}
	\|\vphi_0-u^0\|_{L^2(E^m_\del)} = \|\eta\|_{L^2(E^m_\del)}
		\le C_1 \| \nabla \vphi_0\|_{L^2(E^m_\del)} \le C_2 \sqrt{\veps} ,
\end{equation}
which proves the first estimate in \eqref{eq:est_phi_0}, since $E_\del$ is the disjoint (finite) union of $E^m_\del$.
The proof of the second estimate in \eqref{eq:est_phi_0} is similar, but relies on \eqref{eq:Poin_ineq_avg} instead of \eqref{eq:Poin_ineq0}; therefore, it is omitted here.
\end{proof}


While Corollary \ref{cor:prelim_est} provides $L^2$ estimates for $\vphi_k$ in the connected components of $D_\del$, the following lemma provides $L^2$ estimates in the neighborhood $\interfaces_\del$ of the discontinuities of $u$.
Especially, it yields that the contribution over $\interfaces_\del$ to the norm of $\vphi_k$ is small.
Note that to deduce this conclusion it is not enough to observe that the volume of $\interfaces_\del$ is small, since the functions $\vphi_k$ themselves depend on $\del$.

\begin{lemma}\label{lem:U_del_est}
There exists a positive constant $C$ such that for every sufficiently small $\veps,\del>0$,
\begin{align}
	\label{eq:U_est_0}
	\|\vphi_0-u^0\|_{L^2(\interfaces_\del)}^2 &\le C\del \\
	\label{eq:U_est_k}
	\|\vphi_k\|_{L^2(\interfaces_\del)}^2 &\le C\del \qquad k=1,\ldots,K.
\end{align}
\end{lemma}

\begin{proof}
Here we show only \eqref{eq:U_est_k}.
We include estimate \eqref{eq:U_est_0} here only for brevity; its proof is similar, though it requires Lemma \ref{lem:phi_0_L2_est}.
Thus, the correct order of our argument is \eqref{eq:U_est_k}, Lemma \ref{prop:avg_inv}, Theorem \ref{thm:lam_1_lower_est}, Lemma \ref{lem:phi_0_L2_est}, and then \eqref{eq:U_est_0}.
Note that by \eqref{eq:U_est_0} we have that $\vphi_0$ also satisfies \eqref{eq:U_est_k}.

Fix $1\le k\le K$, let $W=B^j$ for some $j=1,\ldots,K$ or $W=\Om^{m}$ for some $m=1,\ldots,M$, and let
\begin{equation}\label{eq:U_est_proof_1_U}
	U_\del =\curb{x\in W :\ \dist(x,\p W)<\del} .
\end{equation}
By Theorem \ref{thm:gen_est_U} we have
\begin{equation}\label{eq:U_est_proof_1}
	\|\vphi_k\|_{L^2(U_\del)}^2
		\le C\rb{\del^{2}\|\nabla \vphi_k\|_{L^2(U_\del)}^2
			+\del\|\vphi_k\|_{H^1(D_\del)}^2} .
\end{equation}
By using $\|\vphi_k\|_{L^2(\Om)}=1$ and \eqref{eq:est1}, we estimate the right hand side of \eqref{eq:U_est_proof_1} and thus for $\del,\veps$ sufficiently small obtain
\begin{equation}
	\|\vphi_k\|_{L^2(U_\del)}^2 \le C_1\del\rb{1+\veps}\le C\del .
\end{equation}
Since $\interfaces_\del$ is a subset of the finite union of all $\ol{U_\del}$, we obtain \eqref{eq:U_est_k} which completes the proof.
\end{proof}


Following Corollary \ref{cor:prelim_est} and Lemma \ref{lem:U_del_est} we know that $\vphi_1,\ldots,\vphi_K$ are approximately piecewise constant, and that the contributions over $\interfaces_\del$ to their $L^2$-norms are small.
This implies that each $\vphi_k$ is close to some function in $X_K$.
As we shall see in Theorem \ref{thm:main}, the converse is also true; i.e., every function in $X_K$ can be well approximated in $\Phi_K^{\veps,\del}$.
Here -- because in each $B^k_\del$, $\vphi_1,\ldots,\vphi_K$ are close to their averages -- this proposition reduces to the invertibility of the matrix of the averages $\lrang{\vphi_j}_{B^k_\del}$.

\begin{lemma}\label{prop:avg_inv}
Let the matrix $\Sig\in\R^{K\times K}$ be given by
\begin{equation}\label{eq:avg_mat}
	\Sig=(\sig_{kj}) ,
	\qquad
	\sig_{kj} =\lrang{\vphi_j}_{B^k_\del} ,
		\quad k,j=1,\ldots,K .
\end{equation}
There exist constants $0<C_1\le C_2$ such that for every sufficiently small $\del$ and $\veps$,
\begin{equation}\label{eq:avg_mat_est}
	C_1|\beta| \le |\Sig\beta| \le C_2|\beta| ,
	\qquad
	\beta\in\R^K .
\end{equation}
\end{lemma}

\begin{proof}
Since the upper estimate in \eqref{eq:avg_mat_est} is simple, here we only show the lower estimate $C_1|\beta|\le |\Sig\beta|$, for some positive constant $C_1$ independent of $\beta$, $\veps$, and $\del$.
Let $\beta\in\R^K$ with $|\beta|=1$ and $\vphi\in\Phi_K^{\veps,\del}$ be given by
\begin{equation}\label{eq:zero_avgs}
	\vphi =\sum_{j=1}^K \beta_j\vphi_j .
\end{equation}
Then, we have
\begin{equation}\label{eq:avg_mat_phi}
	(\Sig\beta)_k = \lrang{\vphi}_{B^k_\del} ,
	\qquad
	k=1,\ldots,K ,
\end{equation}
where $(\Sig\beta)_k$ denotes the $k$-th entry of $\Sig\beta$.
Since $\vphi_1,\ldots,\vphi_K$ are orthonormal and $|\beta|=1$, we get
\begin{equation}
	\label{eq:avg_phi_norm}
	1=\|\vphi\|_{L^2(\Om)}^2 =\|\vphi\|_{L^2(E_\del)}^2+\|\vphi\|_{L^2(\interfaces_\del)}^2
		+\sum_{k=1}^K \|\vphi\|_{L^2(B^k_\del)}^2 .
\end{equation}
Due to \eqref{eq:avg_mat_phi}, the function $\vphi -(\Sig\beta)_k$ has zero average over $B^k_\del$ and is, therefore, orthogonal to the constant in $L^2(B^k_\del)$.
Thus,
\begin{equation}
	\|\vphi\|_{L^2(B^k_\del)}^2 = \|\vphi -(\Sig\beta)_k\|_{L^2(B^k_\del)}^2 +\cL(B^k_\del)(\Sig\beta)_k^2 ,
	\qquad
	k=1,\ldots, K .
\end{equation}
By Poincaré's inequality \eqref{eq:Poin_ineq_avg},
\begin{equation}
	\|\vphi -(\Sig\beta)_k\|_{L^2(B^k_\del)}^2 \le C\|\nabla \vphi\|_{L^2(B^k_\del)}^2
\end{equation}
and by the triangle inequality and \eqref{eq:est1}, we have
\begin{equation}
	\|\nabla \vphi\|_{L^2(B^k_\del)} \le \sum_{j=1}^K |\beta_j|\|\nabla\vphi_j\|_{L^2(B^k_\del)}
		\le C\sqrt{\veps} .
\end{equation}
We further estimate $\|\vphi\|_{L^2(E_\del)}$ and $\|\vphi\|_{L^2(\interfaces_\del)}$ in \eqref{eq:avg_phi_norm} using \eqref{eq:est_phi_k} and \eqref{eq:U_est_k}, and use that $B^k_\del\sbt B^k$, to obtain
\begin{equation}
	1 \le C(\veps+\del) +\sum_{k=1}^K \cL(B^k_\del)(\Sig\beta)_k^2
		\le C(\veps+\del) +\max_{k} \cL(B^k)\, |\Sig\beta|^2 .
\end{equation}
Thus, for every $\del$ and $\veps$ sufficiently small,
\begin{equation}
	\wt{C} \le \max_{k} \cL(B^k) \, |\Sig\beta|^2
\end{equation}
which completes the proof.
\end{proof}

\subsection{Main results}

Next we show that if $\veps,\del>0$ are sufficiently small, then the first eigenvalue of $L_\veps[u_\del]$ is bounded from below by a constant independent of $\veps,\del$.

\begin{theorem}\label{thm:lam_1_lower_est}
There exists a positive constant $C$ such that for every $\veps,\del>0$ sufficiently small and for every $v\in H^1_0(\Om)$,
\begin{equation}\label{eq:lam_1_lower_est}
	C\|\nabla v\|_{L^1(\Om)} \le \sqrt{B_{\veps,\del}[v,v]} ,
	\qquad
	C\|v\|_{L^2(\Om)}^2 \le B_{\veps,\del}[v,v] .
\end{equation}
In particular, the second estimate yields $\lam_1\ge C>0$.
\end{theorem}

\begin{proof}
We begin by showing the first estimate of \eqref{eq:lam_1_lower_est}.
Let $v\in H^1_0(\Om)$.
Hölder's inequality and Lemma 4 of \cite{Baffet_2021} yield
\begin{equation}
	\del \|\nabla v\|_{L^2(\interfaces_\del)}^2
		\ge \frac{\del}{\cL(\interfaces_\del)} \|\nabla v\|_{L^1(\interfaces_\del)}^2 
		\ge C \|\nabla v\|_{L^1(\interfaces_\del)}^2.
\end{equation}
Similarly we use Hölder's inequality to estimate $\|\nabla v\|_{L^2(D_\del)}^2$ from below by $C\|\nabla v\|_{L^1(D_\del)}^2$ and thus, by \eqref{eq:est0.5}, for $\veps>0$ sufficiently small we get
\begin{equation}\label{eq:proof_lam_1_lower_est_1}
	C_1 \|\nabla v\|_{L^1(\Om)}^2 \le B_{\veps,\del}[v,v] ,
\end{equation}
which is equivalent to the first estimate of \eqref{eq:lam_1_lower_est}.

To verify the second estimate of \eqref{eq:lam_1_lower_est}, we only need to show that the smallest eigenvalue $\lam_1$ of $L_{\veps}[u_\del]$ in $H^1_0(\Omega)$ is bounded from below by a constant $C$ independent of $\veps$ and $\del$. Thus, for the proof we may set $\cV^\del_0=H^1_0(\Om)$ and show that for every $\veps,\del>0$ sufficiently small,
\[
	\lam_1=B_{\veps,\del}[\vphi_1,\vphi_1]\ge C .
\]
Substituting $v=\vphi_1$ into \eqref{eq:proof_lam_1_lower_est_1} yields
\begin{equation}
	\lam_1 = B_{\veps,\del}[\vphi_1,\vphi_1] \ge C_1  \|\nabla \vphi_1\|_{L^1(\Om)}^2 .
\end{equation}
Thus for $\veps,\del>0$ sufficiently small, by Poincaré's inequality \eqref{eq:Poin_ineq0} we get
\begin{equation}
	\lam_1\ge C_1 \|\nabla\vphi_1\|_{L^1(\Om)}^2 \ge C_2  \|\vphi_1\|_{L^1(\Om)}^2.
\end{equation}
Therefore,
\begin{equation}
	\sqrt{\lam_1} \ge \sqrt{C_2} \,
		\sum_{k=1}^K \cL(B^k_\del)|\lrang{\vphi_1}_{B^k_\del}| .
\end{equation}
As a consequence, for every $0<\del\le \del_0$, with $\del_0$ sufficiently small, we have
\begin{equation}
	\sqrt{\lam_1} \ge C_3\min_{k}\cL(B^k_{\del_0})\sum_{k=1}^K |\lrang{\vphi_1}_{B^k_\del}| ,
\end{equation}
where we have used that $B^k_{\del}\supset B^k_{\del_0}$.
Finally, Lemma \ref{prop:avg_inv} yields
\begin{equation}
	\sum_{k=1}^K |\lrang{\vphi_1}_{B^k_\del}| = |\Sig e_1|_{\ell^1}
		\ge |\Sig e_1| \ge C>0 ,
\end{equation}
where $\Sig$ is given by \eqref{eq:avg_mat} and $e_1=(1,0,\ldots,0)^T\in\R^K$, and thus $\lam_1\ge C>0$.
Since $\lam_1$ is the minimum of the Rayleigh quotient in $H^1_0(\Om)\setminus\{0\}$, the last two estimates yield the second inequality in \eqref{eq:lam_1_lower_est}, which completes the proof.
\end{proof}

Recall that we have not yet derived estimate \eqref{eq:U_est_0} of Lemma \ref{lem:U_del_est}.
To do so, we will need to estimate the norm of $\vphi_0-u^0$ (or $\vphi_0$) in $\Om$.

\begin{lemma}\label{lem:phi_0_L2_est}
There exists a positive constant $C$, such that for each positive $\del,\veps>0$ sufficiently small
\begin{equation}\label{eq:lem:phi_0_L2_est}
	\|\vphi_0-u^0\|_{L^2(\Om)} \le C ,
	\qquad
	k=1,\ldots,K .
\end{equation}%
\end{lemma}

\begin{proof}
Let $\eta\in\cV^\del_0$ be given by $\eta=\vphi_0-u^0_\del$, where $u^0_\del$ is the admissible approximation of $u^0$.
Then, \eqref{eq:phi0BVP_analysis} implies
\begin{equation}
	B_{\veps,\del}[\vphi_0,\eta]=0
\end{equation}
and thus using \eqref{eq:est0} we obtain
\begin{equation}
	B_{\veps,\del}[\eta,\eta] \le B_{\veps,\del}[\eta,\eta] +B_{\veps,\del}[\vphi_0,\vphi_0]
		= B_{\veps,\del}[u^0_\del,u^0_\del] \le C .
\end{equation}
Since $\eta=0$ on $\p\Om$, we also have
\begin{equation}
	\lam_1\|\eta\|_{L^2(\Om)}^2 \le B_{\veps,\del}[\eta,\eta] \le C .
\end{equation}
By Theorem \ref{thm:lam_1_lower_est}, $\lam_1$ is bounded from below by a positive constant independent of $\del$ and $\veps$, therefore,
\begin{equation}
	\|\vphi_0-u^0_\del\| _{L^2(\Om)}
		=\|\eta\|_{L^2(\Om)} \le C ,
\end{equation}
which yields \eqref{eq:lem:phi_0_L2_est} by the triangle inequality, \eqref{eq:u_del_decomp} and \eqref{eq:L2_conv} and thus completes the proof.
\end{proof}


We can now prove our main results.
Here, as above, we suppose $u$, given by \eqref{eq:u_decomp}, is approximated by admissible $u_\del$ as defined in Section \ref{sec:reg_approx}, and let $X_K$ be given by \eqref{eq:X_K_def}.
For $\veps,\del$ positive, $L_\veps[\cdot]$ is given by \eqref{eq:linear_op} with $\mu_\veps[\cdot]$ given by \eqref{eq:mu_hatmu}.
Finally, we let $\vphi_0$ satisfy \eqref{eq:phi0BVP_analysis} and $\vphi_1,\ldots,\vphi_K$ satisfy \eqref{eq:eigenValProb_analysis} weakly in $\cV^\del_0$.

\begin{theorem}\label{thm:main}
\begin{enumerate}
\item
Let $\Pi_K^{\veps}[u_\del]$ be the orthogonal projection on $\Phi_K^{\veps,\del}$, given by \eqref{eq:orthogonal-projector-PK}.
There exists a positive constant $C$ such that for every $v\in X_K$ and every $\veps,\del$ sufficiently small,
\begin{equation}\label{eq:main_est_0}
	\norm{v-\Pi_K^{\veps}[u_\del] v }_{L^2(\Om)} \le C\sqrt{\veps+\del}\, \|v\|_{L^2(\Om)}.
\end{equation}
In particular, $v=\wt{u}$ and $v=\chi_{A^k}$ ($k=1,\ldots,K$) satisfy \eqref{eq:main_est_0}.

\item
Let $Q_K^{\veps}[u_\del]$ be the $L^2$-orthogonal projection on $\vphi_0+\Phi_K^{\veps,\del}$, given by \eqref{eq:best_L2_affine_QK}.
There exists a positive constant $C$ such that for every $v\in u^0+X_K$ and every $\veps,\del$ sufficiently small,
\begin{equation}\label{eq:main_est_1}
	\bignorm{v-Q_K^{\veps}[u_\del](v)}_{L^2(\Om)}
		\le C\sqrt{\veps+\del}\, \bigrb{\|v-u^0\|_{L^2(\Om)} + 1} .
\end{equation}
In particular, $v=u$ and $v=u^0$ satisfy \eqref{eq:main_est_1}.
\end{enumerate}
\end{theorem}%

\begin{remark}\label{rem:thm:main}
Theorem \ref{thm:main} estimates the projection error for piecewise constant functions in $X_K$ and $u^0+X_K$.
From these we immediately obtain error estimates for $L^2(\Om)$ functions, e.g, admissible approximations of elements of $X_K$ or $u^0+X_K$.
By the triangle inequality and Theorem \ref{thm:main}, for every $v\in X_K$ and $w\in L^2(\Om)$, we have
\begin{equation}
		\label{eq:main_est_2}
	\bignorm{w-\Pi_K^{\veps}[u_\del] w }_{L^2(\Om)}
		\le C\sqrt{\veps+\del} \, \|w\|_{L^2(\Om)} + \|w-v\|_{L^2(\Om)}\, ,
\end{equation}
and, similarly, if $v\in u^0 +X_K$ and $w \in L^2(\Om)$, then
\begin{equation}
	\label{eq:main_est_3}
	\bignorm{w-Q_K^{\veps}[u_\del](w)}_{L^2(\Om)}
		\le C\sqrt{\veps+\del} \, \bigrb{\|w-u^0\|_{L^2(\Om)} + 1}
			+ \|w-v\|_{L^2(\Om)}\, .
\end{equation}
In particular, \eqref{eq:main_est_2} is satisfied for $v=\wt{u}$ and $w=\wt{u}_\del$, and \eqref{eq:main_est_3} is satisfied for $v=u$ and $w=u_\del$ and for $v=u^0$ and $w=u^0_\del$.
\end{remark}

Similarly to Corollary 6 of \cite{Baffet_2021}, we have the following:

\begin{corollary}\label{cor:main}
\begin{enumerate}
\item
If $u_h$ is the interpolation of $u$ in a FE space $V_h$ as in Proposition \ref{prop:FE}, and either $\cV^h=V_h$ or $\cV^h=H^1(\Om)$, then for every $\veps,h>0$ sufficiently small, estimates \eqref{eq:main_est_0} and \eqref{eq:main_est_1} with $\del=h$ hold true.

\item
If $u_\del$ is the mollification of $u$ as in Proposition \ref{prop:conv} and $\cV^\del=H^1(\Om)$, then for every $\veps,\del>0$ sufficiently small, estimates \eqref{eq:main_est_0} and \eqref{eq:main_est_1} hold true.
\end{enumerate}
\end{corollary}

\begin{proof}
This corollary is a direct result of Theorem \ref{thm:main} and Propositions \ref{prop:FE} and \ref{prop:conv}.
%
\end{proof}

Assertion 1 of Corollary \ref{cor:main} is particularly important for applications, as
it implies that our main estimates are valid not only for the continuous setting, but also for Galerkin FE discretizations as follows:
Let $u_h$ be the interpolant of $u$ in a FE space $V_h$ with mesh size $h$, and let $\vphi_0,\vphi_1,\ldots,\vphi_K$ be the Galerkin FE solutions of \eqref{eq:phi0BVP_analysis}, \eqref{eq:eigenValProb_analysis} with $u_\del=u_h$ in $V_h$.
By Assertion 1 of Corollary \ref{cor:main}, the projections $\Pi_K^{\veps}[u_h], Q_K^{\veps}[u_h]$ defined by \eqref{eq:orthogonal-projector-PK} and \eqref{eq:best_L2_affine_QK} -- using the \emph{computed} FE solutions $\vphi_0,\vphi_1,\ldots,\vphi_K$ -- satisfy \eqref{eq:main_est_0} and \eqref{eq:main_est_1} with $\del=h$.

\begin{proof}[Proof of Theorem \ref{thm:main}]
Here, we only show \eqref{eq:main_est_1}; the proof of \eqref{eq:main_est_0} is similar.
We have
\begin{equation}
	\bignorm{v-Q_K^{\veps}[u_\del](v) }_{L^2(\Om)}
		= \min_{\beta\in\R^K}
			\biggnorm{(v -\vphi_0) -\sum_{k=1}^K \beta_k \vphi_k }_{L^2(\Om)} .
\end{equation}
By Lemma \ref{prop:avg_inv} there exists a unique vector $\beta=(\beta_k)\in\R^K$ such that
\begin{equation}\label{eq:beta_def}
	\sum_{j=1}^K \beta_j \lrang{\vphi_j}_{B^k_\del} = \lrang{v-\vphi_0}_{B^k_\del} ,
	\qquad
	k=1,\ldots,K,
\end{equation}
and, moreover,
\begin{equation}
	|\beta|^2 \le C_1 \sum_{k=1}^K \lrang{v-\vphi_0}_{B^k_\del}^2
\end{equation}
for $C_1>0$ independent of $\veps$, $\del$ and $v$.
Thus, we get
\begin{equation}
	|\beta| \le \sqrt{C_1}
			\sqb{\sum_{k=1}^K \|v-\vphi_0\|_{L^2(B^k_\del)}^2}^{\frac{1}{2}}.
\end{equation}
By Lemma \ref{lem:phi_0_L2_est}, we have
\begin{equation}
	\|v-\vphi_0\|_{L^2(B^k_\del)}
		\le \|v-u^0\|_{L^2(B^k_\del)} + \|u^0-\vphi_0\|_{L^2(B^k_\del)}
		\le \|v-u^0\|_{L^2(B^k_\del)} + C
\end{equation}
which yields
\begin{equation}\label{eq:beta_phi_0_coef_est}
	|\beta| \le C \bigrb{\|v-u^0\|_{L^2(\Om)}+1} ,
\end{equation}
with $C>0$ independent of $\veps$, $\del$ and $v$.
Define
\begin{equation}
	\vphi =\vphi_0 +\wt\vphi ,
	\qquad
	\wt\vphi = \sum_{k=1}^K \beta_k \vphi_k \in \Phi_K^{\veps,\del}.
\end{equation}
Thus, we have
\begin{equation}
	\bignorm{v-Q_K^{\veps}[u_\del](v) }_{L^2(\Om)}
		\le \|v-\vphi\|_{L^2(\Om)} .
\end{equation}
By the triangle inequality, we get
\begin{equation}\label{eq:proof_err_est_0}
\begin{aligned}
	\|v-\vphi\|_{L^2(\Om)} &\le \|v-\vphi\|_{L^2(E_\del)}
			+\|v-u^0\|_{L^2(\interfaces_\del)} +\|u^0- \vphi\|_{L^2(\interfaces_\del)} \\
		&\quad +\sum_{k=1}^K \|v-\vphi\|_{L^2(B^k_\del)} .
\end{aligned}
\end{equation}
Next we estimate each of the terms on the right hand side.
Since $v=u^0$ in $E_\del$, we can estimate the first term as follows:
\begin{equation}
\begin{aligned}
	\|v-\vphi\|_{L^2(E_\del)} & \le \|u^0-\vphi_0\|_{L^2(E_\del)}
			+\|\wt\vphi\|_{L^2(E_\del)} \\
		& \le \|u^0-\vphi_0\|_{L^2(E_\del)}
			+\sum_{k=1}^K |\beta_k| \|\vphi_k\|_{L^2(E_\del)} \\
		& \le C\sqrt{\veps} \, \bigrb{\|v-u^0\|_{L^2(\Om)}+1}
\end{aligned}
\end{equation}
because of \eqref{eq:est_phi_0}, \eqref{eq:est_phi_k} and \eqref{eq:beta_phi_0_coef_est}.
The second term on the right hand side of \eqref{eq:proof_err_est_0} is the $L^2$ norm of the piecewise constant function $w=v-u^0$ in $\interfaces_\del$.
Since $w=0$ a.e.\ in $\Omega\setminus\ol{\bigcup_{k=1}^K B^k}$, we have
\begin{equation}
	\|v-u^0\|_{L^2(\interfaces_\del)}^2 = \int_{\interfaces_\del} w^2
		=\sum_{k=1}^K \int_{\interfaces_\del\cap B^k} w^2 .
\end{equation}
We now use that $w^2$ is constant in each $B^k$ and that $\cL(\interfaces_\del\cap B^k) = \cO(\del)$ \cite[Lemma 4]{Baffet_2021} to get
\begin{equation}
	\int_{\interfaces_\del\cap B^k} w^2
		= \cL(\interfaces_\del\cap B^k)\, w^2|_{B^k}
		\le C \del\, \cL(B^k)\, w^2|_{B^k}
		= C\del \int_{B^k} w^2
\end{equation}
which yields
\begin{equation}
	\|v-u^0\|_{L^2(\interfaces_\del)} \le C\sqrt{\del} \, \|v-u^0\|_{L^2(\Om)} .
\end{equation}
To estimate the third term we use Lemma \ref{lem:U_del_est} and \eqref{eq:beta_phi_0_coef_est} to obtain
\begin{equation}
\begin{aligned}
	\|u^0- \vphi\|_{L^2(\interfaces_\del)}
		& \le \|u^0-\vphi_0\|_{L^2(\interfaces_\del)}
			+\sum_{k=1}^K |\beta_k|\|\vphi_k\|_{L^2(\interfaces_\del)} \\
		& \le C\sqrt{\del} \, \bigrb{\|v-u^0\|_{L^2(\Om)}+1} .
\end{aligned}
\end{equation}
For each $k=1,\ldots,K$, we estimate $\|v-\vphi\|_{L^2(B^k_\del)}$ as follows:
Since $\beta$ solves \eqref{eq:beta_def}, we have $\lrang{v-\vphi}_{B^k_\del}=0$, which by the Poincaré inequality \eqref{eq:Poin_ineq_avg} yields
\begin{equation}
	\|v-\vphi\|_{L^2(B^k_\del)} \le C\|\nabla(v-\vphi)\|_{L^2(B^k_\del)} .
\end{equation}
Since $\nabla v=0$ in $B^k_\del$, estimates \eqref{eq:est1} and \eqref{eq:beta_phi_0_coef_est} yield
\begin{equation}
\begin{aligned}
	\|v-\vphi\|_{L^2(B^k_\del)} &\le C_1 \|\nabla\vphi\|_{L^2(B^k_\del)}
		\le C_1 \rb{\|\nabla\vphi_0\|_{L^2(B^k_\del)}
			+ \sum_{j=1}^K |\beta_j| \|\nabla\vphi_j\|_{L^2(B^k_\del)}} \\
		& \le C_2\sqrt{\veps} \bigrb{\|v-u^0\|_{L^2(\Om)}+1} .
\end{aligned}
\end{equation}
Finally, by combining the above, we obtain
\begin{equation}
	\bignorm{v-Q_K^{\veps}[u_\del](v) }_{L^2(\Om)}
		\le \|v-\vphi\|_{L^2(\Om)} \le C\sqrt{\veps+\del} \, \bigrb{\|v-u^0\|_{L^2(\Om)}+1}
\end{equation}
which completes the proof.
\end{proof}

\section{Numerical examples}\label{sec:num_res}


Here we present numerical examples which illustrate the main results of our analysis and, in particular, the remarkable accuracy of AS decompositions for piecewise constant media\footnote{We will use the term \emph{medium} for functions from $\Omega\subset\mathbb{R}^{2}$ into $\mathbb{R}$.}.
First, we consider media comprised of a constant background $u^0$ and a single characteristic function.
Secondly we consider a medium which consists of an inhomogeneous background comprised of five sets $\Omega^{m}$, $m=1,\ldots,5$, and four interior inclusions $A^{k}$, $k=1,\ldots,4$ (see Section \ref{sec:medium}).
In the third example, we consider a medium which consists of four adjacent squares in a constant background.
Since the boundaries of the squares are not mutually disjoint, this example is not covered by our theory.
Next we apply the AS decomposition to two more complex examples that are not covered by our theory: a polygonal approximation of the map of Switzerland with its 26 cantons and the well-known Marmousi model from seismic imaging.
Finally, we devise a simple iterative inversion algorithm based on AS decompositions to solve a standard deconvolution inverse problem from optical imaging \cite{bertero2021inverse}.

In all examples the domain $\Omega \subset \mathbb{R}^{2}$ is rectangular and we use a regular, uniform triangular mesh $\mathcal{T}_{h}$ whose vertices lie on an equidistant Cartesian grid of size $h>0$.
We let $\mathcal{V}^{\delta} \subset H^{1}(\Omega)$, with $\delta = h$, be the standard $\mathcal{P}^{1}$ FE space of continuous piecewise linear functions and set $\mathcal{V}^{\delta}_{0} = \mathcal{V}^{\delta} \cap H^{1}_{0}$.
For piecewise constant $u$, we let $u_\delta$ denote the $H^{1}$-conforming (continuous) interpolation of $u$ in the FE space $\cV^\delta$.

We consider decompositions associated with $L_{\varepsilon}[u_{\delta}]$ given by \eqref{eq:linear_op} with $\mu_\veps[\cdot]$ of the form \eqref{eq:muq} with $q = 2$.
We compute the approximation of the background $\varphi_{0}$ and the first few eigenfunctions $\varphi_{k}$ of $L_{\varepsilon}[u_{\delta}]$ by numerically solving \eqref{eq:phi0BVP_analysis} and \eqref{eq:eigenValProb_analysis} using the Galerkin FE method.
The discretization of \eqref{eq:eigenValProb_analysis} leads to a generalized eigenvalue problem
\begin{align}\label{eq:num_gen_eig_val_prob}
	\mathrm{A}\varphi_{k} = \lambda_{k}\mathrm{M}\varphi_{k} \quad\text{for}\quad k = 1, \ldots, K,
\end{align}
where the stiffness matrix $\mathrm{A}$ corresponds to the discretization of $L_{\varepsilon}[u_{\delta}]$ and $\textrm{M}$ is the mass matrix.
We solve \eqref{eq:num_gen_eig_val_prob} numerically using the MATLAB function \texttt{eigs}.

Once we have obtained $\varphi_{0} \in \mathcal{V}^{\delta}$ and $\varphi_{k} \in \mathcal{V}^{\delta}_{0}$ for $k=1,\ldots,K$, we can compute the projections $\Pi_K^{\veps}[u_\del]$ and $Q_K^{\veps}[u_\del]$ given by \eqref{eq:orthogonal-projector-PK} and \eqref{eq:best_L2_affine_QK}.
Since $\{\varphi_{k}\}_{k=1}^K$ are computed numerically, they satisfy $\lrang{\vphi_k,\vphi_j}=\del_{kj}$ only up to a small error.
This slight loss of orthonormality causes small errors when computing the projection $\Pi_K^{\veps}[u_\del]$ directly from the Fourier expansion 
\begin{align*}
	\Pi_{K}^{\varepsilon}[u_{\delta}]v = \sum_{k=1}^{K} \langle \varphi_{k},v \rangle\, \varphi_{k} .
\end{align*}
To avoid these errors, we instead compute $\Pi_K^{\veps}[u_\del] v$ by solving the $K$-dimensional least squares problem 
\begin{align*}
	\Pi_{K}^{\varepsilon}[u_{\delta}]v = \operatorname*{argmin}_{w \in \Phi_{K}^{\varepsilon,\delta}}\|v-w\|_{L^2(\Om)},
	\qquad \Phi_{K}^{\varepsilon,\delta} = \operatorname{span}\{\varphi_{k}\}_{k=1}^{K} .
\end{align*}

When validating the conclusion of Theorem \ref{thm:main} and its corollary in Remark \ref{rem:thm:main}, we shall focus on two types of errors
\begin{equation}\label{eq:NumEx.ApproxError}
	\| u - Q_{K}^{\varepsilon}[u_{\delta}](u) \|_{L^2(\Om)}
		\qquad\text{and} \qquad
	\| u_{\delta} - Q_{K}^{\varepsilon}[u_{\delta}](u_{\delta}) \|_{L^2(\Om)} ;
\end{equation}
the first measures the misfit to the true medium $u$ whereas the second measures the misfit to the continuous interpolant $u_{\delta}$. Note that in both cases the same AS basis is used.
Computing these expressions requires the evaluation of $L^2$ inner products.
As the functions participating in the expression on the right lie in the FE space $\cV^\del$, we can evaluate the needed integrals exactly.
In contrast, the expression on the left includes inner products involving a piecewise constant function whose discontinuities are, in general, not aligned with the mesh.
Thus, to evaluate the integrals for the error on the left in \eqref{eq:NumEx.ApproxError}, we use a numerical quadrature rule from ACM TOMS algorithm $\#584$ \cite{laurie1982algorithm} with degree of precision of $8$ and $19$ quadrature points.

In principle, $\varepsilon > 0$ should be as small as possible, while sufficiently large so that the matrix $\mathrm{A}$ is well-conditioned.
Unless specified otherwise, we always use ${\varepsilon = 10^{-8}}$.

\subsection{Four simple shapes}

\begin{figure}[t]
\centering
\begin{subfigure}[c]{0.24\textwidth}
\centering
\includegraphics[width=0.9\linewidth]{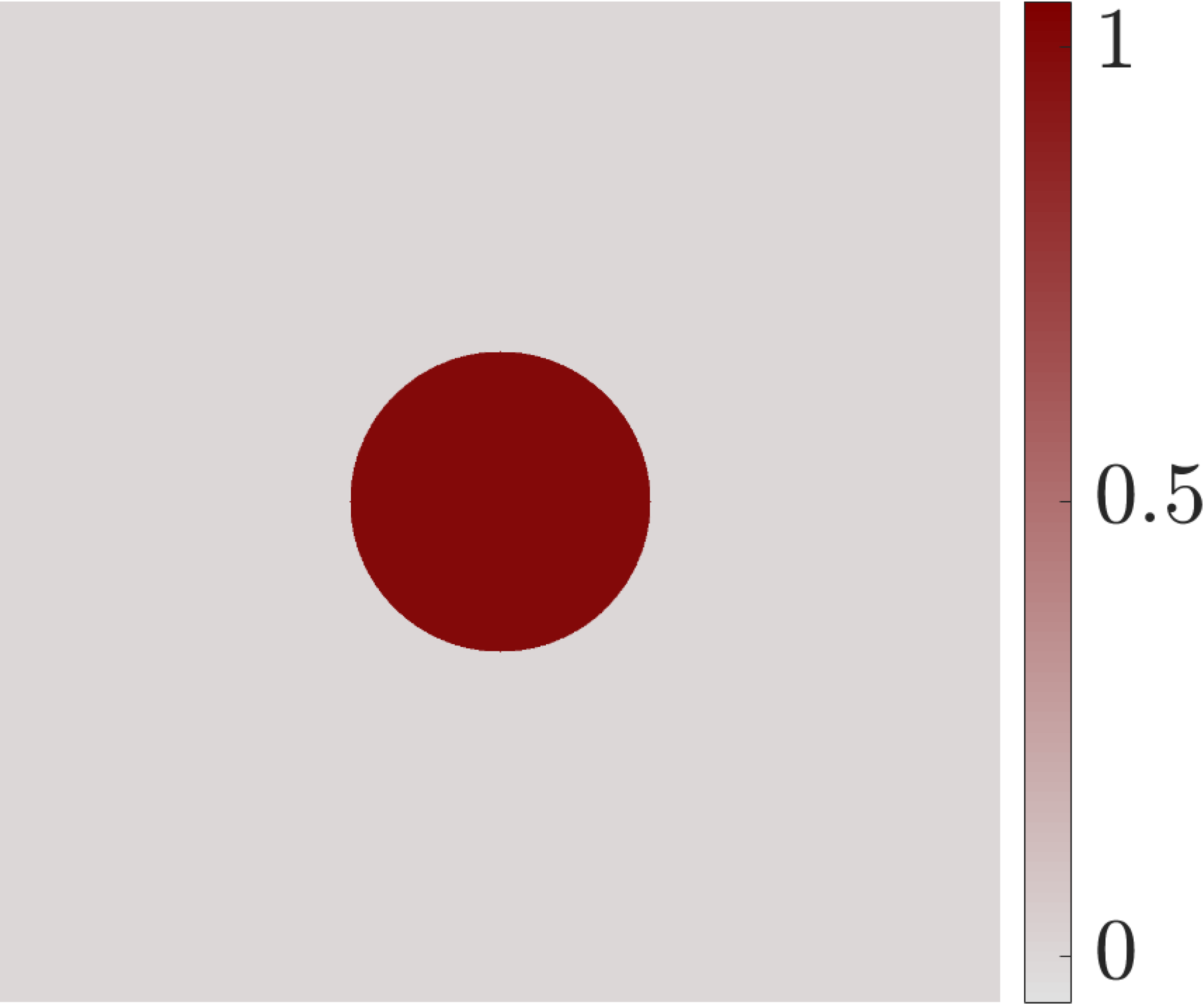}
\caption{disc}
\end{subfigure}\hfill
\begin{subfigure}[c]{0.24\textwidth}
\centering
\includegraphics[width=0.9\linewidth]{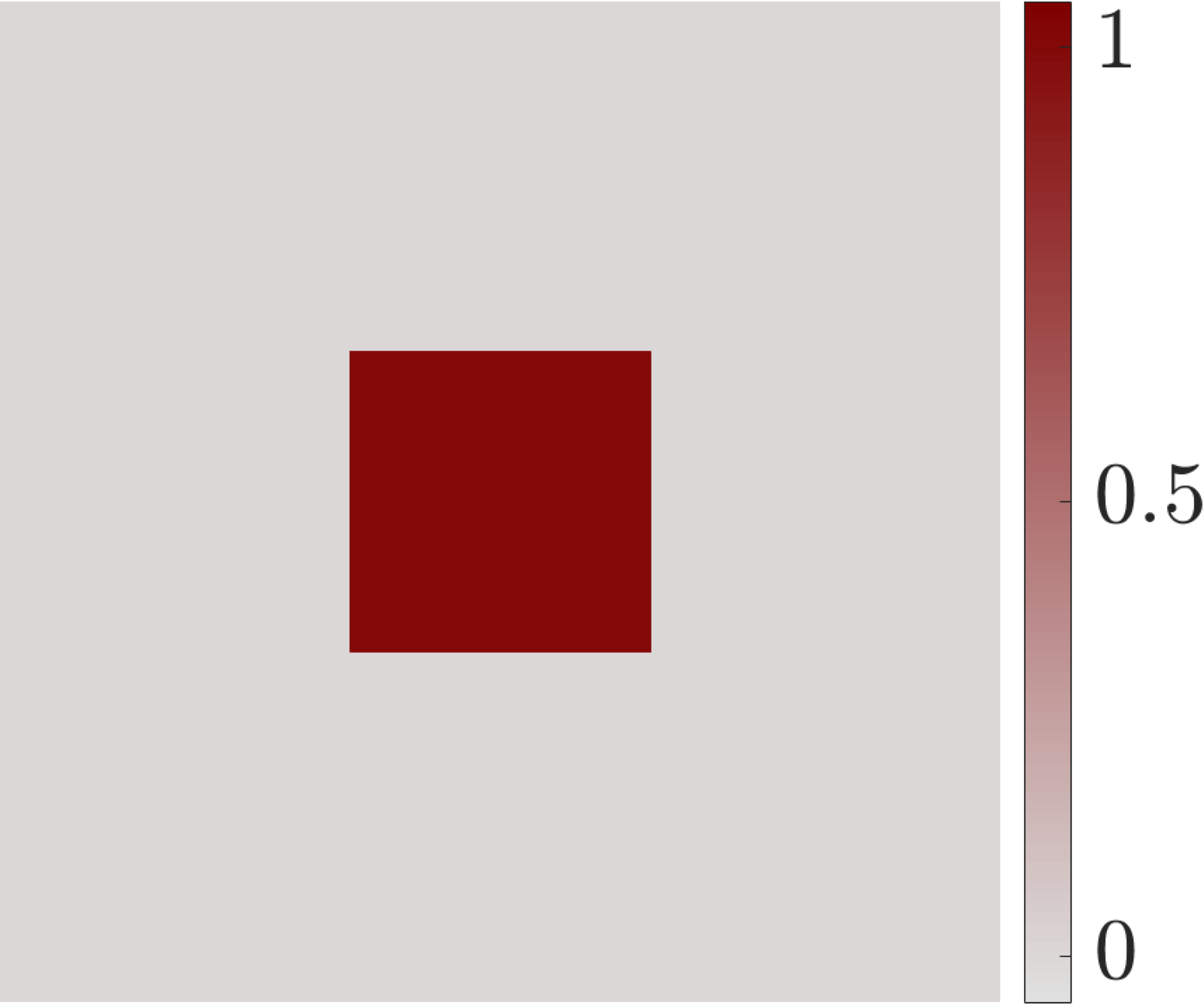}
\caption{square}
\end{subfigure}\hfill
\begin{subfigure}[c]{0.24\textwidth}
\centering
\includegraphics[width=0.9\linewidth]{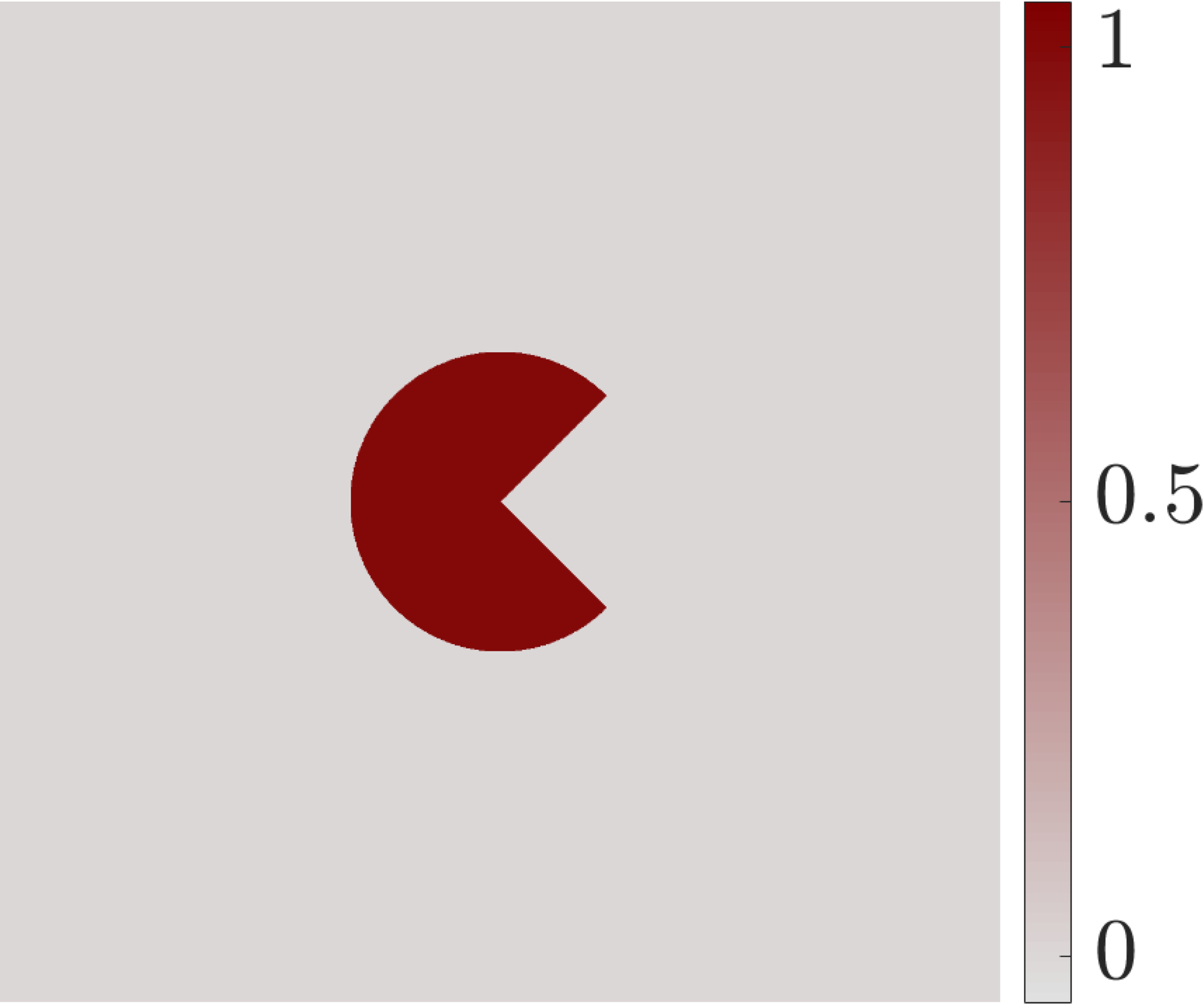}
\caption{Pac-Man}
\end{subfigure}\hfill
\begin{subfigure}[c]{0.24\textwidth}
\centering
\includegraphics[width=0.9\linewidth]{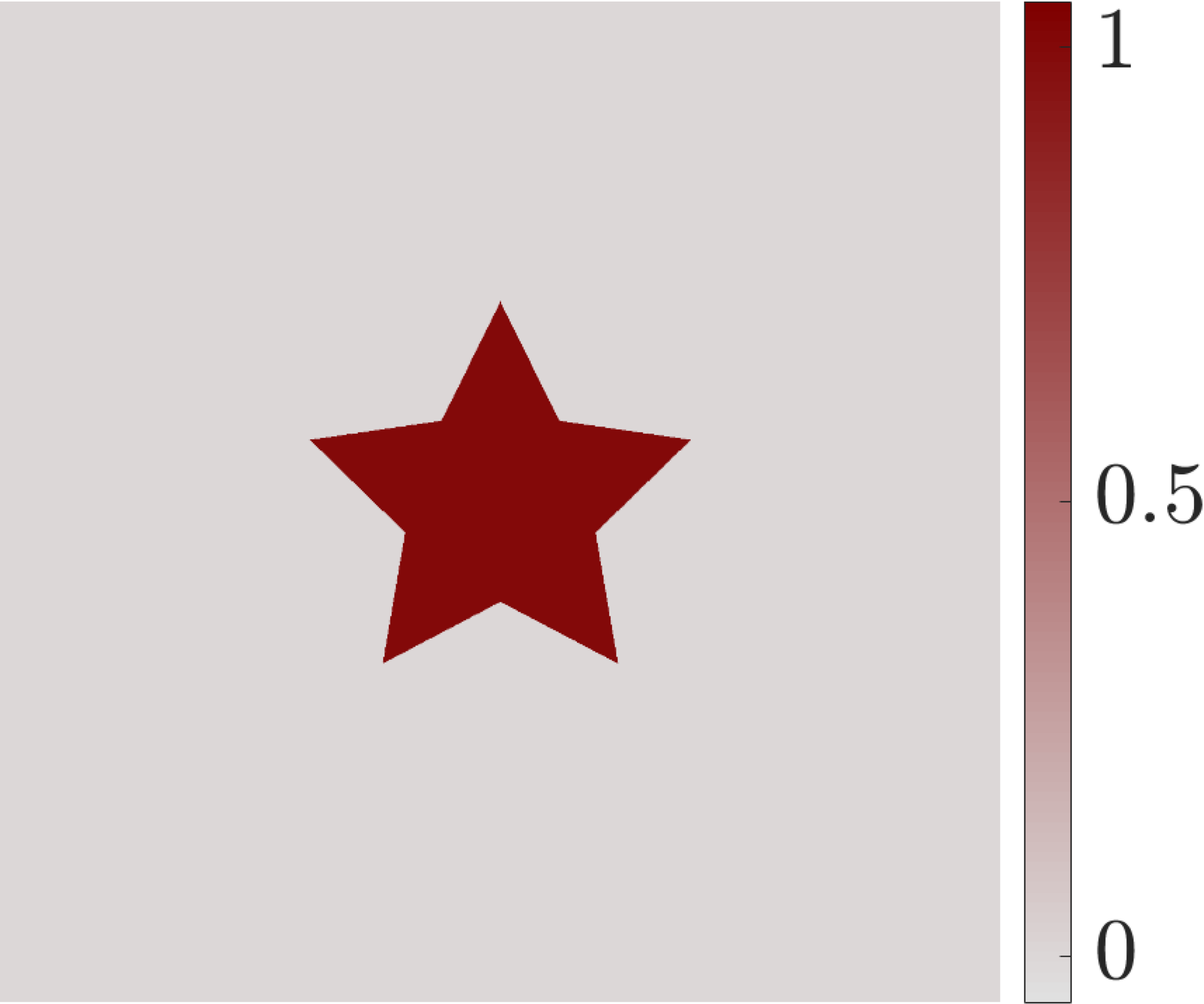}
\caption{star}
\end{subfigure}
\caption{Four simple shapes. The exact medium $u$ (or $u_{\delta}$) consists of a single characteristic function $\chi_{A^{1}}$ and vanishing $u^{0}$.}
\label{fig:NumEx.SimpleShapes}
\end{figure}

We consider the four 2-dimensional piecewise constant media ${u:\Omega\to\mathbb{R}}$, in $\Omega = (0,1)^{2}$, shown in Fig.\ \ref{fig:NumEx.SimpleShapes}.
All four vanish on the boundary $\partial\Omega$ and correspond to the characteristic function
\begin{align}
	u(x) = \widetilde{u}(x) = \chi_{A^{1}}(x), \quad x \in \Omega
\end{align}
of a Lipschitz domain and are therefore covered by our analysis.
The sets are chosen purposely with different geometric properties:
the disc is convex with a smooth boundary; the square is convex, but its boundary is only piecewise smooth; the Pac-Man and the star are both non-convex with piecewise smooth boundaries.

\begin{figure}[t]
\centering
\begin{subfigure}{0.49\textwidth}
\centering
%
%
\definecolor{mycolor1}{rgb}{0.00000,0.44700,0.74100}%
\definecolor{mycolor2}{rgb}{0.85000,0.32500,0.09800}%
\definecolor{mycolor3}{rgb}{0.92900,0.69400,0.12500}%
\definecolor{mycolor4}{rgb}{0.49400,0.18400,0.55600}%
\begin{tikzpicture}[scale=0.8]

\begin{axis}[%
width=1\textwidth,
height=0.8\textwidth,
at={(1.054in,0.642in)},
scale only axis,
xmode=log,
xmin=0.000620568933378345,
xmax=0.0314731352948542,
xminorticks=true,
xlabel style={font=\color{white!15!black}},
xlabel={$\delta$},
ymode=log,
ymin=0.00943987526480663,
ymax=0.10792590738485,
yminorticks=true,
ylabel style={font=\color{white!15!black}},
ylabel={Error},
axis background/.style={fill=white},
xmajorgrids,
ymajorgrids,
legend style={at={(0.03,0.97)}, anchor=north west, legend cell align=left, align=left, draw=white!15!black}
]
\addplot [color=mycolor1, mark size=5.0pt, mark=o, mark options={solid, mycolor1}]
  table[row sep=crcr]{%
0.025	0.0605230826642486\\
0.0125	0.0435884369710871\\
0.00625	0.0305945029200308\\
0.003125	0.0215258543953637\\
0.0015625	0.0149163299791692\\
0.00078125	0.0105917142534217\\
};
\addlegendentry{disc}

\addplot [color=mycolor2, mark size=3pt, mark=square, mark options={solid, mycolor2}]
  table[row sep=crcr]{%
0.025	0.0961890662056829\\
0.0125	0.069331826894853\\
0.00625	0.0495067500871785\\
0.003125	0.035179914904505\\
0.0015625	0.0249377938795556\\
0.00078125	0.0176556436289523\\
};
\addlegendentry{square}

\addplot [color=mycolor3, mark size=3pt, mark=triangle, mark options={solid, mycolor3}]
  table[row sep=crcr]{%
0.025	0.0678603006176516\\
0.0125	0.0477742286279695\\
0.00625	0.0332789151784551\\
0.003125	0.0234451834958269\\
0.0015625	0.0163458864657808\\
0.00078125	0.0115886009971929\\
};
\addlegendentry{Pac-Man}

\addplot [color=mycolor4, mark size=3pt, mark=triangle, mark options={solid, rotate=270, mycolor4}]
  table[row sep=crcr]{%
0.025	0.0743249327917898\\
0.0125	0.0503018826907414\\
0.00625	0.0346352918756755\\
0.003125	0.0247304042501057\\
0.0015625	0.0176509001470463\\
0.00078125	0.0123601274880543\\
};
\addlegendentry{star}

\addplot [color=black, dashed]
  table[row sep=crcr]{%
0.025	0.05\\
0.0125	0.0353553390593274\\
0.00625	0.025\\
0.003125	0.0176776695296637\\
0.0015625	0.0125\\
0.00078125	0.00883883476483184\\
};
\addlegendentry{$\sqrt{\delta}$}

\end{axis}
\end{tikzpicture}%
\end{subfigure}
\hfill
\begin{subfigure}{0.49\textwidth}
\centering
%
%
\definecolor{mycolor1}{rgb}{0.00000,0.44700,0.74100}%
\definecolor{mycolor2}{rgb}{0.85000,0.32500,0.09800}%
\definecolor{mycolor3}{rgb}{0.92900,0.69400,0.12500}%
\definecolor{mycolor4}{rgb}{0.49400,0.18400,0.55600}%
\begin{tikzpicture}[scale=0.8]

\begin{axis}[%
width=1\textwidth,
height=0.8\textwidth,
at={(1.054in,0.642in)},
scale only axis,
xmode=log,
xmin=3.16227766016838e-09,
xmax=3.16227766016838,
xminorticks=true,
xtick={1e-08,1e-06,0.0001,0.01,1},
xticklabels={{$10^{-8}$},{$10^{-6}$},{$10^{-4}$},{$10^{-2}$},{$10^0$}},
xlabel style={font=\color{white!15!black}},
xlabel={$\varepsilon$},
ymode=log,
ymin=0.000105917142534217,
ymax=0.180082212610157,
yminorticks=true,
ylabel style={font=\color{white!15!black}},
ylabel={Error},
axis background/.style={fill=white},
xmajorgrids,
ymajorgrids,
legend style={at={(0.03,0.03)}, anchor=south west, legend cell align=left, align=left, draw=white!15!black}
]
\addplot [color=mycolor1, mark size=5.0pt, mark=o, mark options={solid, mycolor1}]
  table[row sep=crcr]{%
1	0.0971752281469363\\
0.1	0.0135520043889688\\
0.01	0.0106246758602478\\
0.001	0.0105921503832557\\
0.0001	0.0105917293924796\\
1e-05	0.0105917154816832\\
1e-06	0.0105917143722937\\
1e-07	0.0105917142642025\\
1e-08	0.0105917142534217\\
};
\addlegendentry{disc}

\addplot [color=mycolor2, mark size=3pt, mark=square, mark options={solid, mycolor2}]
  table[row sep=crcr]{%
1	0.0924643394546158\\
0.1	0.0196838515662359\\
0.01	0.0176856385498676\\
0.001	0.0176568695929709\\
0.0001	0.0176557485299788\\
1e-05	0.0176556539329666\\
1e-06	0.0176556446482563\\
1e-07	0.0176556437216046\\
1e-08	0.0176556436289523\\
};
\addlegendentry{square}

\addplot [color=mycolor3, mark size=3pt, mark=triangle, mark options={solid, mycolor3}]
  table[row sep=crcr]{%
1	0.124732111302158\\
0.1	0.0157054254290873\\
0.01	0.011640810677246\\
0.001	0.0115899409553064\\
0.0001	0.0115886964802611\\
1e-05	0.011588610152612\\
1e-06	0.011588601900646\\
1e-07	0.0115886010792897\\
1e-08	0.0115886009971929\\
};
\addlegendentry{Pac-Man}

\addplot [color=mycolor4, mark size=3pt, mark=triangle, mark options={solid, rotate=270, mycolor4}]
  table[row sep=crcr]{%
1	0.143044386047869\\
0.1	0.0178665255082266\\
0.01	0.0124255557227186\\
0.001	0.0123613081363377\\
0.0001	0.0123601924227082\\
1e-05	0.012360133439398\\
1e-06	0.012360128065749\\
1e-07	0.0123601275325634\\
1e-08	0.0123601274880543\\
};
\addlegendentry{star}

\addplot [color=black, dashed]
  table[row sep=crcr]{%
1	0.05\\
0.316227766016838	0.0158113883008419\\
0.1	0.005\\
0.0316227766016838	0.00158113883008419\\
0.01	0.0005\\
0.00316227766016838	0.000158113883008419\\
0.001	5e-05\\
0.000316227766016838	1.58113883008419e-05\\
0.0001	5e-06\\
3.16227766016838e-05	1.58113883008419e-06\\
1e-05	5e-07\\
3.16227766016838e-06	1.58113883008419e-07\\
1e-06	5e-08\\
3.16227766016838e-07	1.58113883008419e-08\\
1e-07	5e-09\\
3.16227766016838e-08	1.58113883008419e-09\\
1e-08	5e-10\\
};
\addlegendentry{$\varepsilon$}

\end{axis}
\end{tikzpicture}%
\end{subfigure}
\caption{Four simple shapes. The error $\|u - \Pi_{1}^{\varepsilon}[u_\delta](u)\|_{L^2(\Om)}$.
Left: the error as a function of $\delta$ for fixed $\varepsilon = 10^{-8}$.
Right: the error as a function of $\varepsilon$ for fixed mesh-size $\delta = 0.05/2^{6}$.}
\label{fig:NumEx.SimpleShapes.ErrorDelta}
\end{figure}
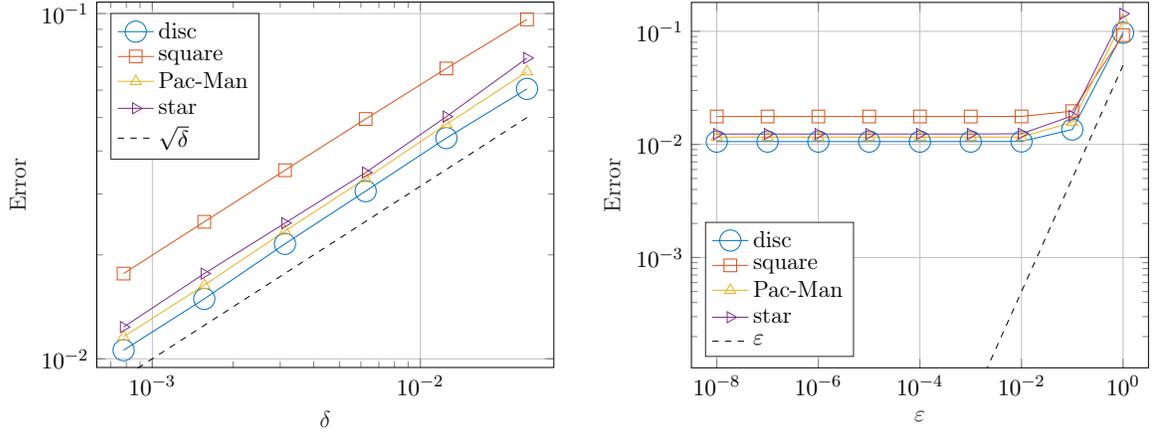

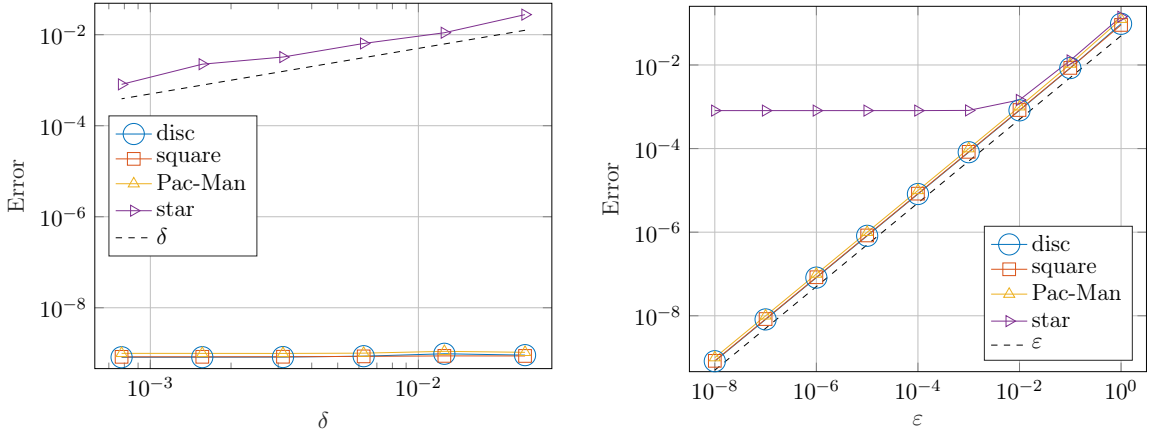
\begin{figure}[t]
\centering
\begin{subfigure}{0.49\textwidth}
\centering
%
%
\definecolor{mycolor1}{rgb}{0.00000,0.44700,0.74100}%
\definecolor{mycolor2}{rgb}{0.85000,0.32500,0.09800}%
\definecolor{mycolor3}{rgb}{0.92900,0.69400,0.12500}%
\definecolor{mycolor4}{rgb}{0.49400,0.18400,0.55600}%
\begin{tikzpicture}[scale=0.8]

\begin{axis}[%
width=1\textwidth,
height=0.8\textwidth,
at={(1.054in,0.642in)},
scale only axis,
xmode=log,
xmin=0.000620568933378345,
xmax=0.0314731352948542,
xminorticks=true,
xlabel style={font=\color{white!15!black}},
xlabel={$\delta$},
ymode=log,
ymin=4.6008533266987e-10,
ymax=0.0490728823757276,
ytick={ 1e-08,  1e-06, 0.0001,   0.01},
yminorticks=true,
ylabel style={font=\color{white!15!black}},
ylabel={Error},
axis background/.style={fill=white},
xmajorgrids,
ymajorgrids,
legend style={at={(0.03,0.5)}, anchor=west, legend cell align=left, align=left, draw=white!15!black}
]
\addplot [color=mycolor1, mark size=5.0pt, mark=o, mark options={solid, mycolor1}]
  table[row sep=crcr]{%
0.025	9.09609857423813e-10\\
0.0125	9.75069621726222e-10\\
0.00625	8.6216356166285e-10\\
0.003125	8.25717635762351e-10\\
0.0015625	8.20637695860758e-10\\
0.00078125	8.18160273947738e-10\\
};
\addlegendentry{disc}

\addplot [color=mycolor2, mark size=3pt, mark=square, mark options={solid, mycolor2}]
  table[row sep=crcr]{%
0.025	8.73139481869383e-10\\
0.0125	8.75527452324404e-10\\
0.00625	8.54612575532191e-10\\
0.003125	8.39056426825817e-10\\
0.0015625	8.36365150278098e-10\\
0.00078125	8.35001211497116e-10\\
};
\addlegendentry{square}

\addplot [color=mycolor3, mark size=3pt, mark=triangle, mark options={solid, mycolor3}]
  table[row sep=crcr]{%
0.025	1.04860553067218e-09\\
0.0125	1.10257567166789e-09\\
0.00625	1.00587449450966e-09\\
0.003125	9.91722754009408e-10\\
0.0015625	9.94185353794202e-10\\
0.00078125	9.95407844059642e-10\\
};
\addlegendentry{Pac-Man}

\addplot [color=mycolor4, mark size=3pt, mark=triangle, mark options={solid, rotate=270, mycolor4}]
  table[row sep=crcr]{%
0.025	0.0275957097060768\\
0.0125	0.0110391735548951\\
0.00625	0.0064216039382446\\
0.003125	0.00322758983854714\\
0.0015625	0.00224875326277236\\
0.00078125	0.000810082710426935\\
};
\addlegendentry{star}

\addplot [color=black, dashed]
  table[row sep=crcr]{%
0.025	0.0125\\
0.0125	0.00625\\
0.00625	0.003125\\
0.003125	0.0015625\\
0.0015625	0.00078125\\
0.00078125	0.000390625\\
};
\addlegendentry{${\delta}$}

\end{axis}
\end{tikzpicture}%
\end{subfigure}
\hfill
\begin{subfigure}{0.49\textwidth}
\centering
%
%
\definecolor{mycolor1}{rgb}{0.00000,0.44700,0.74100}%
\definecolor{mycolor2}{rgb}{0.85000,0.32500,0.09800}%
\definecolor{mycolor3}{rgb}{0.92900,0.69400,0.12500}%
\definecolor{mycolor4}{rgb}{0.49400,0.18400,0.55600}%
\begin{tikzpicture}[scale=0.8]

\begin{axis}[%
width=1\textwidth,
height=0.8\textwidth,
at={(1.054in,0.642in)},
scale only axis,
xmode=log,
xmin=3.16227766016838e-09,
xmax=3.16227766016838,
xminorticks=true,
xtick={1e-08,1e-06,0.0001,0.01,1},
xticklabels={{$10^{-8}$},{$10^{-6}$},{$10^{-4}$},{$10^{-2}$},{$10^0$}},
xlabel style={font=\color{white!15!black}},
xlabel={$\varepsilon$},
ymode=log,
ymin=4.60085332672989e-10,
ymax=0.253285988540525,
ytick={ 1e-08,  1e-06, 0.0001,   0.01},
yminorticks=true,
ylabel style={font=\color{white!15!black}},
ylabel={Error},
axis background/.style={fill=white},
xmajorgrids,
ymajorgrids,
legend style={at={(0.97,0.03)}, anchor=south east, legend cell align=left, align=left, draw=white!15!black}
]
\addplot [color=mycolor1, mark size=5.0pt, mark=o, mark options={solid, mycolor1}]
  table[row sep=crcr]{%
1	0.0965598452936105\\
0.1	0.00843698705229274\\
0.01	0.000820738138194139\\
0.001	8.18418785893827e-05\\
0.0001	8.18186692276003e-06\\
1e-05	8.18163483485796e-07\\
1e-06	8.18161229112806e-08\\
1e-07	8.18161145815149e-09\\
1e-08	8.18160273953284e-10\\
};
\addlegendentry{disc}

\addplot [color=mycolor2, mark size=3pt, mark=square, mark options={solid, mycolor2}]
  table[row sep=crcr]{%
1	0.0908444024158343\\
0.1	0.00851279541804094\\
0.01	0.000836680665757519\\
0.001	8.35170440690219e-05\\
0.0001	8.35018960968788e-06\\
1e-05	8.35003810574725e-07\\
1e-06	8.35002338801173e-08\\
1e-07	8.35002277087482e-09\\
1e-08	8.35001211500054e-10\\
};
\addlegendentry{square}

\addplot [color=mycolor3, mark size=3pt, mark=triangle, mark options={solid, mycolor3}]
  table[row sep=crcr]{%
1	0.124174262827624\\
0.1	0.0105018271114606\\
0.01	0.00100082110707407\\
0.001	9.95950642183895e-05\\
0.0001	9.95464488739122e-06\\
1e-05	9.95415877973437e-07\\
1e-06	9.95410902413201e-08\\
1e-07	9.95411296108619e-09\\
1e-08	9.95407844059642e-10\\
};
\addlegendentry{Pac-Man}

\addplot [color=mycolor4, mark size=3pt, mark=triangle, mark options={solid, rotate=270, mycolor4}]
  table[row sep=crcr]{%
1	0.142433178448027\\
0.1	0.0128630846835854\\
0.01	0.00145789508268221\\
0.001	0.000818707159623978\\
0.0001	0.000810139656040533\\
1e-05	0.000810080137327749\\
1e-06	0.00081008217761613\\
1e-07	0.000810082429227165\\
1e-08	0.000810082710426935\\
};
\addlegendentry{star}

\addplot [color=black, dashed]
  table[row sep=crcr]{%
1	0.05\\
0.316227766016838	0.0158113883008419\\
0.1	0.005\\
0.0316227766016838	0.00158113883008419\\
0.01	0.0005\\
0.00316227766016838	0.000158113883008419\\
0.001	5e-05\\
0.000316227766016838	1.58113883008419e-05\\
0.0001	5e-06\\
3.16227766016838e-05	1.58113883008419e-06\\
1e-05	5e-07\\
3.16227766016838e-06	1.58113883008419e-07\\
1e-06	5e-08\\
3.16227766016838e-07	1.58113883008419e-08\\
1e-07	5e-09\\
3.16227766016838e-08	1.58113883008419e-09\\
1e-08	5e-10\\
};
\addlegendentry{$\varepsilon$}

\end{axis}
\end{tikzpicture}%
\end{subfigure}
\caption{Four simple shapes.
The error $\|u_\delta - \Pi_{1}^{\varepsilon}[u_\delta](u_\delta)\|_{L^2(\Om)}$.
Left: the error as a function of $\delta$ for fixed $\varepsilon = 10^{-8}$.
Right: the error as a function of $\varepsilon$ for fixed mesh-size $\delta = 0.05/2^{6}$.}
\label{fig:NumEx.SimpleShapes.ErrorEpsilon}
\end{figure}

In Figure \ref{fig:NumEx.SimpleShapes.ErrorDelta}, we show the error ${\|u - \Pi_{1}^{\varepsilon}[u_\delta](u)\|_{L^2(\Om)}}$.
The left frame shows the error for varying mesh-size $\delta$ but fixed $\varepsilon = 10^{-8}$.
For all four shapes, the error decays as $\mathcal{O}(\sqrt{\delta})$, as proved in Theorem \ref{thm:main}.
The right frame of Figure \ref{fig:NumEx.SimpleShapes.ErrorDelta} shows the error $\|u - \Pi_{1}^{\varepsilon}[u_\delta](u)\|_{L^2(\Om)}$ for varying $\varepsilon$ on the fixed finest mesh, i.e., with smallest $\del$.
The error initially decreases with $\varepsilon$ but then levels off at about $10^{-2}$, at which point it can only be improved by further refining the mesh.

To eliminate the interpolation error and thereby illustrate the estimates of Remark \ref{rem:thm:main}, we show in Figure \ref{fig:NumEx.SimpleShapes.ErrorEpsilon} the projection error $\|u_{\delta} - \Pi_{1}^{\varepsilon}[u_\delta](u_{\delta})\|_{L^2(\Om)}$.
On the left, we show the approximation error for varying $\delta$, with $\varepsilon = 10^{-8}$ fixed:
The projections of the disc, the square, and the Pac-Man in the AS basis are remarkably good, with errors at about $10^{-9}$.
For these cases, the projection of each $u_\delta$ (hence the first eigenfunction $\varphi_1$ of $L_\veps[u_\delta]$) essentially coincides with $u_\delta$ itself.
In contrast, the error for the star is larger, though it decays at a rate of $\mathcal{O}(\delta)$, still faster than the upper estimate of $\mathcal{O}(\sqrt{\delta})$ in Remark \ref{rem:thm:main}.
In all cases, the errors here are significantly smaller than those in the left frame of Figure \ref{fig:NumEx.SimpleShapes.ErrorDelta}, indicating that the errors in Figure \ref{fig:NumEx.SimpleShapes.ErrorDelta} are mainly due to interpolating $u$ in $\cV^\delta$.

The error $\|u_{\delta} - \Pi_{1}^{\varepsilon}[u_\delta](u_{\delta})\|_{L^2(\Om)}$ for varying $\varepsilon$ and fixed $\delta$ is shown in the right frame of Figure \ref{fig:NumEx.SimpleShapes.ErrorEpsilon}.
Here we observe a decay rate of $\mathcal{O}(\varepsilon)$, which is also faster than the upper estimate in Remark \ref{rem:thm:main}.
Here, for all shapes but the star, the error decreases with $\varepsilon$ down to about $10^{-9}$.
In contrast, the error for the star levels off at about~$10^{-3}$.

The significant difference in the behavior of the error for the star compared to the other shapes, shown in Figure \ref{fig:NumEx.SimpleShapes.ErrorEpsilon}, is due to the geometry of the discontinuities in the media and the mesh.
Indeed, if we repeat the experiment for the star but with a locally adapted mesh aligned with the star's geometry, as shown in Figure \ref{fig:NumEx.SimpleShapes.Star}, the error $\|u_{\delta} - \Pi_{1}^{\varepsilon}[u_\delta](u_{\delta})\|_{L^2(\Om)}$ also drops below~$10^{-8}$.
Note that while $\del$ is smaller in this test than it is in the tests shown in Figure \ref{fig:NumEx.SimpleShapes.ErrorEpsilon}, this reduction by itself is not sufficient to explain the difference in the errors between figures \ref{fig:NumEx.SimpleShapes.ErrorEpsilon} and \ref{fig:NumEx.SimpleShapes.Star}, which is of about 6 orders of magnitude.
\begin{figure}[t]
\centering
\begin{subfigure}{0.49\textwidth}
\centering
\includegraphics[width=0.75\linewidth]{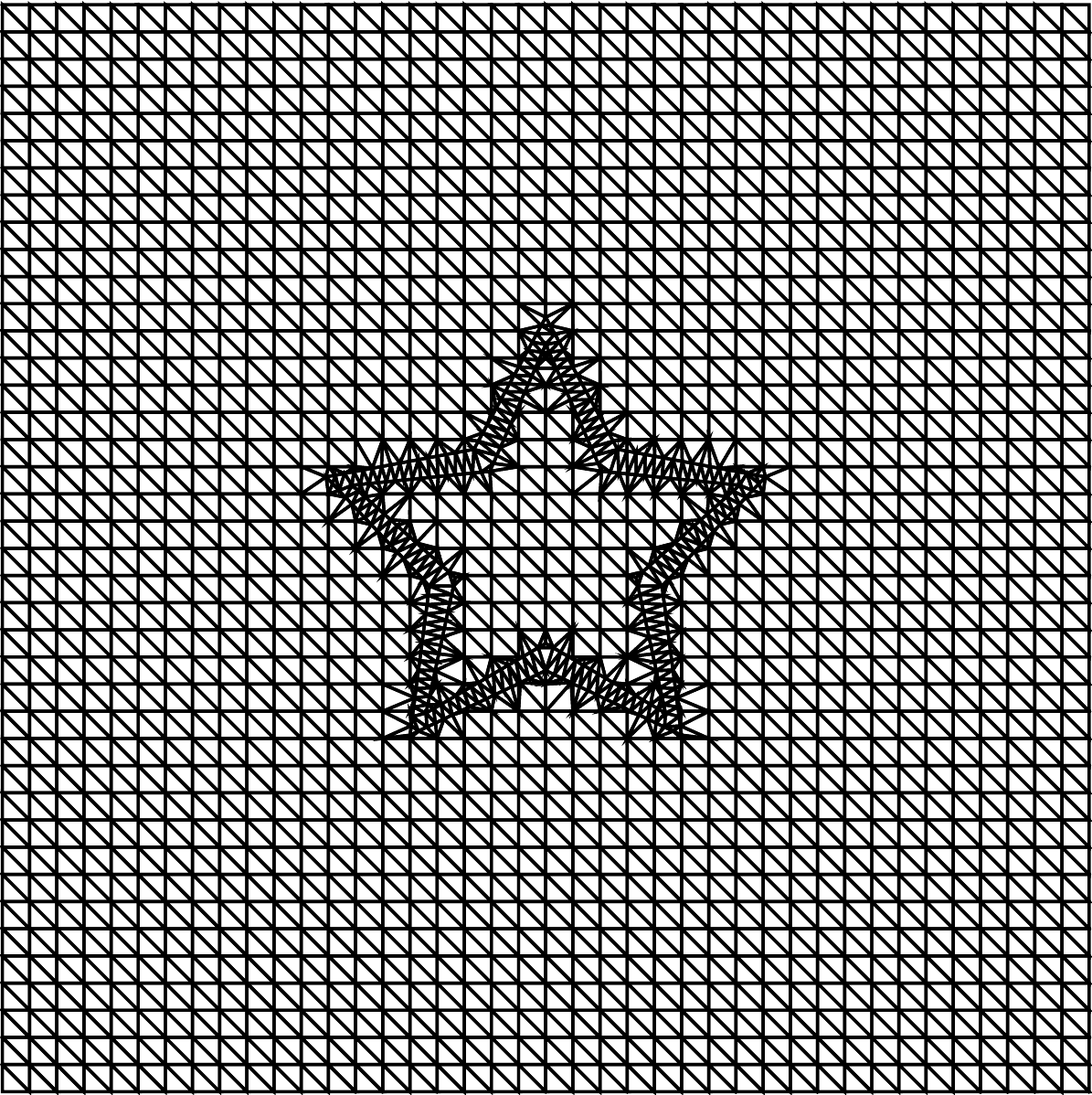}
\end{subfigure}\hfill
\begin{subfigure}{0.49\textwidth}
\centering
%
%
\definecolor{mycolor1}{rgb}{0.00000,0.44700,0.74100}%
\definecolor{mycolor2}{rgb}{0.85000,0.32500,0.09800}%
\definecolor{mycolor3}{rgb}{0.92900,0.69400,0.12500}%
\definecolor{mycolor4}{rgb}{0.49400,0.18400,0.55600}%
\begin{tikzpicture}[scale=0.8]

\begin{axis}[%
width=1\textwidth,
height=0.8\textwidth,
at={(1.054in,0.642in)},
scale only axis,
xmode=log,
xmin=0.000620568933378345,
xmax=0.0314731352948542,
xminorticks=true,
xlabel style={font=\color{white!15!black}},
xlabel={$\delta$},
ymode=log,
ymin=6.03873970818284e-10,
ymax=0.0444569852509731,
ytick={ 1e-08,  1e-06, 0.0001,   0.01},
yminorticks=true,
ylabel style={font=\color{white!15!black}},
ylabel={Error},
axis background/.style={fill=white},
xmajorgrids,
ymajorgrids,
legend style={at={(0.03,0.5)}, anchor=west, legend cell align=left, align=left, draw=white!15!black}
]
\addplot [color=mycolor4, mark size=3pt, mark=triangle, mark options={solid, rotate=270, mycolor4}]
  table[row sep=crcr]{%
0.025	8.47283812397147e-09\\
0.0125	4.28200132087559e-09\\
0.00625	3.00091097617177e-09\\
0.003125	1.45946913320421e-09\\
0.0015625	1.20034704880217e-09\\
0.00078125	1.08771791047698e-09\\
};
\addlegendentry{star}

\addplot [color=black, dashed]
  table[row sep=crcr]{%
0.025	0.025\\
0.0125	0.0125\\
0.00625	0.00625\\
0.003125	0.003125\\
0.0015625	0.0015625\\
0.00078125	0.00078125\\
};
\addlegendentry{${\delta}$}

\end{axis}
\end{tikzpicture}%
\end{subfigure}
\caption{Four simple shapes. Left: The aligned mesh for the star-shaped medium with $\delta = 0.05/2^{2}$.
Right: the error $\|u_{\delta} - \Pi_{1}^{\varepsilon}[u_\delta](u_{\delta})\|_{L^2(\Om)}$ for mesh-sizes $\delta = 0.05/2^{m}$, $m = 1, \ldots, 6$, and fixed $\varepsilon = 10^{-8}$.}
\label{fig:NumEx.SimpleShapes.Star}
\end{figure}
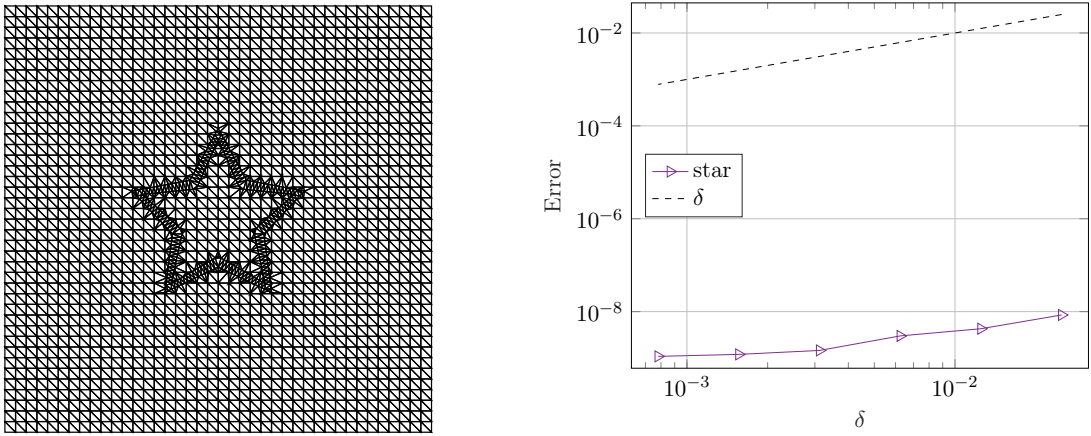

\subsection{Nonuniform background}

\begin{figure}[t]
\centering
\begin{subfigure}[c]{0.32\textwidth}
\centering
\includegraphics[height=0.9\linewidth]{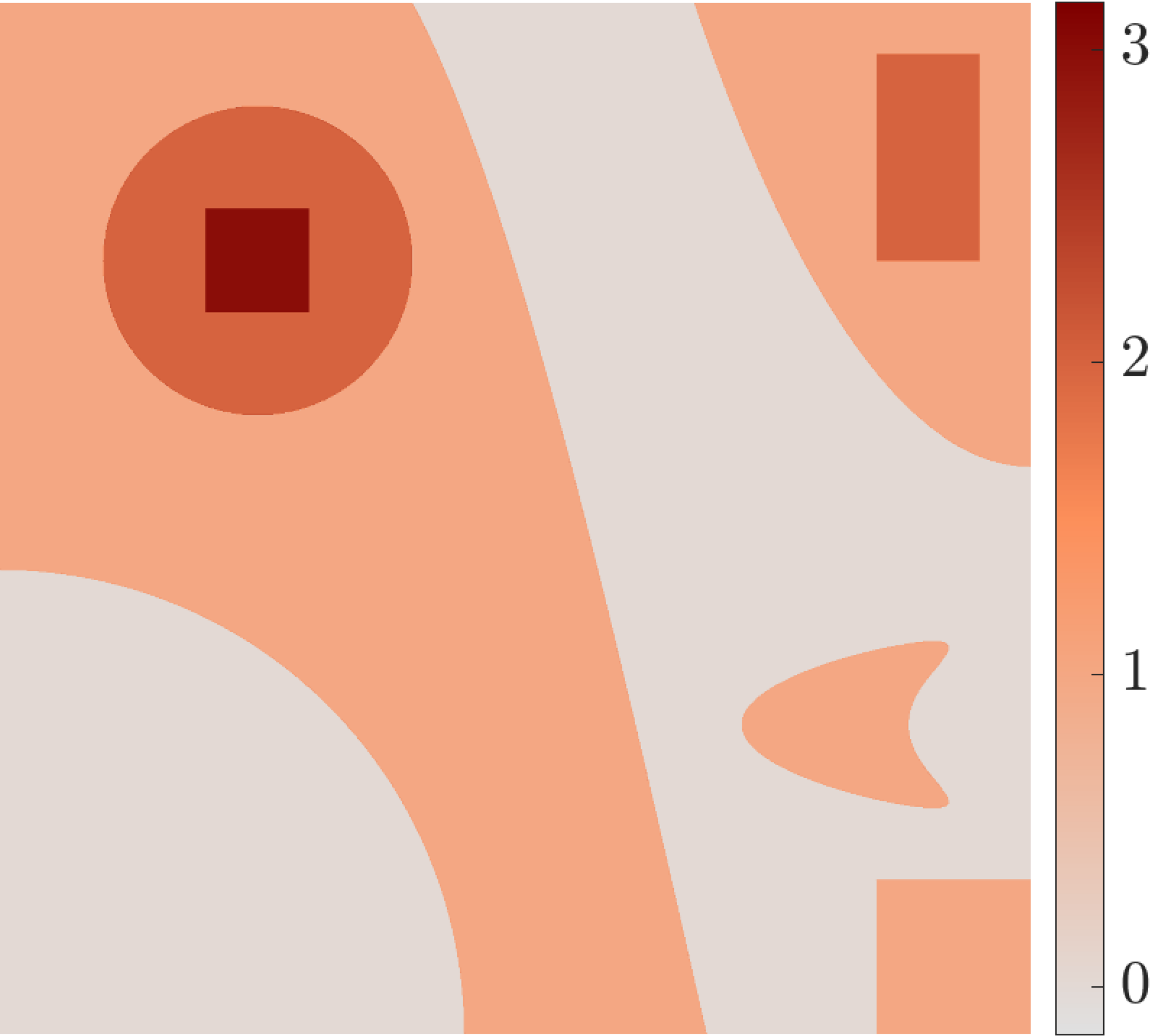}
\caption{The medium $u$ (or $u_{\delta}$)}
\label{subfig:NumEx.ComplexShape.Exact}
\end{subfigure}\hfill
\begin{subfigure}[c]{0.32\textwidth}
\centering
\includegraphics[height=0.9\linewidth]{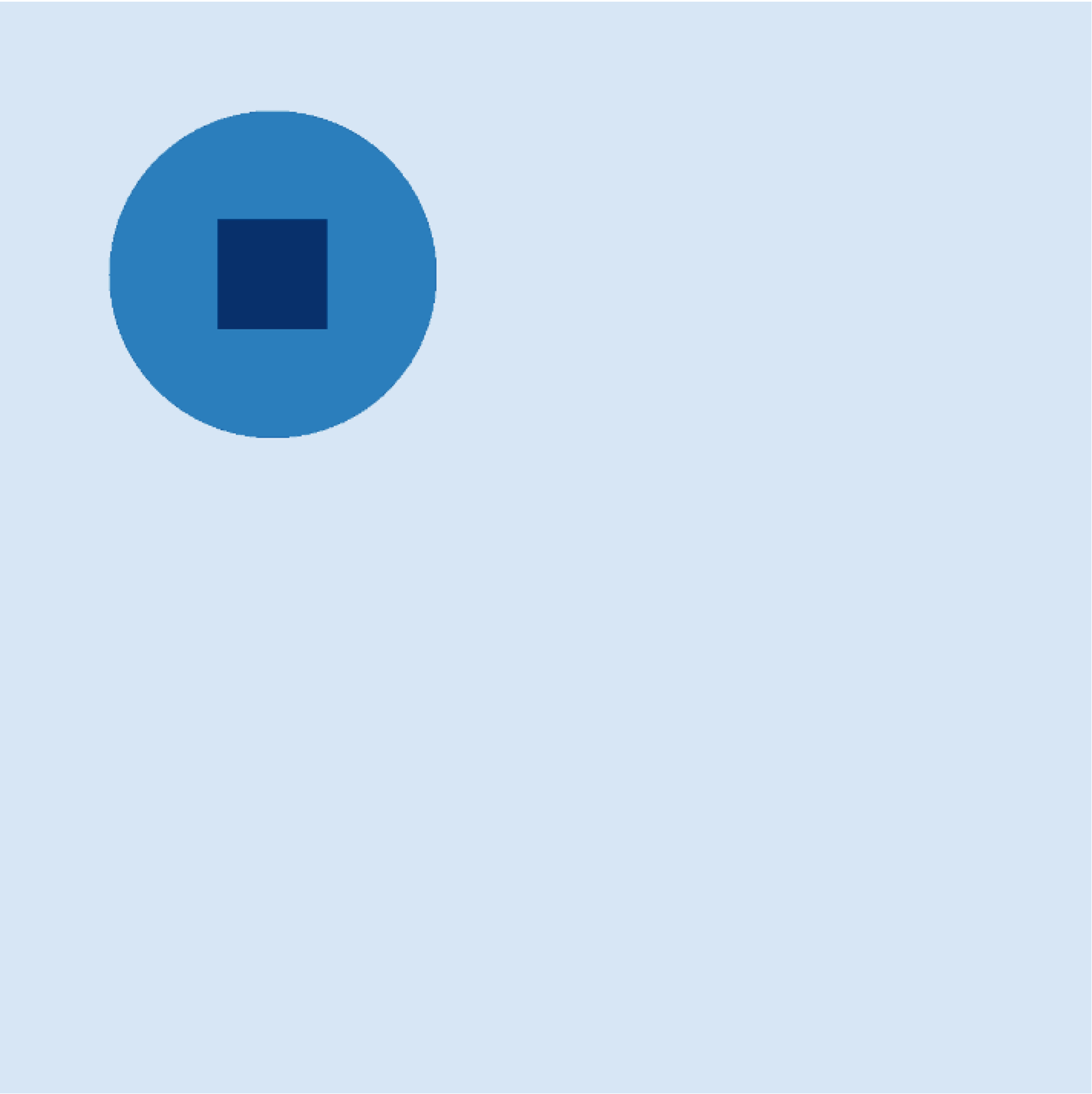}
\caption{$\varphi_{1}$ with $\lambda_{1} \approx 14.37$}
\end{subfigure}\hfill
\begin{subfigure}[c]{0.32\textwidth}
\centering
\includegraphics[height=0.9\linewidth]{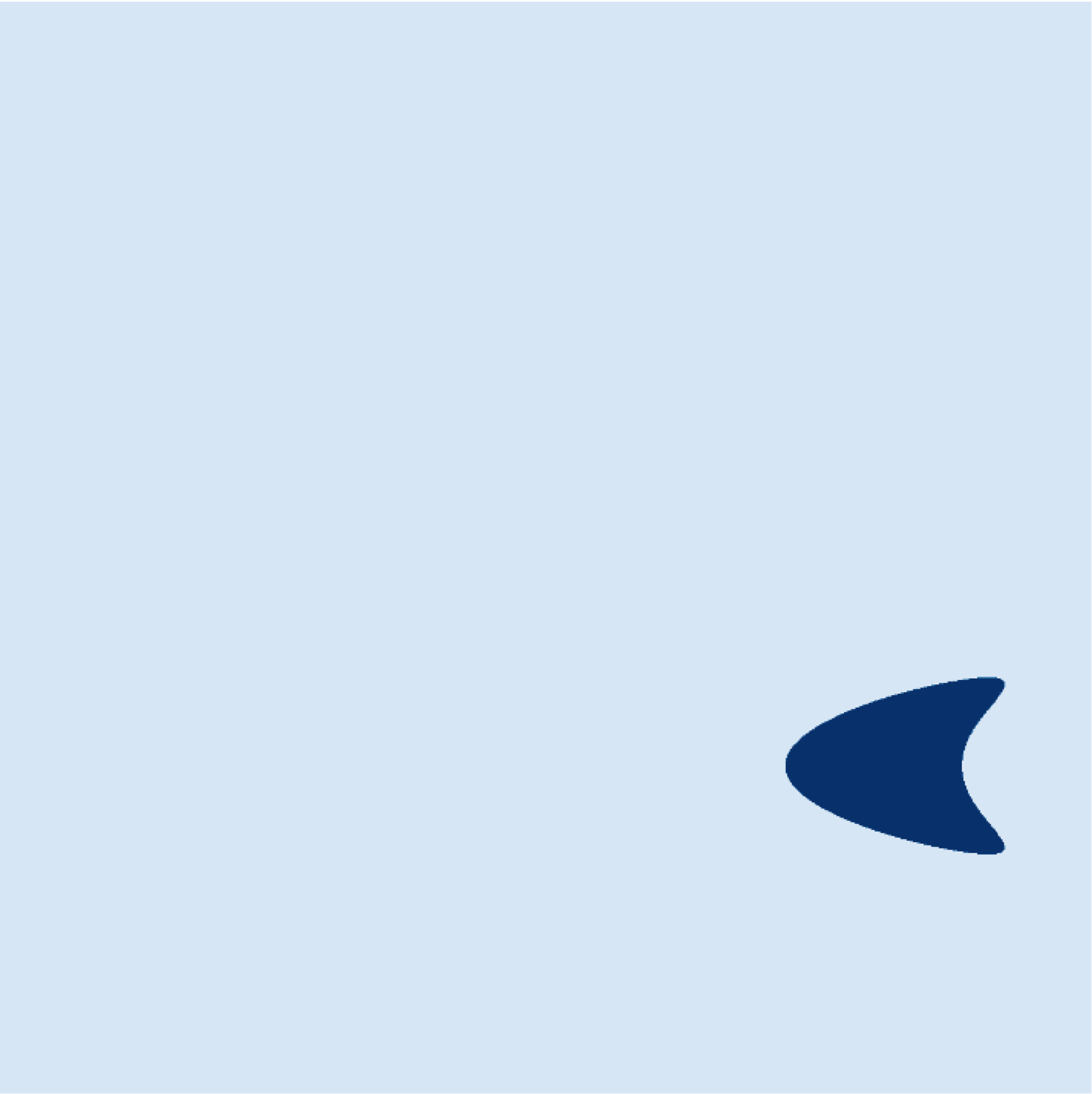}
\caption{$\varphi_{3}$ with $\lambda_{3} \approx 36.04$}
\end{subfigure}
\\[0.7em]
\begin{subfigure}[c]{0.32\textwidth}
\centering
\includegraphics[height=0.9\linewidth]{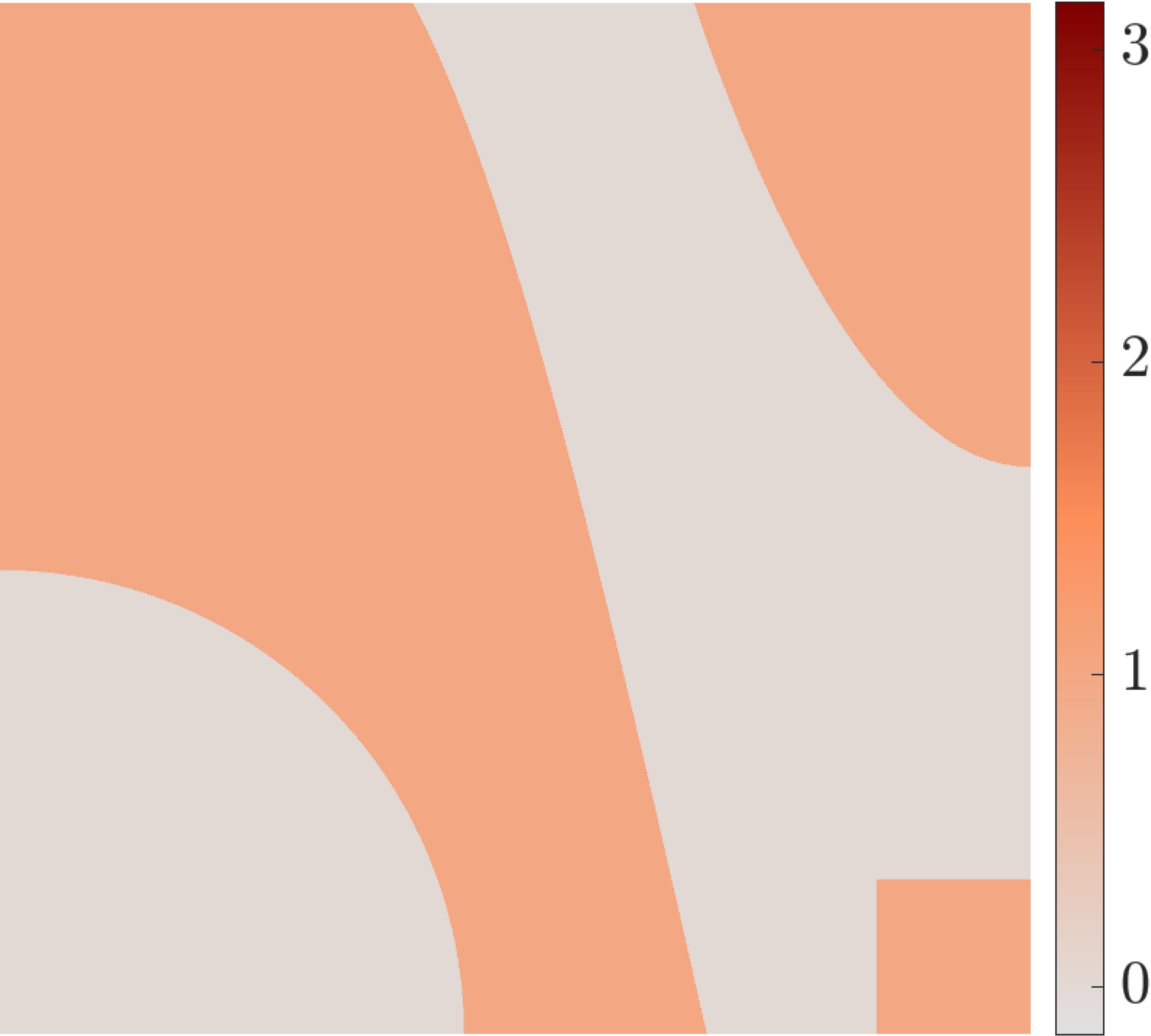}
\caption{$\varphi_{0}$}
\end{subfigure}\hfill
\begin{subfigure}[c]{0.32\textwidth}
\centering
\includegraphics[height=0.9\linewidth]{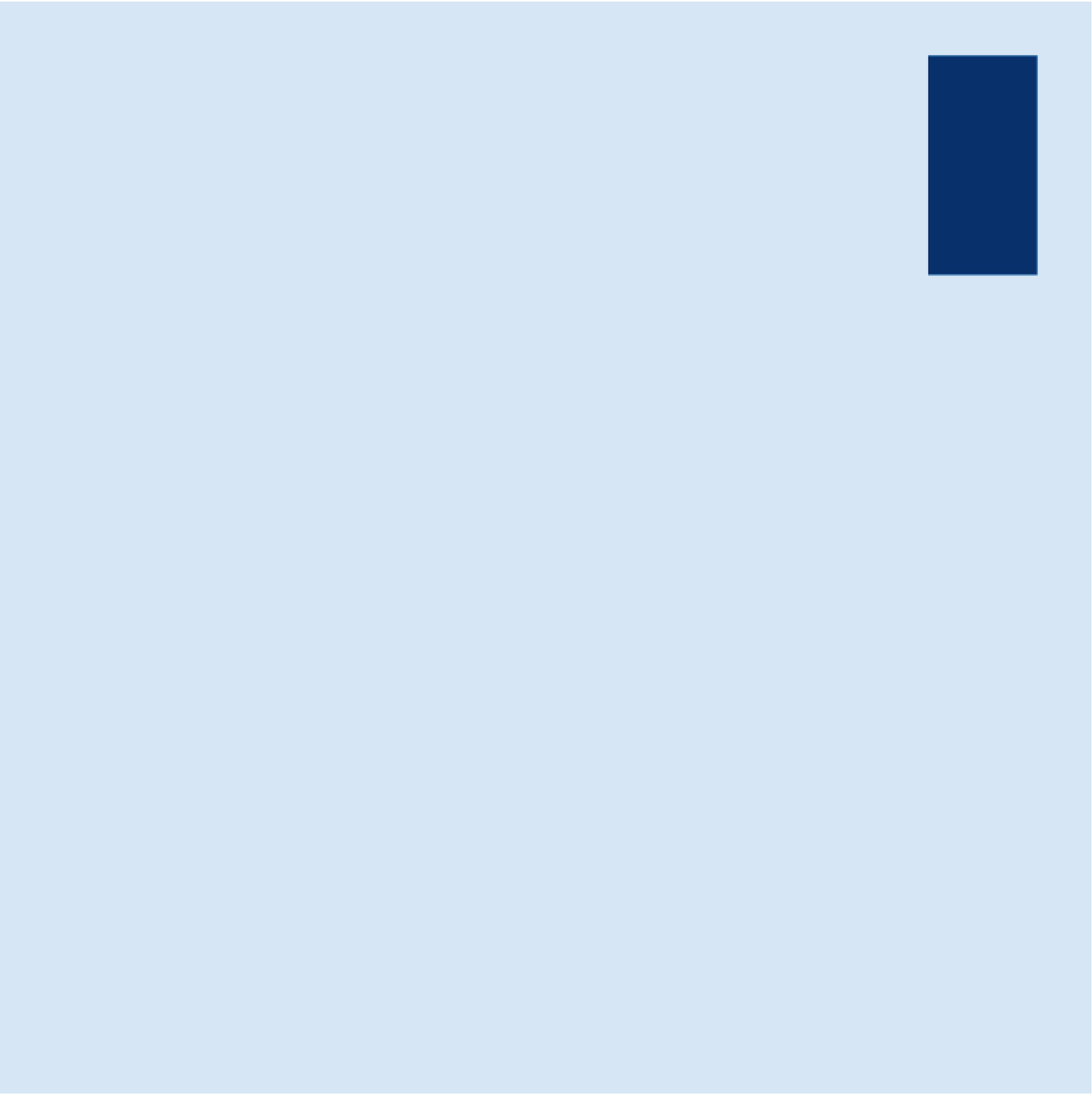}
\caption{$\varphi_{2}$ with $\lambda_{2} \approx 29.88$}
\end{subfigure}\hfill
\begin{subfigure}[c]{0.32\textwidth}
\centering
\includegraphics[height=0.9\linewidth]{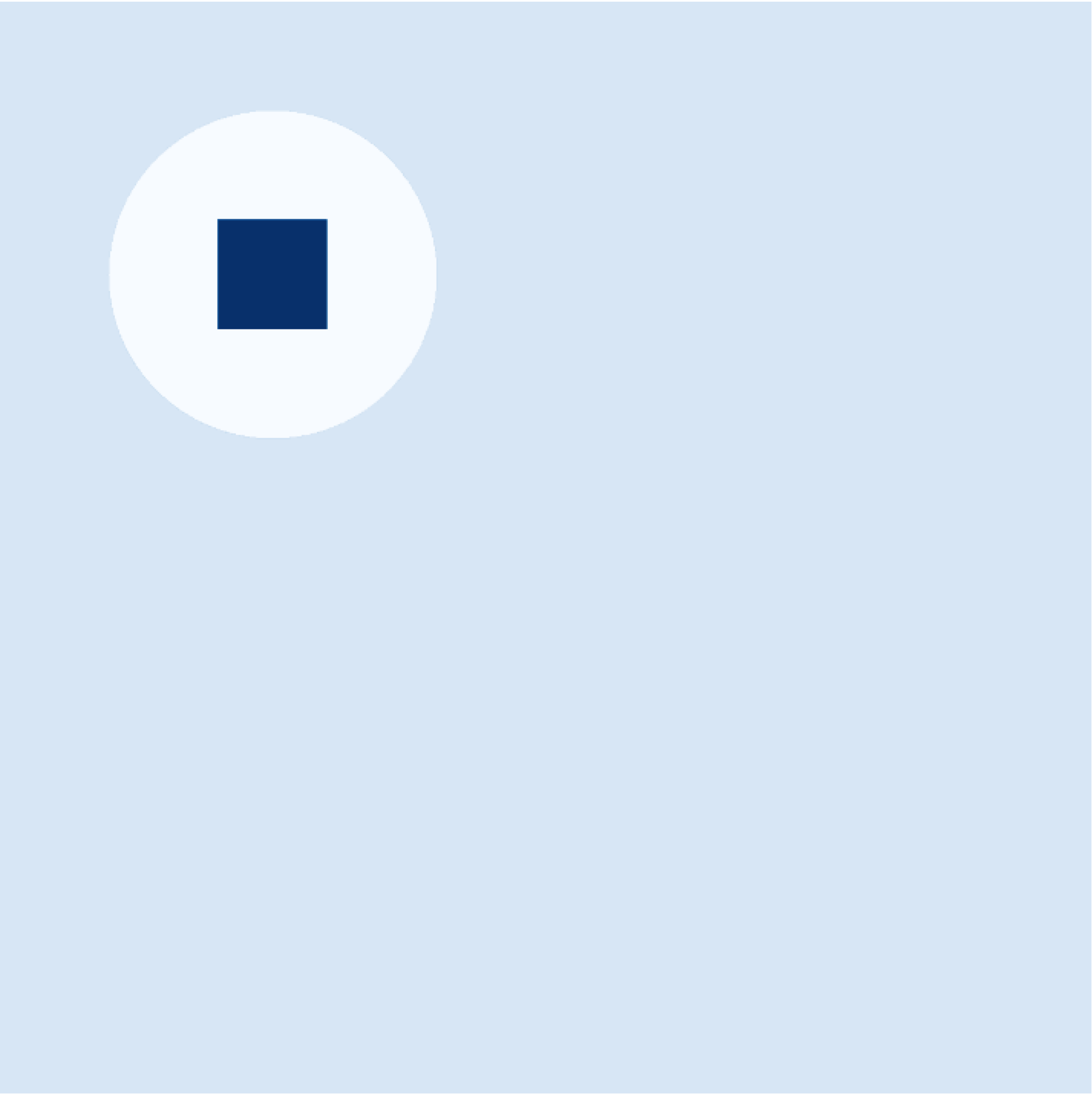}
\caption{$\varphi_{4}$ with $\lambda_{4} \approx 50.48$}
\end{subfigure}
\caption{Nonuniform background. The exact medium u with its background $\varphi_{0}$ and first four eigenpairs $(\lambda_{i},\varphi_{i})$, $i=1,\ldots,4$.}
\label{fig:NumEx.ComplexShape}
\end{figure}

Next we consider a medium $u$ with non-constant background $u^{0}$.
We let ${u:\Omega\to\mathbb{R}}$ be the medium shown in frame (a) of Fig.\ \ref{fig:NumEx.ComplexShape}, and $\Omega = (0,1)^{2}$.
Here $u$ admits a decomposition \eqref{eq:u_decomp}, \eqref{eq:u0_wtu} with ${M = 5}$ and ${K = 4}$.
Figure \ref{fig:NumEx.ComplexShape} also shows the approximation $\varphi_{0}$ of the background and the first four eigenfunctions $\vphi_1,\ldots,\vphi_4$ of $L_\veps[u_\delta]$.

Figure \ref{fig:NumEx.ComplexShape.Error} (left) shows the error $\|u - Q_{K}^{\varepsilon}[u_\delta](u)\|_{L^2(\Om)}$ with $K=4$, for six different meshes with $\delta = 0.05/2^{m}$, $m=1,\ldots,6$.
Here we observe an error decay of $\cO(\sqrt{\del})$, consistent with our theoretical estimates.
The right frame of Figure \ref{fig:NumEx.ComplexShape.Error} shows the error $\|u_{\delta} - Q_{K}^{\varepsilon}[u_\delta](u_{\delta})\|_{L^2(\Om)}$ with $K=4$, as a function of $\veps$ with fixed $\delta = 0.05/2^{6}$. Again, we observe a convergence rate of $\mathcal{O}(\varepsilon)$, faster than the $\mathcal{O}(\sqrt{\varepsilon})$ rate proved in Remark \ref{rem:thm:main}.

\begin{figure}[t]
\centering
\begin{subfigure}{0.49\textwidth}
\centering
%
%
\definecolor{mycolor1}{rgb}{0.00000,0.44700,0.74100}%
\begin{tikzpicture}[scale=0.8]

\begin{axis}[%
width=1\textwidth,
height=0.8\textwidth,
at={(1.054in,0.642in)},
scale only axis,
xmode=log,
xmin=0.000696289795416989,
xmax=0.0280504613575491,
xminorticks=true,
xlabel style={font=\color{white!15!black}},
xlabel={$\delta$},
ymode=log,
ymin=0.0243880596088015,
ymax=0.1710195442934,
ytick={0.01,0.0316227766016838,0.1},
yticklabels={{$10^{-2}$},{$10^{-1.5}$},{$10^{-1}$}},
yminorticks=true,
ylabel style={font=\color{white!15!black}},
ylabel={Error},
axis background/.style={fill=white},
xmajorgrids,
ymajorgrids,
legend style={at={(0.03,0.97)}, anchor=north west, legend cell align=left, align=left, draw=white!15!black}
]
\addplot [color=mycolor1, mark size=4.0pt, mark=o, mark options={solid, mycolor1}]
  table[row sep=crcr]{%
0.025	0.152421329290698\\
0.0125	0.108839935066193\\
0.00625	0.0779047793743842\\
0.003125	0.054856772214692\\
0.0015625	0.038750266508334\\
0.00078125	0.0273638529456916\\
};
\addlegendentry{medium}

\addplot [color=black, dashed]
  table[row sep=crcr]{%
0.025	0.142302494707577\\
0.0125	0.100623058987491\\
0.00625	0.0711512473537885\\
0.003125	0.0503115294937453\\
0.0015625	0.0355756236768943\\
0.00078125	0.0251557647468726\\
};
\addlegendentry{$\sqrt{\delta}$}

\end{axis}
\end{tikzpicture}%
\end{subfigure}\hfill
\begin{subfigure}{0.49\textwidth}
\centering
%
%
\definecolor{mycolor1}{rgb}{0.00000,0.44700,0.74100}%
\begin{tikzpicture}[scale=0.8]

\begin{axis}[%
width=1\textwidth,
height=0.8\textwidth,
at={(1.054in,0.642in)},
scale only axis,
xmode=log,
xmin=5.62341325190349e-09,
xmax=1.77827941003892,
xtick={1e-08,1e-06,0.0001,0.01,1},
xticklabels={{$10^{-8}$},{$10^{-6}$},{$10^{-4}$},{$10^{-2}$},{$10^0$}},
xminorticks=true,
xlabel style={font=\color{white!15!black}},
xlabel={$\varepsilon$},
ymode=log,
ymin=2.11135529485285e-09,
ymax=0.345605515179625,
yminorticks=true,
ylabel style={font=\color{white!15!black}},
ylabel={Error},
axis background/.style={fill=white},
xmajorgrids,
ymajorgrids,
legend style={at={(0.03,0.97)}, anchor=north west, legend cell align=left, align=left, draw=white!15!black}
]
\addplot [color=mycolor1, mark size=4.0pt, mark=o, mark options={solid, mycolor1}]
  table[row sep=crcr]{%
1	0.194348263399203\\
0.1	0.0318492760078309\\
0.01	0.0034110353619469\\
0.001	0.000344025619626833\\
0.0001	3.44383889473692e-05\\
1e-05	3.47888573932799e-06\\
1e-06	3.63739216694582e-07\\
1e-07	3.73367424979791e-08\\
1e-08	3.75457964811348e-09\\
};
\addlegendentry{medium}

\addplot [color=black, dashed]
  table[row sep=crcr]{%
1	1\\
0.1	0.1\\
0.01	0.01\\
0.001	0.001\\
0.0001	0.0001\\
1e-05	1e-05\\
1e-06	1e-06\\
1e-07	1e-07\\
1e-08	1e-08\\
};
\addlegendentry{$\varepsilon$}

\end{axis}
\end{tikzpicture}%
\end{subfigure}
\caption{Nonuniform background. Left: the error $\|u - Q_{4}^{\varepsilon}[u_\delta](u)\|_{L^2(\Om)}$ for mesh-sizes $\delta = 0.05/2^{m}$, $m = 1, \ldots, 6$, and fixed $\varepsilon = 10^{-8}$.
Right: the error $\|u_{\delta} - Q_{4}^{\varepsilon}[u_\delta](u_{\delta})\|_{L^2(\Om)}$ for $\varepsilon = 10^{-m}$, $m = 0,\ldots,8$, and fixed mesh-size $\delta = 0.05/2^6$.}
\label{fig:NumEx.ComplexShape.Error}
\end{figure}
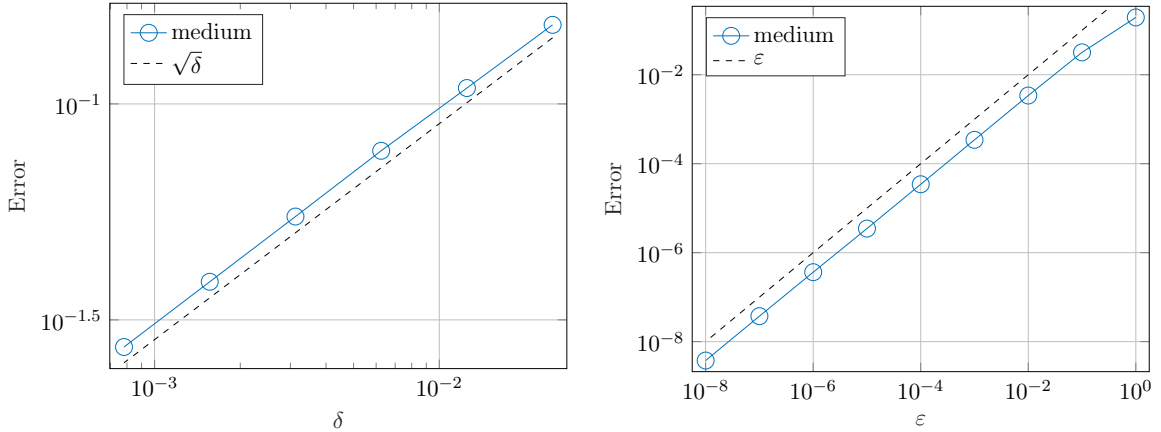

\subsection{Four adjacent Squares}

\begin{figure}[t]
\centering
\begin{subfigure}[c]{0.32\textwidth}
\centering
\includegraphics[height=0.9\linewidth]{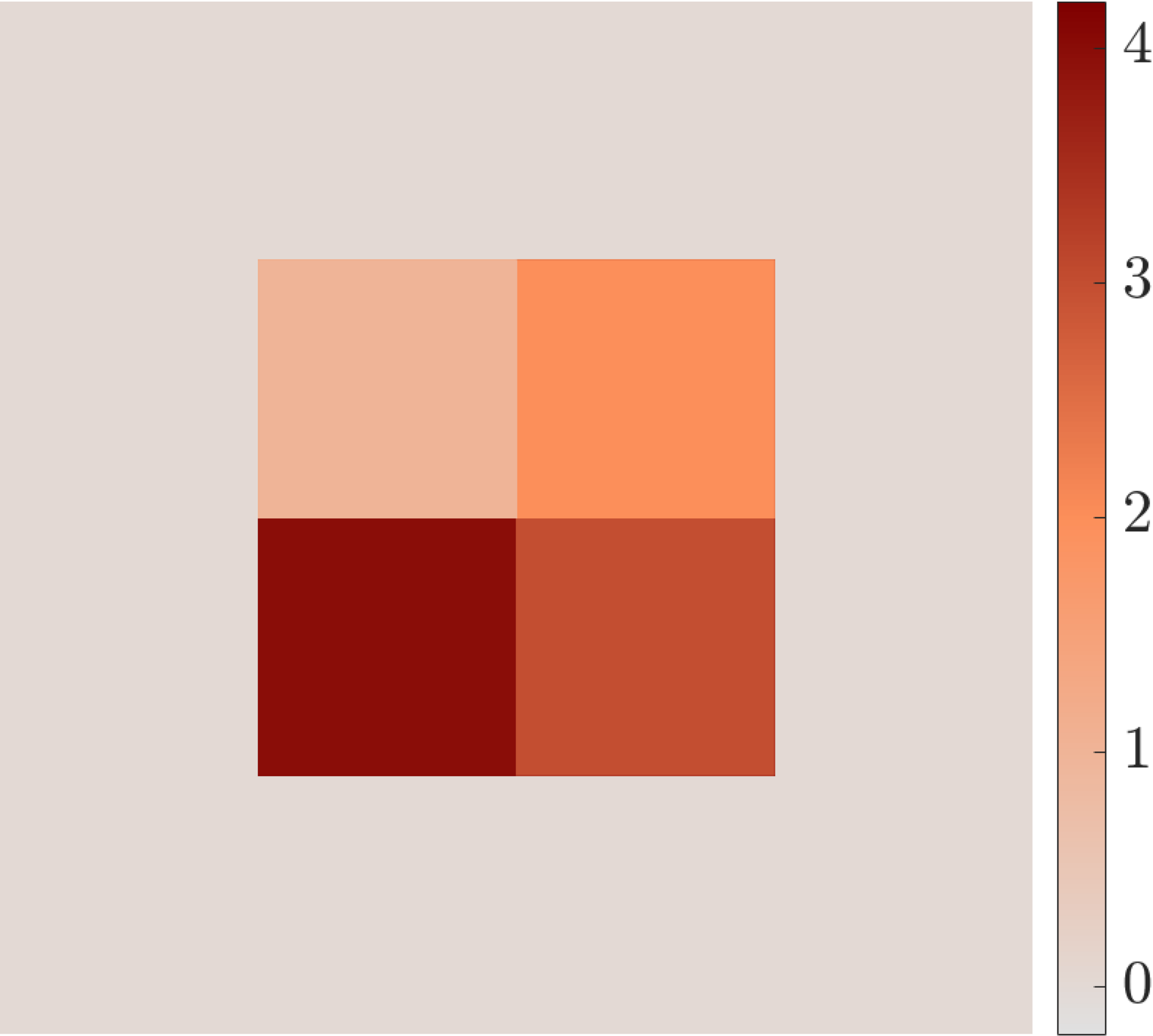}
\caption{The medium $u$ (or $u_{\delta}$)} \label{subfig:NumEx.FourSquares.Medium}
\end{subfigure}\hfill
\begin{subfigure}[c]{0.32\textwidth}
\centering
\includegraphics[height=0.9\linewidth]{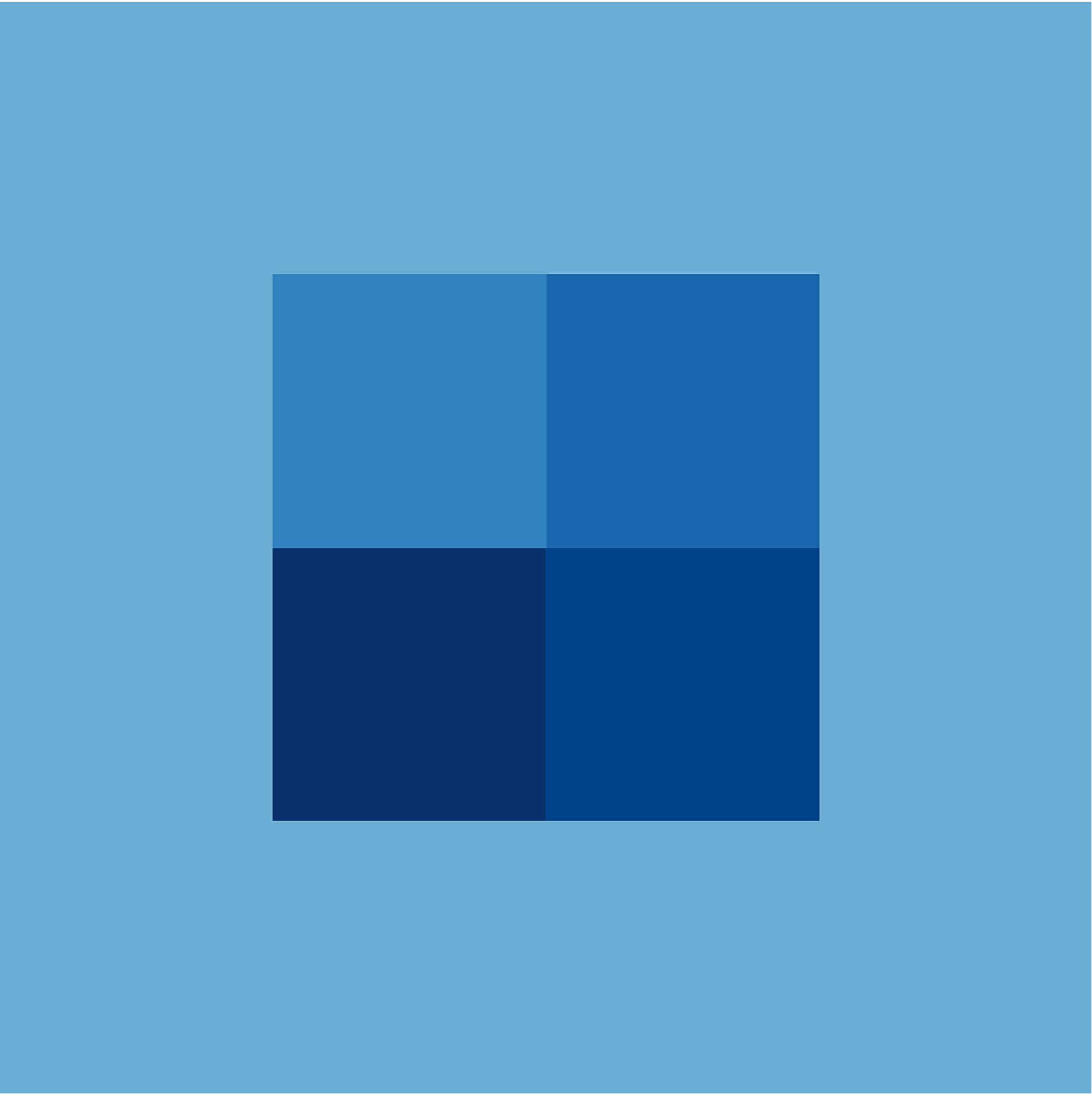}
\caption{$\varphi_{1}$ with $\lambda_{1} \approx 3.3$}
\end{subfigure}\hfill
\begin{subfigure}[c]{0.32\textwidth}
\centering
\includegraphics[height=0.9\linewidth]{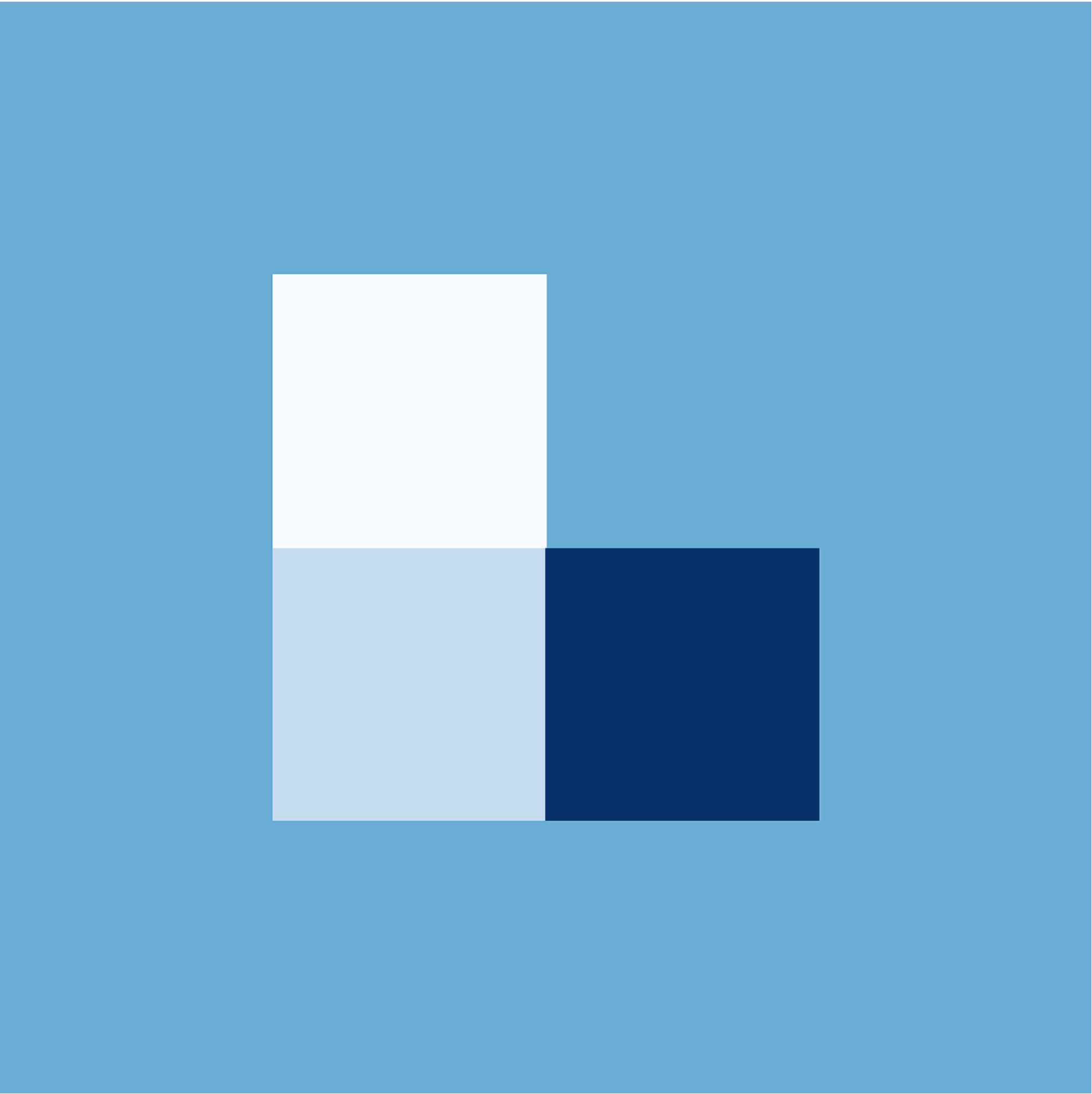}
\caption{$\varphi_{3}$ with $\lambda_{3} \approx 12.65$}
\end{subfigure}
\\[0.7em]
\begin{subfigure}[c]{0.32\textwidth}
\centering
\includegraphics[height=0.9\linewidth]{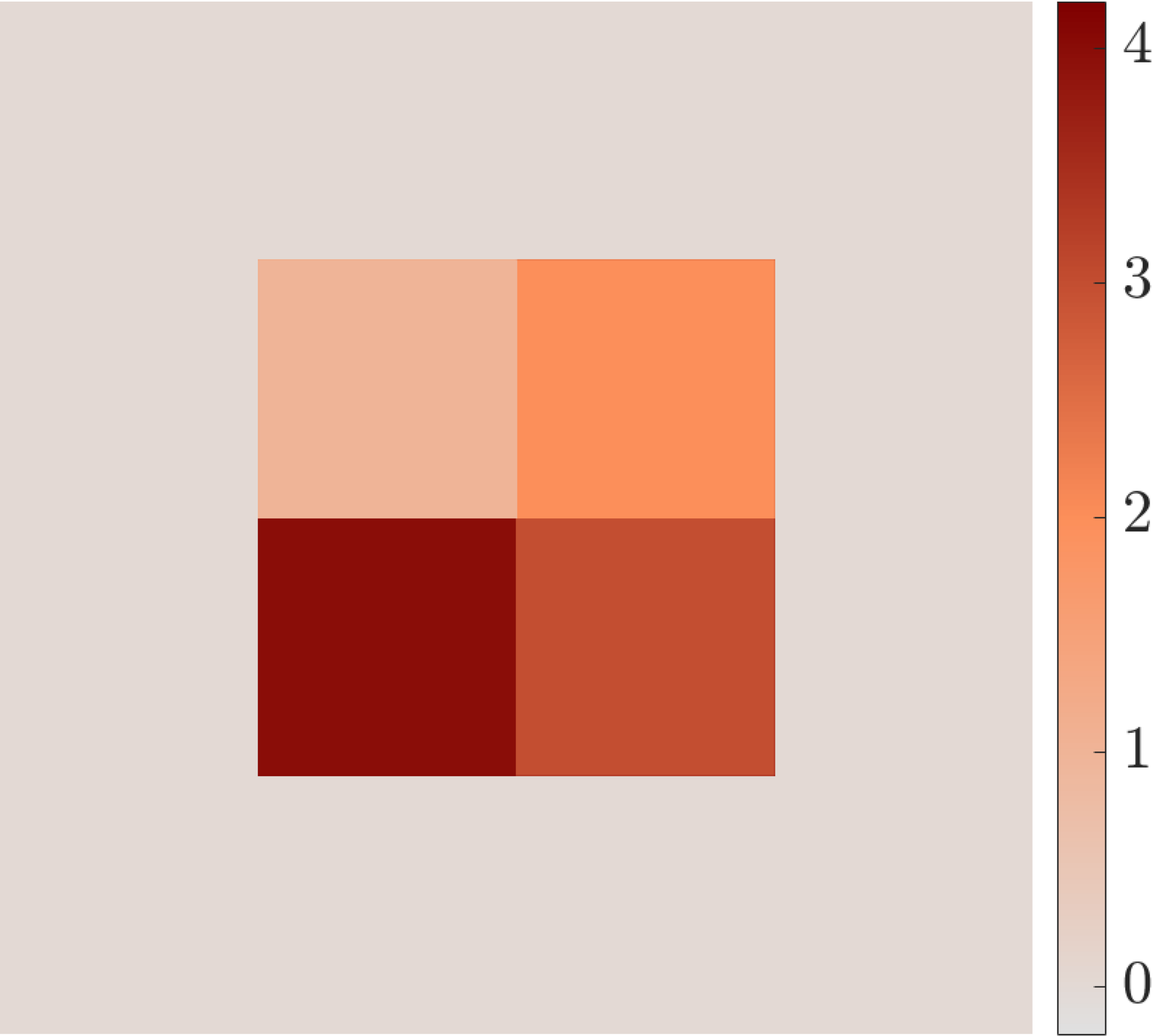}
\caption{$\Pi_{4}^{\varepsilon}[u_\delta]u_{\delta}$} \label{subfig:NumEx.FourSquares.ASdecomposition}
\end{subfigure}\hfill
\begin{subfigure}[c]{0.32\textwidth}
\centering
\includegraphics[height=0.9\linewidth]{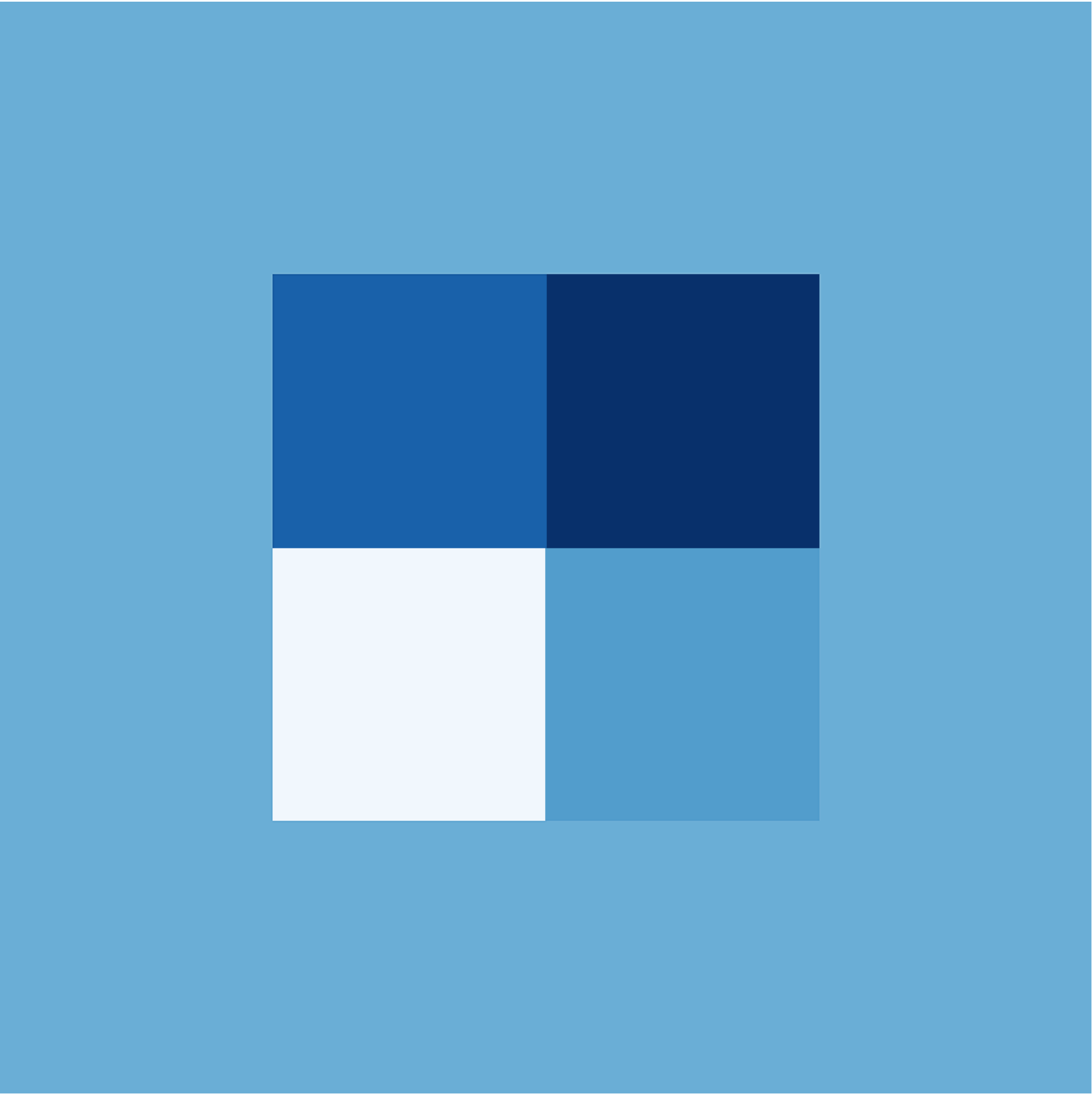}
\caption{$\varphi_{2}$ with $\lambda_{2} \approx 8.88$}
\end{subfigure}\hfill
\begin{subfigure}[c]{0.32\textwidth}
\centering
\includegraphics[height=0.9\linewidth]{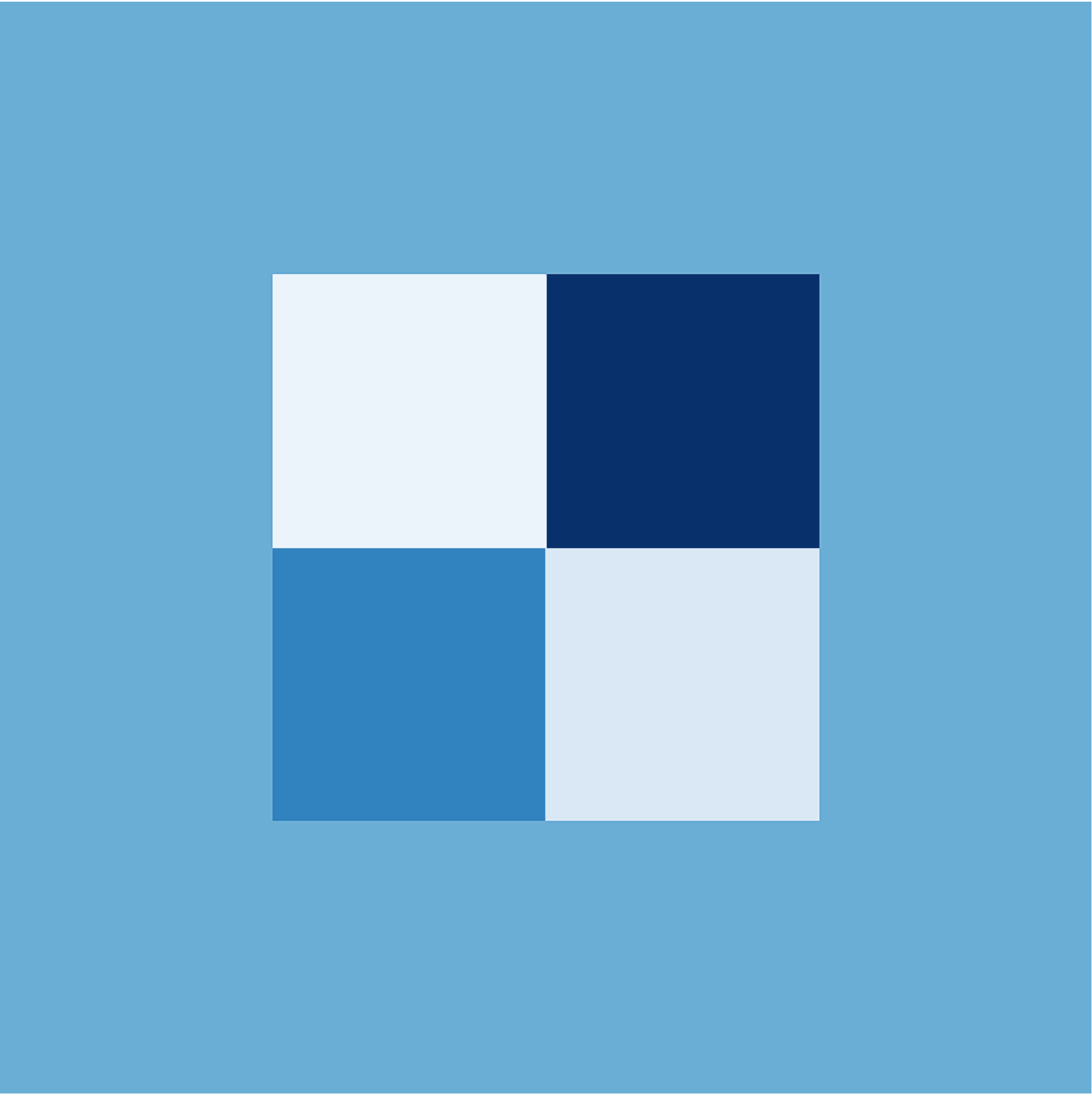}
\caption{$\varphi_{4}$ with $\lambda_{4} \approx 18.37$}
\end{subfigure}
\caption{Adjacent squares. The medium $u$ and the first four eigenfunctions $\varphi_{k}$, $k=1,\ldots,4$, of the operator $L_{\varepsilon}[u_{\delta}]$, together with its AS decomposition $\Pi_{4}^{\varepsilon}[u_\delta](u_{\delta})$ computed on a mesh with $\delta = 0.05/2^{6}$.}
\label{fig:NumEx.FourSquares}
\end{figure}

Let $\Om$ be the unit square $\Omega = (0,1)^{2}$ and
\begin{align}
	u(x) = \sum_{k=1}^{4}\alpha_{k}\chi_{A^{k}}(x), \quad x \in \Omega,
\end{align}
with $\alpha_{k} = k$, for $k = 1, \ldots, 4$, the piecewise constant medium shown in Fig.\ \ref{fig:NumEx.FourSquares}.
Since the boundaries $\partial A^{k}$ of the squares $A^{k}$ are not mutually disjoint, this example is not covered by our analysis.
However, we may still compute the AS approximation and measure the approximation error.

In Figure \ref{fig:NumEx.FourSquares.Error} we still observe errors of $\cO(\sqrt{\del})$, consistent with our theoretical estimates.
Again, the error with respect to $\varepsilon$ decays with a rate of $\mathcal{O}(\varepsilon)$, as seen in Figure~\ref{fig:NumEx.FourSquares.Error}.

\begin{figure}[t]
\centering
\begin{subfigure}{0.49\textwidth}
\centering
%
%
\definecolor{mycolor1}{rgb}{0.00000,0.44700,0.74100}%
\begin{tikzpicture}[scale=0.8]

\begin{axis}[%
width=1\textwidth,
height=0.8\textwidth,
at={(1.054in,0.642in)},
scale only axis,
xmode=log,
xmin=0.000696289795416989,
xmax=0.0280504613575491,
xminorticks=true,
xlabel style={font=\color{white!15!black}},
xlabel={$\delta$},
ymode=log,
ymin=0.0609897642467857,
ymax=0.427687835052746,
ytick={0.1,0.316227766016838},
yticklabels={{$10^{-1}$},{$10^{-0.5}$}},
yminorticks=true,
ylabel style={font=\color{white!15!black}},
ylabel={Error},
axis background/.style={fill=white},
xmajorgrids,
ymajorgrids,
legend style={at={(0.03,0.97)}, anchor=north west, legend cell align=left, align=left, draw=white!15!black}
]
\addplot [color=mycolor1, mark size=4.0pt, mark=o, mark options={solid, mycolor1}]
  table[row sep=crcr]{%
0.025	0.381177184219151\\
0.0125	0.271703379326638\\
0.00625	0.192886848511477\\
0.003125	0.136661182666816\\
0.0015625	0.096729323522213\\
0.00078125	0.0684316410084196\\
};
\addlegendentry{medium}

\addplot [color=black, dashed]
  table[row sep=crcr]{%
0.025	0.316227766016838\\
0.0125	0.223606797749979\\
0.00625	0.158113883008419\\
0.003125	0.111803398874989\\
0.0015625	0.0790569415042095\\
0.00078125	0.0559016994374947\\
};
\addlegendentry{$\sqrt{\delta}$}

\end{axis}
\end{tikzpicture}%
\end{subfigure}\hfill
\begin{subfigure}{0.49\textwidth}
\centering
%
%
\definecolor{mycolor1}{rgb}{0.00000,0.44700,0.74100}%
\begin{tikzpicture}[scale=0.8]

\begin{axis}[%
width=1\textwidth,
height=0.8\textwidth,
at={(1.054in,0.642in)},
scale only axis,
xmode=log,
xmin=5.62341325190349e-09,
xmax=1.77827941003892,
xtick={1e-08,1e-06,0.0001,0.01,1},
xticklabels={{$10^{-8}$},{$10^{-6}$},{$10^{-4}$},{$10^{-2}$},{$10^0$}},
xminorticks=true,
xlabel style={font=\color{white!15!black}},
xlabel={$\varepsilon$},
ymode=log,
ymin=6.52486170217346e-10,
ymax=0.14386692722579,
ytick={1e-08,1e-06,0.0001,0.01},
yticklabels={{$10^{-8}$},{$10^{-6}$},{$10^{-4}$},{$10^{-2}$}},
yminorticks=true,
ylabel style={font=\color{white!15!black}},
ylabel={Error},
axis background/.style={fill=white},
xmajorgrids,
ymajorgrids,
legend style={at={(0.03,0.97)}, anchor=north west, legend cell align=left, align=left, draw=white!15!black}
]
\addplot [color=mycolor1, mark size=4.0pt, mark=o, mark options={solid, mycolor1}]
  table[row sep=crcr]{%
1	0.0809023185072145\\
0.1	0.00892306050068951\\
0.01	0.000901338785228069\\
0.001	9.02268005380683e-05\\
0.0001	9.04166017680455e-06\\
1e-05	9.49848264477137e-07\\
1e-06	1.04414510850724e-07\\
1e-07	1.08933642131198e-08\\
1e-08	1.16030272183266e-09\\
};
\addlegendentry{medium}

\addplot [color=black, dashed]
  table[row sep=crcr]{%
1	1\\
0.1	0.1\\
0.01	0.01\\
0.001	0.001\\
0.0001	0.0001\\
1e-05	1e-05\\
1e-06	1e-06\\
1e-07	1e-07\\
1e-08	1e-08\\
};
\addlegendentry{$\varepsilon$}

\end{axis}
\end{tikzpicture}%
\end{subfigure}
\caption{Adjacent squares. Left: the error $\|u - \Pi_{4}^{\varepsilon}[u_\delta](u)\|_{L^2(\Om)}$ for mesh-sizes $\delta = 0.05/2^{m}$, $m = 1, \ldots, 6$, and fixed $\varepsilon = 10^{-8}$.
Right: the error $\|u_{\delta} - \Pi_{4}^{\varepsilon}[u_\delta](u_{\delta})\|_{L^2(\Om)}$ for $\varepsilon = 10^{-m}$, $m = 0,\ldots,8$, and fixed mesh-size $\delta = 0.05/2^6$;}
\label{fig:NumEx.FourSquares.Error}
\end{figure}
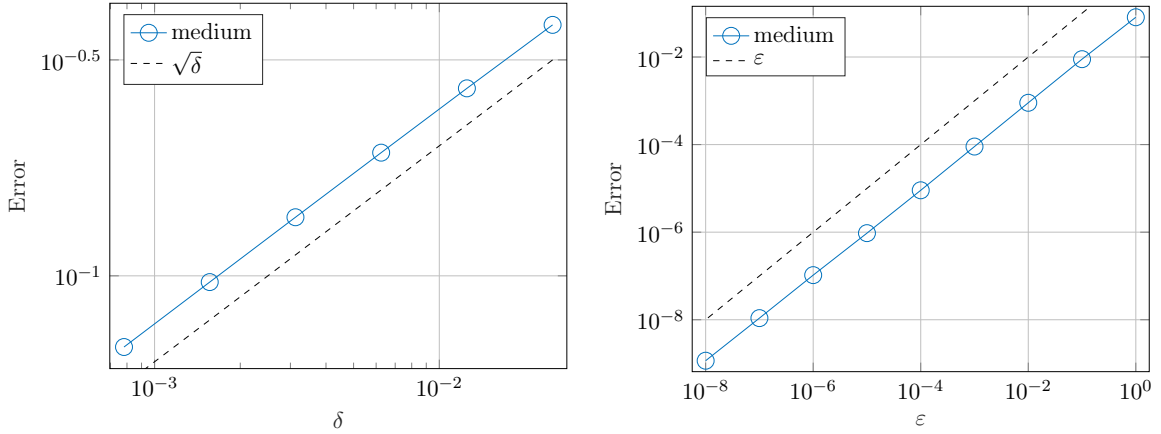

\subsection{Map of Switzerland}


\begin{figure}[t]
\centering
\subcaptionbox{Polygonal Switzerland%
		\label{subfig:NumEx.Switzerland.Polygonal}}%
	{\includegraphics[width=0.50\linewidth]
		{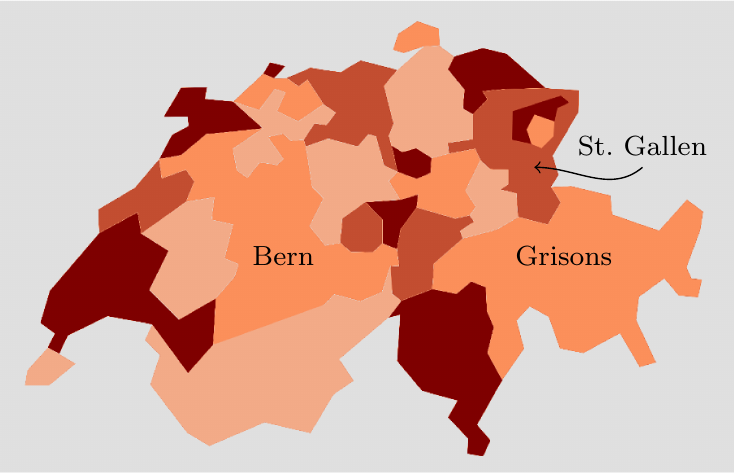}}
\hfill
\subcaptionbox{3D-view of $\varphi_{5}$}%
	{\includegraphics[width=0.4\linewidth]
		{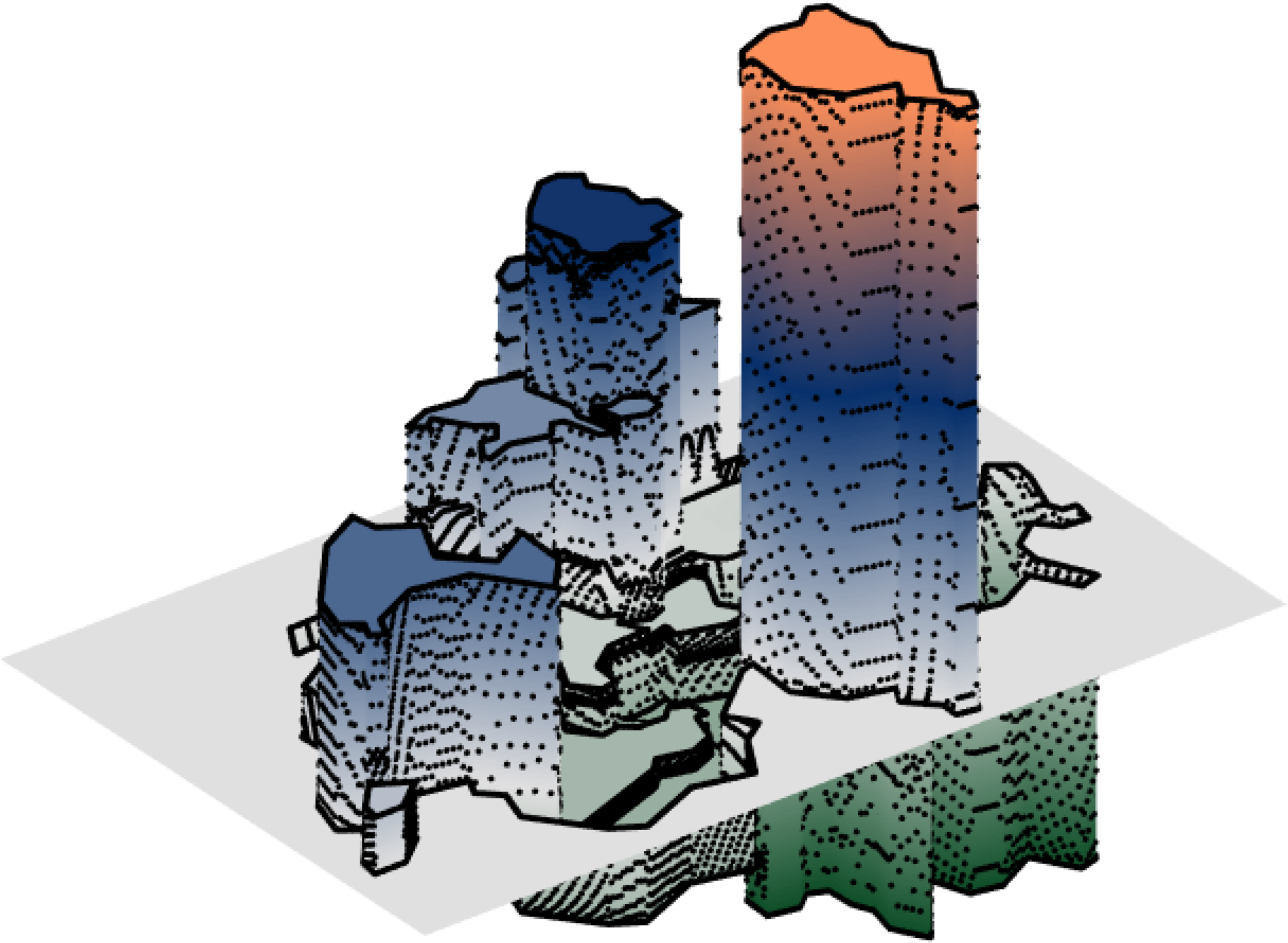}}
\\[0.7em]
\begin{subfigure}[c]{0.32\textwidth}
\centering
\includegraphics[width=0.9\linewidth]{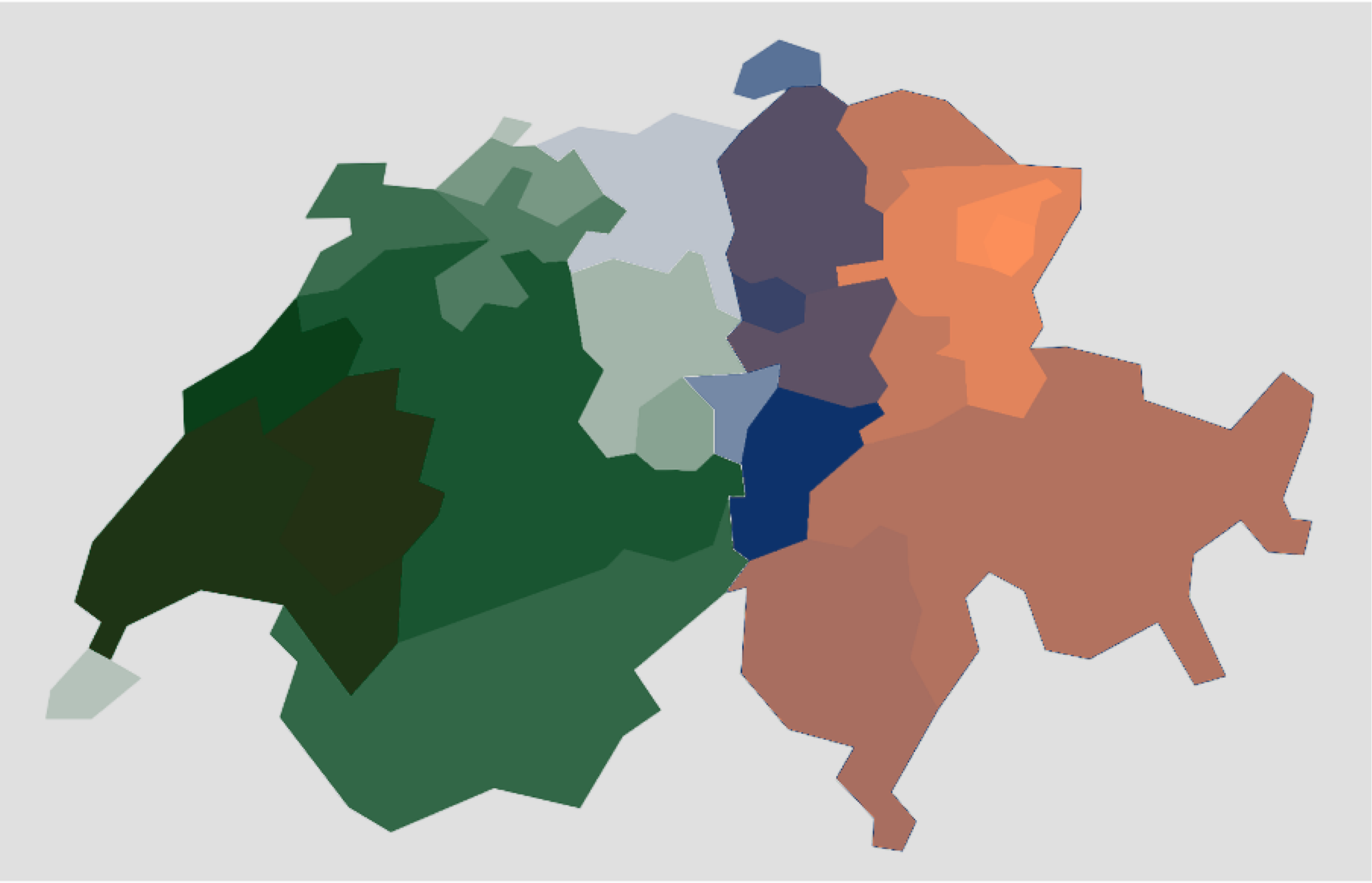}
\caption{$\varphi_{2}$ with $\lambda_{2} \approx 28.75$}
\end{subfigure}
\hfill
\begin{subfigure}[c]{0.32\textwidth}
\centering
\includegraphics[width=0.9\linewidth]{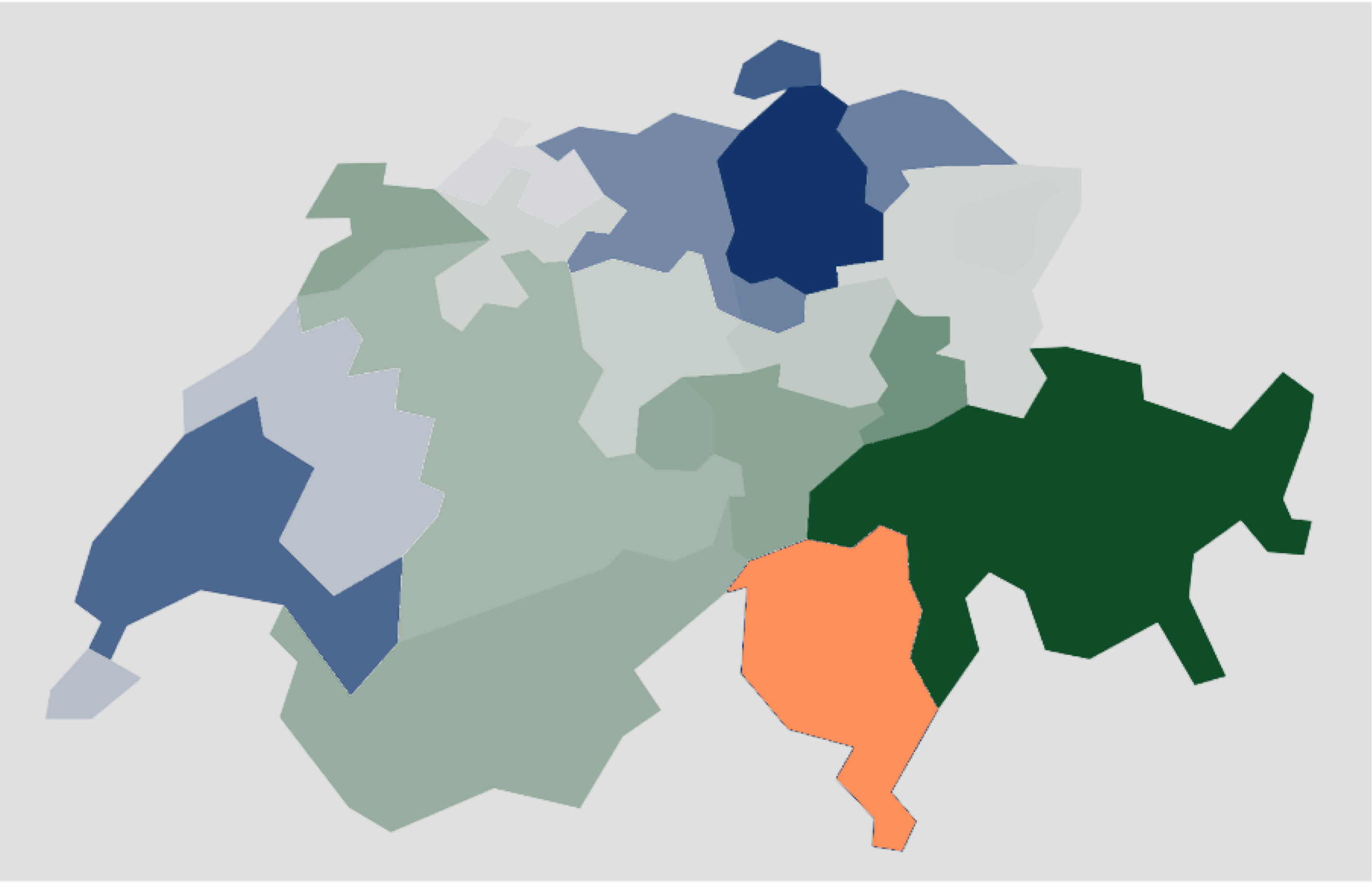}
\caption{$\varphi_{5}$ with $\lambda_{5} \approx 85.03$}
\end{subfigure}
\hfill
\begin{subfigure}[c]{0.32\textwidth}
\centering
\includegraphics[width=0.9\linewidth]{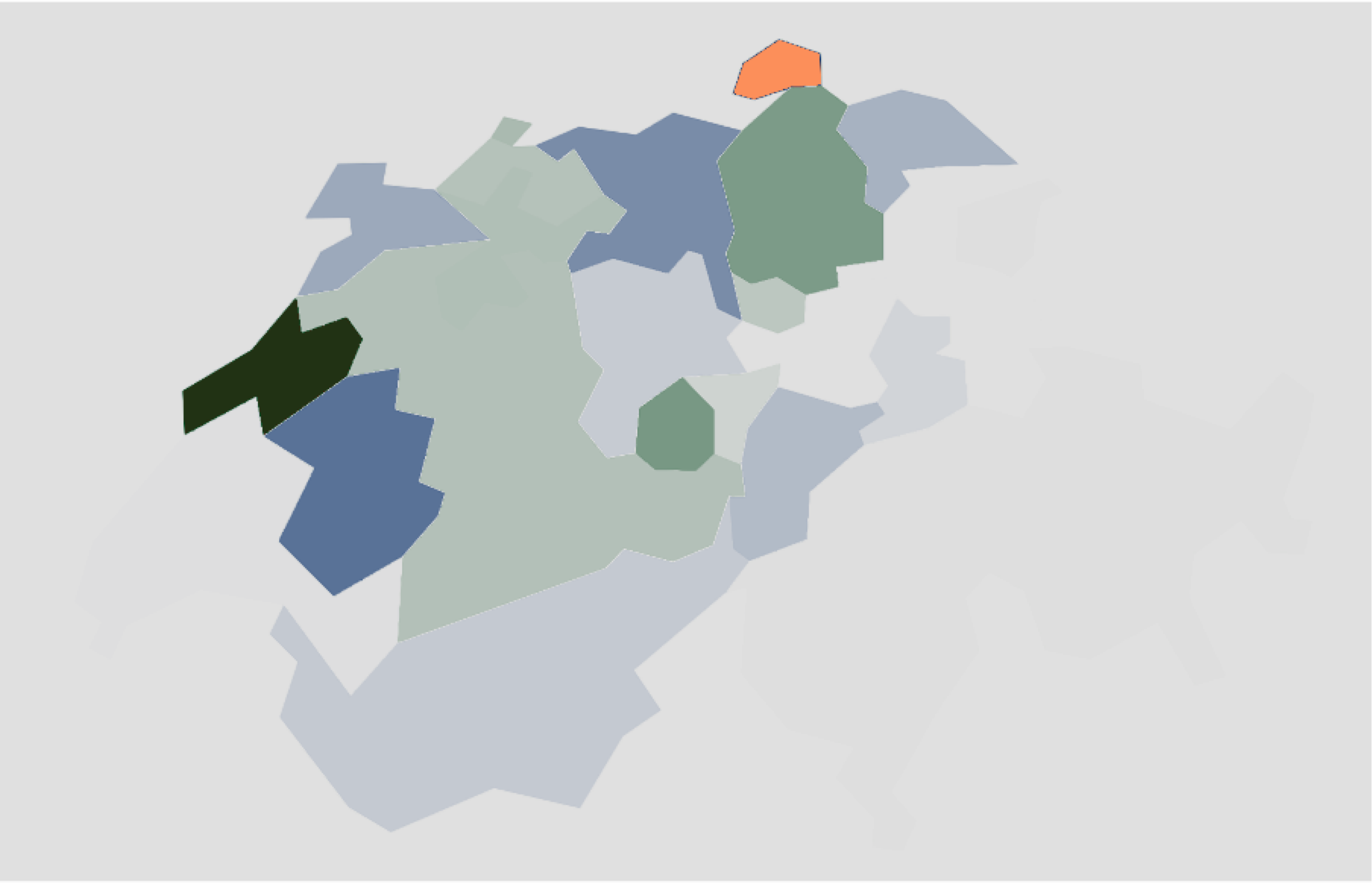}
\caption{$\varphi_{15}$ with $\lambda_{15} \approx 217.59$}
\end{subfigure}
\caption{
Polygonal approximation of the map of Switzerland $u_{\delta}$ and its $26$ cantons (top left), together with three eigenfunctions $\varphi_{k}$, $k=2,5,15$, of the operator $L_{\varepsilon}[u_{\delta}]$.}
\label{fig:NumEx.Switzerland}
\end{figure}

Here we consider the polygonal approximation of the map of Switzerland with its $K = 26$ cantons, shown in frame (a) of Figure \ref{fig:NumEx.Switzerland}, where each canton admits a constant value.
The data of the map are given on a discrete rectangular pixel based $1563\,\mathrm{px} \times 1002\,\mathrm{px}$ grid with grid-size $\delta = 1 \,\mathrm{px}$.
We interpolate the data to obtain $u_{\delta} \in \mathcal{V}^{\delta}_{0}$, and compute the first $K=26$ eigenfunctions, $\vphi_1,\ldots,\vphi_K$ of $L_\veps[u_\del]$;
frames (c), (d) and (e) of Figure \ref{fig:NumEx.Switzerland} show three of the eigenfunctions.

Although a single eigenfunction does not necessarily correspond to any particular canton, we may still represent each canton in $\Phi_{26}^{\veps,\del}=\Span\{\vphi_k\}_{k=1}^{26}$.
If $u^{\mathrm{c}}$ is the characteristic function for a canton shown in the map in Figure \ref{fig:NumEx.Switzerland}, and $u^{\mathrm{c}}_\del$ is its continuous (piecewise linear) interpolant in $\mathcal{V}^{\delta}$, we can use the AS basis $\{\varphi_{k}\}_{k=1}^{K}$ to approximate $u_{\delta}^{\mathrm{c}}$ as
\begin{align*}
	u_{\delta}^{\mathrm{c}}
		\approx \Pi_{K}^{\veps}[u_\delta]\, u_{\delta}^{\mathrm{c}}
		=\sum_{k=1}^{K} \beta_{k}\varphi_{k} ,
\end{align*}
with $K=26$.
In Figure \ref{fig:NumEx.Switzerland.DecompositionCantons} we show the approximations for the cantons of Bern, Grisons, and St.\ Gallen in $\Phi_{26}^{\veps,\del}=\Span\{\vphi_k\}_{k=1}^K$.
These reconstructions approximate very well the exact cantons in Figure~\ref{fig:NumEx.Switzerland}.

\begin{figure}[t]
\centering
\begin{subfigure}[c]{0.32\textwidth}
\centering
\includegraphics[width=0.9\linewidth]{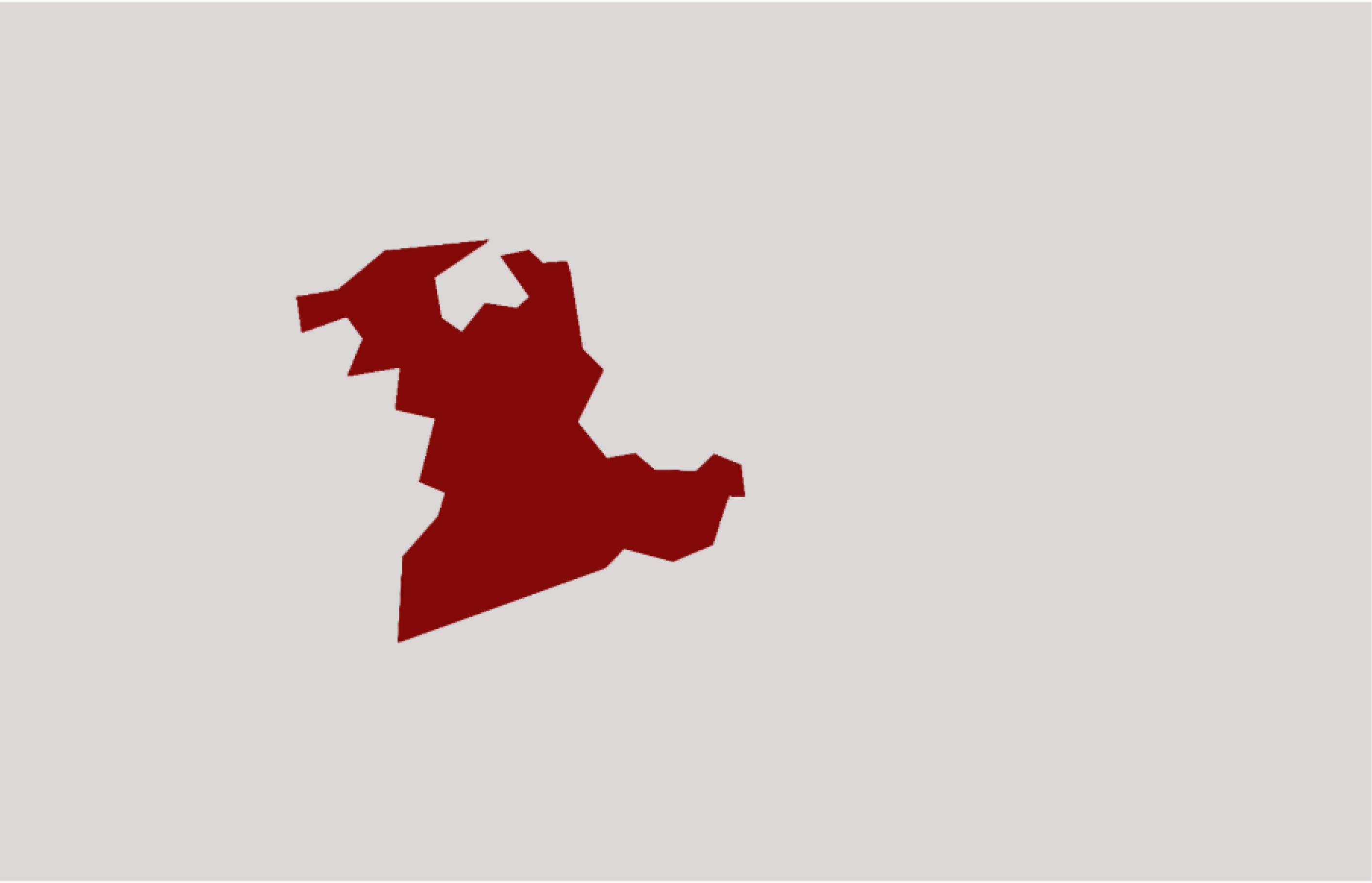}
\caption{Canton of Bern}
\end{subfigure}\hfill
\begin{subfigure}[c]{0.32\textwidth}
\centering
\includegraphics[width=0.9\linewidth]{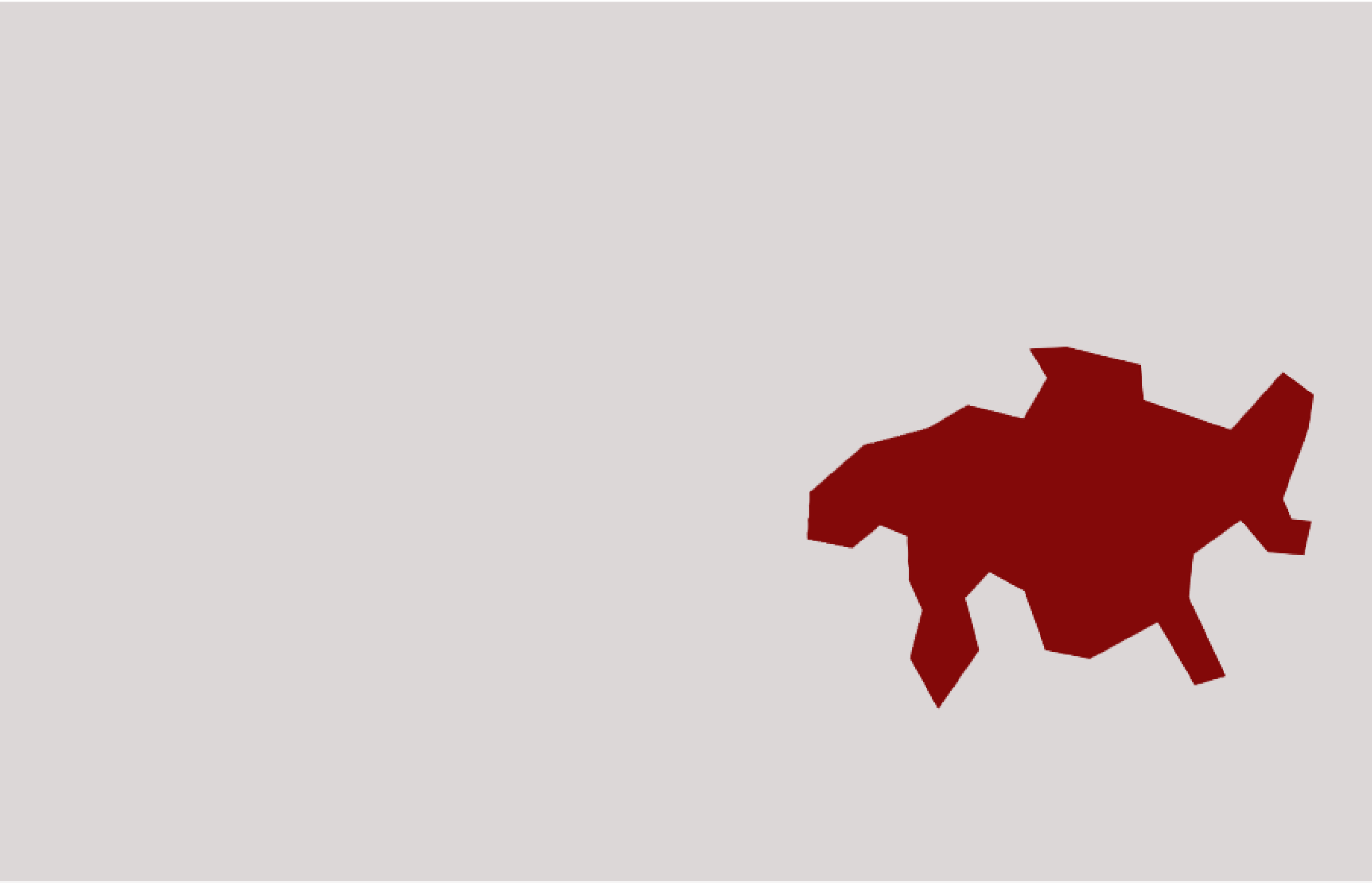}
\caption{Canton of Grisons}
\end{subfigure}\hfill
\begin{subfigure}[c]{0.32\textwidth}
\centering
\includegraphics[width=0.9\linewidth]{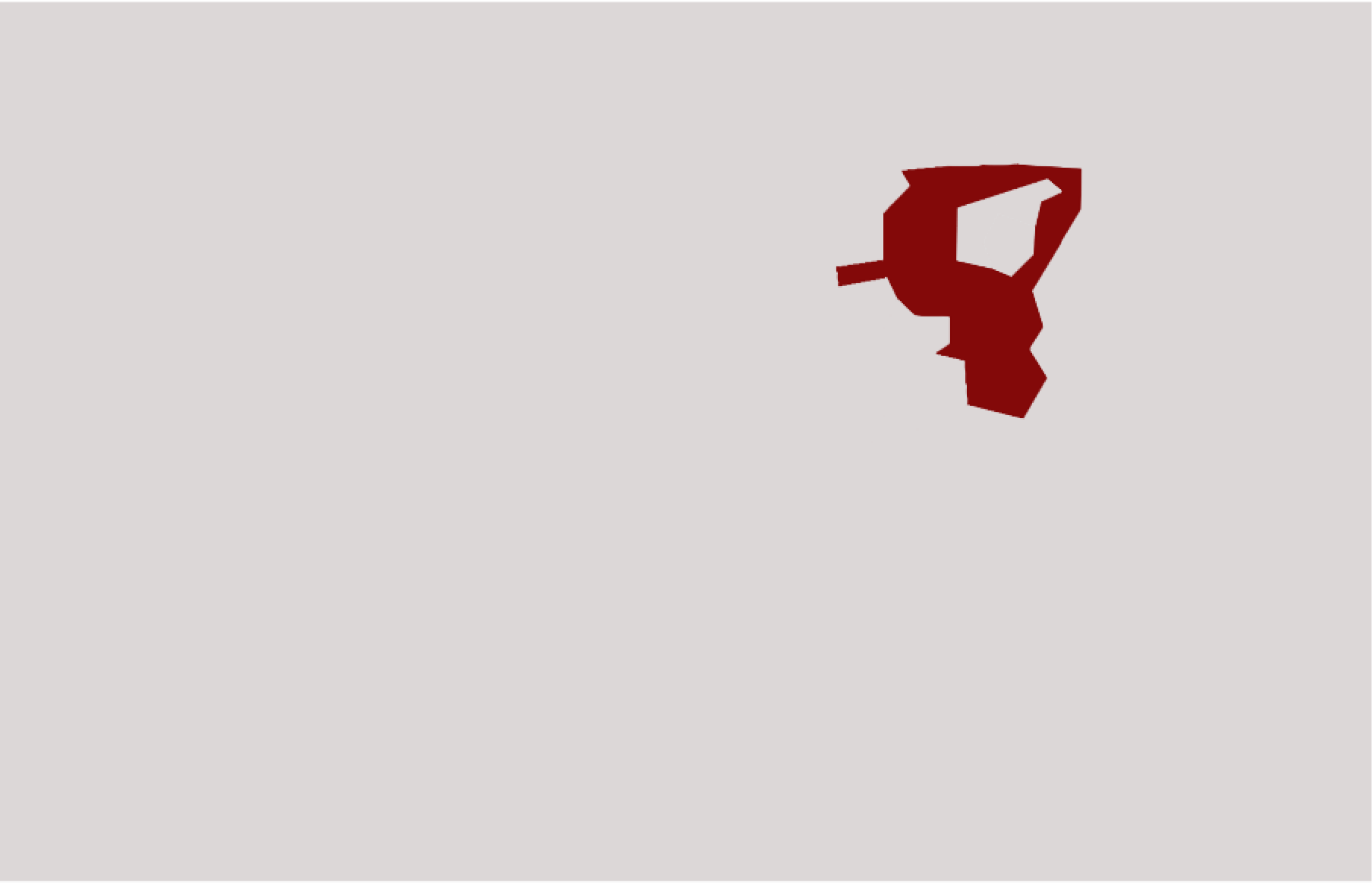}
\caption{Canton of St.\ Gallen}
\end{subfigure}
\caption{Map of Switzerland. Three cantons approximated in the truncated AS basis $\{\varphi_{k}\}_{k=1}^{K}$ with $K=26$.}
\label{fig:NumEx.Switzerland.DecompositionCantons}
\end{figure}

\begin{figure}[t]
\centering
\begin{subfigure}[c]{0.8\textwidth}
\centering
\includegraphics[width=\linewidth]{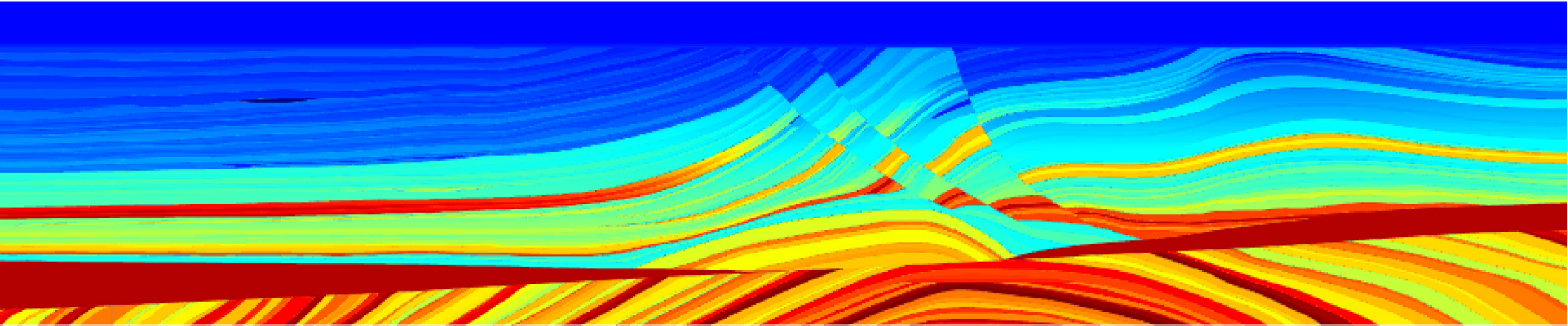}
\caption{The Marmousi model}
\end{subfigure}\\
\begin{subfigure}[c]{0.8\textwidth}
\centering
\includegraphics[width=\linewidth]{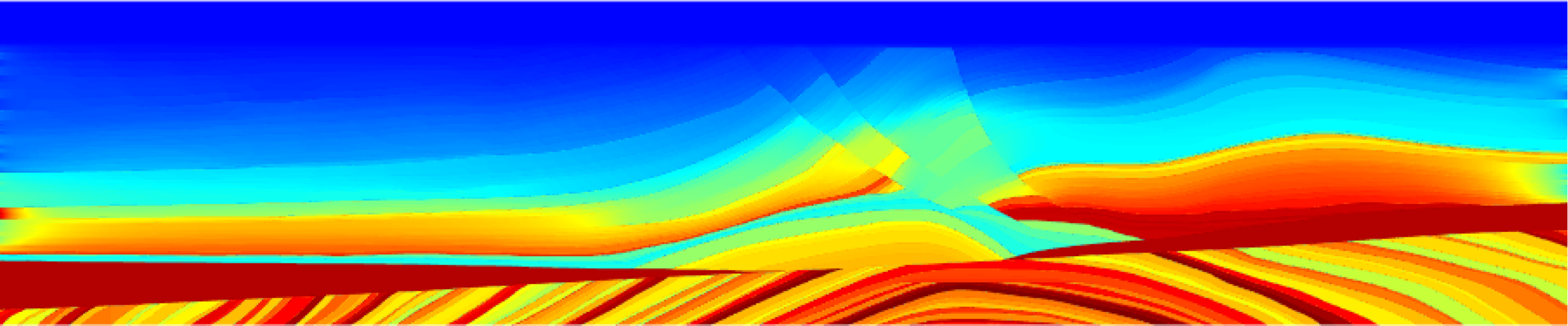}
\caption{$\varphi_{0}$ with a relative $L^{2}$ error of $12.8 \%$}
\end{subfigure}\\
\begin{subfigure}[c]{0.8\textwidth}
\centering
\includegraphics[width=\linewidth]{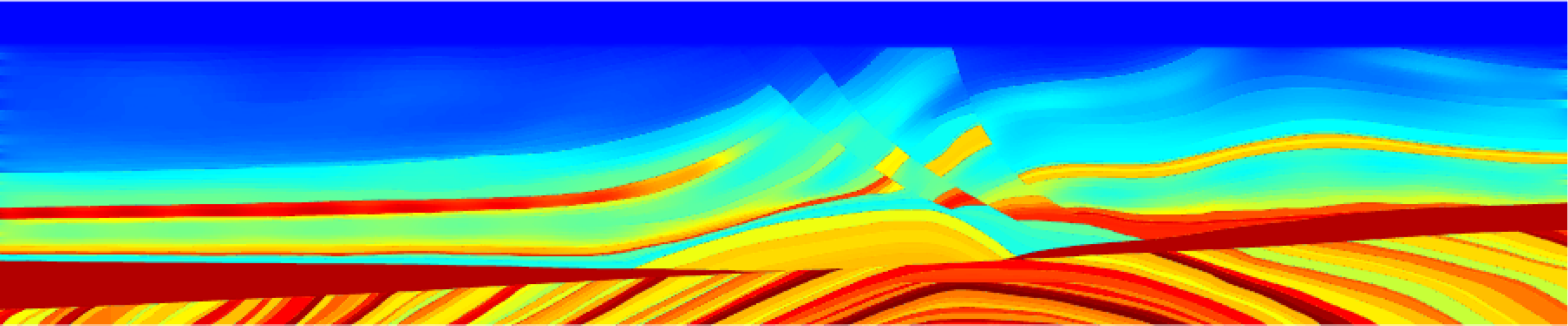}
\caption{$Q_{100}^{\veps}[u_\delta](u_{\delta})$ with a relative $L^{2}$ error of $3.8 \%$}
\end{subfigure}
\caption{The original Marmousi model with its background $\varphi_{0}$ and AS decomposition with $100$ eigenfunctions.}
\label{fig:NumEx.Marmousi2}
\end{figure}

\subsection{The Marmousi model}

As a last example we consider the subsurface model of the P-wave velocity of the AGL elastic Marmousi model shown in Figure \ref{fig:NumEx.Marmousi2}, see \cite{martin2006marmousi2,MarmousiWebsite}.
The data of the model is given as nodal values on a discrete rectangular mesh representing a $17\, \mathrm{km} \times 3.5\, \mathrm{km}$ area.
We interpolate the data in $\mathcal{V}^{\delta}$ with $\delta = 2.5 \,\mathrm{m}$ to obtain $u_{\delta}$.
Next, we compute the background $\varphi_{0} \in \mathcal{V}^{\delta}$ as well as the first $100$ eigenfunctions of the operator $L_{\varepsilon}[u_{\delta}]$.

Remarkably, the background $\varphi_{0}$ already yields a good approximation of the model with a relative error of 
\begin{align*}
	\frac{\|u_{\delta}-Q_{0}^{\varepsilon}[u_{\delta}](u_{\delta})\|_{L^2(\Om)}}{\|u_{\delta}\|_{L^2(\Om)}} = \frac{\|u_{\delta}-\varphi_{0}\|_{L^2(\Om)}}{\|u_{\delta}\|_{L^2(\Om)}} \approx 12.8 \%,
\end{align*}
probably because many of the internal layers in the model reach the boundary and thus can be recovered by $\varphi_{0}$.
In contrast, the eigenfunctions $\varphi_{k}$ ($k\ge1$) account for variations of the medium in the interior of the domain.
Here, the additional contribution of the first $K=100$ eigenfunctions to the approximation further reduces the relative error to
$\|u_{\delta}-Q_{K}^{\varepsilon}[u_{\delta}](u_{\delta})\|_{L^2(\Om)}/\|u_{\delta}\|_{L^2(\Om)} \approx 3.8\%$.

\subsection{Inverse Problem}

Here we devise an iterative inversion algorithm based on AS decompositions to solve a standard linear deconvolution inverse problem which occurs in optical imaging \cite{bertero2021inverse}.
Hence, we consider the Fredholm integral equation of the first kind
\begin{equation}\label{eq:inverse.problem}
	Fu =y
\end{equation}
where $F: L^2(\Omega) \to L^2(\Omega)$ is the convolution operator
\begin{equation}
	Fu(x) = \int_\Omega g(x - x^\prime) u(x^\prime) \; dx^\prime ,
\end{equation}
with $\Omega = (0,1)^2$ and $g$ the Gaussian kernel
\begin{equation}
	g(x) = \frac{1}{2 \pi \gamma^2} \, e^{-\frac{|x|^2}{2 \gamma^2}}, \qquad \gamma = \frac{1}{32}.
\end{equation}
Given the noisy observation $y^\eta$ of $y^\dagger=Fu^\dagger$ where $\|y^\dagger - y^\eta\|_{L^2(\Omega)} \leq \eta$, we wish to reconstruct the true medium/image $u^\dagger$.
In doing so, we assume the FE interpolant $u^\dagger_h$ of $u^\dagger$ is known on the boundary $\p\Om$.

First, we formulate the problem as the minimization of
\begin{equation}
	\cJ(u) = \frac{1}{2}\| Fu - y^\eta \|_{L^2(\Omega)}^2
\end{equation}
in some appropriate space.
Then, we proceed iteratively as follows:
In the $m$-th iteration, given the previous estimate $u^{(m-1)}$ of $u^\dagger$, we compute $\varphi_k^{(m)}$ $(k = 0, \ldots, K)$ by solving 
\begin{align}\label{eq:asb.iteratively}
	\begin{aligned}
			L_{\varepsilon}[u^{(m-1)}] \varphi_0^{(m)} &= 0 \; &&\text{in}\;\Omega, \qquad \varphi_0^{(m)} = u^\dagger_h \; &&\text{on}\;\partial\Omega, \\
			L_{\varepsilon}[u^{(m-1)}] \varphi_k^{(m)} &= \lambda_k \varphi_k^{(m)} \; &&\text{in}\; \Omega, \qquad \varphi_k^{(m)} = 0 \; &&\text{on}\;\partial\Omega .
	\end{aligned}
\end{align}
Next, we compute the current estimate, $u^{(m)}$, by solving the least squares (LS) problem:
\begin{equation}\label{eq:iter_u^(m)}
	u^{(m)} = \arg\min\Bigcurb{\cJ(u):\ u \in \varphi_0^{(m)} + \Phi^{(m)}_K} ,
	\qquad
	\Phi_K^{(m)} = \Span\Bigcurb{\vphi_k^{(m)}}_{k=1}^K.
\end{equation}
Since the dimension $K$ of this LS problem is small, we may solve it directly.
The iteration stops when the discrepancy principle,
\begin{align}\label{eq:discrepancy.principle}
	\| F u^{(m)} - y^\eta \|_{L^2(\Omega)} \leq \tau \eta ,
\end{align}
is satisfied for some fixed $\tau \geq 1$; then $u^\mathrm{ASI}$ denotes the estimate $u^{(m)}$ at the final iteration.

In practice, we solve the problem numerically with the FE method.
As in the previous numerical examples, we use standard $\mathcal{P}^1$-FE on a uniform triangular mesh with mesh size $h = \delta = 0.00625$ 
to discretize the deconvolution problem \eqref{eq:inverse.problem} and the AS problems \eqref{eq:asb.iteratively}, \eqref{eq:iter_u^(m)}.
The discretization of \eqref{eq:inverse.problem} yields a linear system of equations
\begin{equation}\label{eq:inverse.problem.discretized}
	F_h \vec u_h = \vec y_h ,
\end{equation}
which is ill-conditioned as the smallest singular value of $F_h$ is $\sigma_\mathrm{min} \approx 10^{-17}$.
For the test below we let the exact medium/image $u^\dagger$ be given by Figure \ref{subfig:NumEx.ComplexShape.Exact}, the noise $\eta \approx 4\%$ and set $u^{(0)} = y^\eta$ and $K = 100$.
Thus the dimension $K = 100$ of the LS problem in \eqref{eq:iter_u^(m)} is indeed small compared to the dimension $N\approx26'000$ of the FE space and we can solve it directly.

For comparison, we also solve the deconvolution problem with two other standard approaches.
In the first, we solve \eqref{eq:inverse.problem.discretized} directly using the LU-decomposition to obtain the solution $u^\mathrm{LU}$.
In the second, we apply the truncated singular value decomposition (TSVD), i.e., we regularize \eqref{eq:inverse.problem.discretized} by replacing all singular values of $F_h$ smaller than $\sqrt{\eta}$ by zeros; see \cite[Chapter 8]{bertero2021inverse} or \cite[Chapter 1]{vogel2002inverse} for more details; that solution is denoted by $u^\mathrm{TSVD}$.

Table \ref{tab:data.ip} provides the relative $L^2$ error 
\begin{align}\label{eq:parameters.ip}
	e_\mathrm{r} = \frac{\|u - u^\dagger_h\|_{L^2(\Omega)}}{\|u^\dagger_h\|_{L^2(\Omega)}}
	\qquad
	\text{and the ratio}
	\qquad \tau = \frac{1}{\eta} \left\| F u - y^\eta \right\|_{L^2(\Omega)}
\end{align}
for the discrepancy principle \eqref{eq:discrepancy.principle}, for the three reconstructions $u^\mathrm{ASI}$, $u^\mathrm{TSVD}$, and $u^\mathrm{LU}$ shown in Figure \ref{subfig:IP.solutions}.
As expected, using the LU-decomposition for solving the inverse problem produces the solution with the largest relative $L^2$ error, despite a rather small misfit.
The TSVD solution $u^\mathrm{TSVD}$ yields an acceptable reconstruction with a relative error less than $20 \;\%$ and $\tau \approx 1$.
Still, as shown in Figure \ref{subfig:IP.tsvd}, the discontinuities are not well represented.
In contrast, the ASI solution in Figure \ref{subfig:IP.asi} has the smallest relative $L^2$ error while discontinuities in the medium are better detected.
Clearly, there are many available image reconstruction techniques more sophisticated than TSVD \cite{hansen2010discrete,vogel2002inverse,bertero2021inverse}, which is only used here for the purpose of illustration.
\begin{table}[h]
	\begin{center}
		\begin{tabular}{c||c|c|c|}
			~& ASI & TSVD & LU \\\hline
			$e_\mathrm{r}$ & $15.1  \; \%$ & $18.9  \; \%$ & $3.3 \cdot 10^{14}  \; \%$ \\
			$\tau$ & $1.06$ & $1.07$ & $0.001$
		\end{tabular}
	\end{center}
	\caption{Inverse problem. The relative error $e_\mathrm{r}$ and the relative misfit $\tau$ given by \eqref{eq:parameters.ip} for the three reconstructions $u^\mathrm{ASI}$, $u^\mathrm{TSVD}$ and $u^\mathrm{LU}$ for the inverse problem \eqref{eq:inverse.problem.discretized} with $4\;\%$ added noise.}
	\label{tab:data.ip}
\end{table}
\begin{figure}[t]
	\centering
	\begin{subfigure}[c]{0.32\textwidth}
		\centering
		\includegraphics[height=0.84\linewidth]{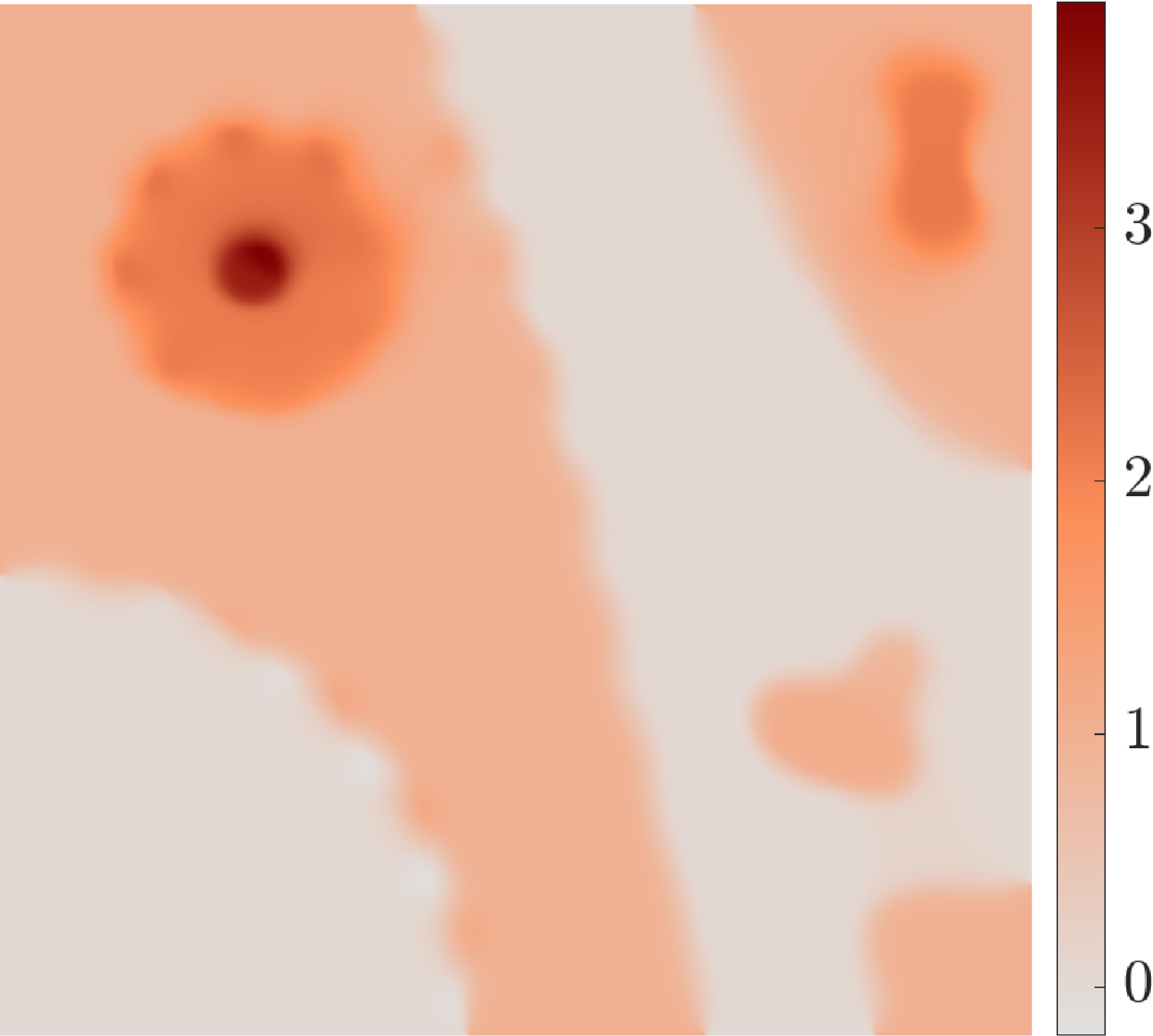}
		\caption{$u^\mathrm{ASI}$}
		\label{subfig:IP.asi}
	\end{subfigure}
	\hfill
	\begin{subfigure}[c]{0.32\textwidth}
		\centering
		\includegraphics[height=0.84\linewidth]{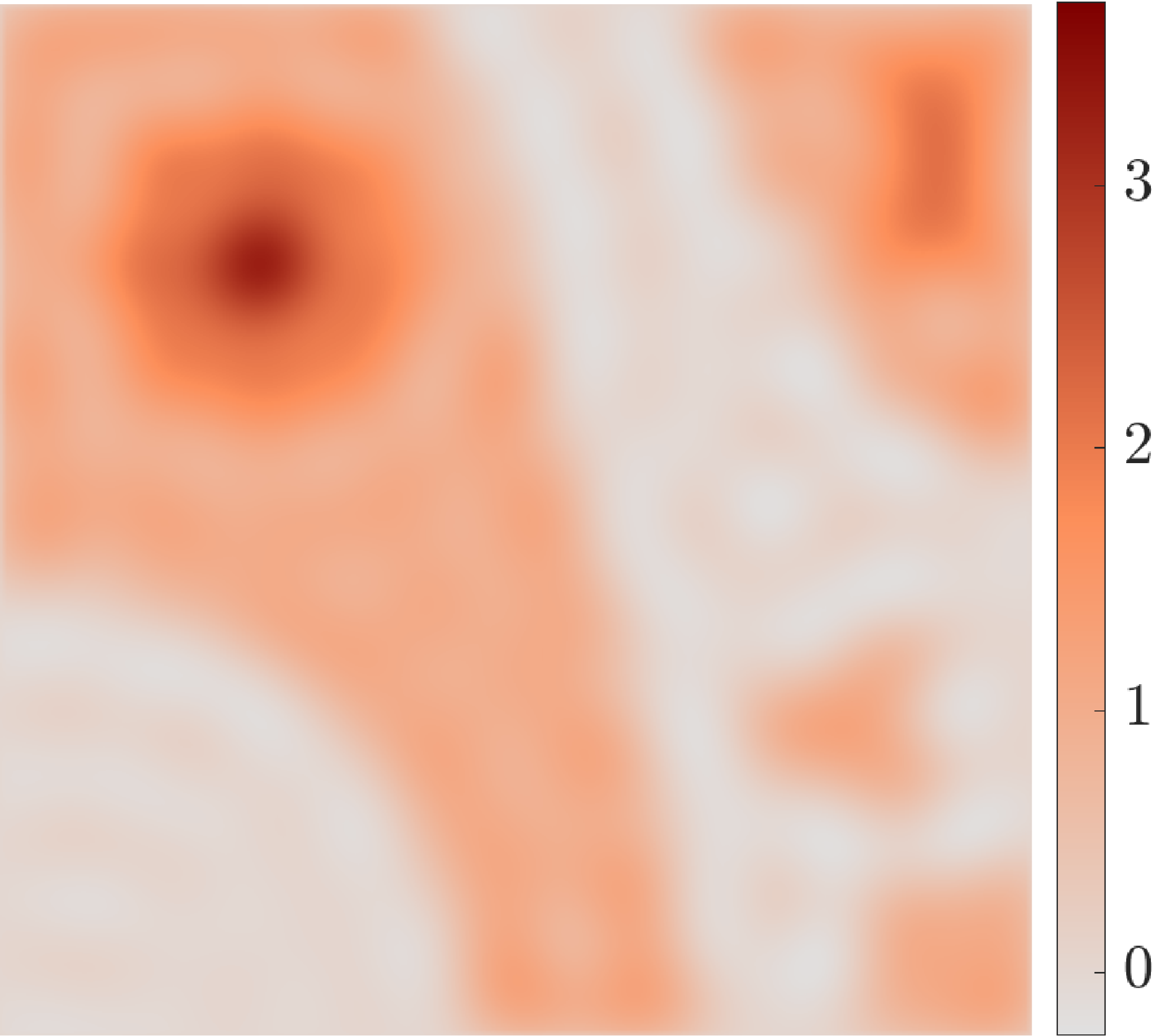}
		\caption{$u^\mathrm{TSVD}$}
		\label{subfig:IP.tsvd}
	\end{subfigure}
	\hfill
	\begin{subfigure}[c]{0.32\textwidth}
		\centering
		\includegraphics[height=0.84\linewidth]{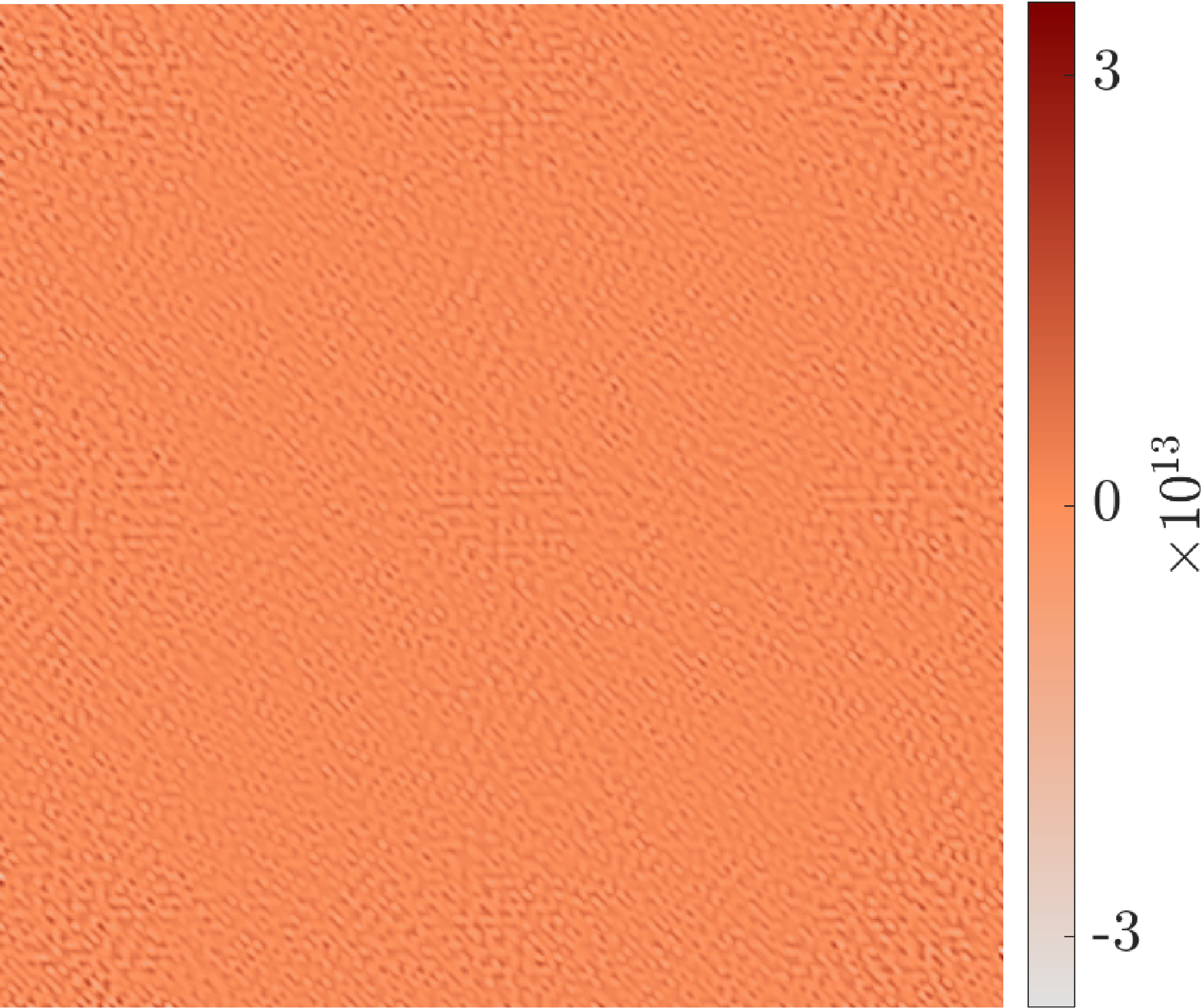}
		\caption{$u^\mathrm{LU}$}
		\label{subfig:IP.lu}
	\end{subfigure}
\caption{Inverse problem. Solutions obtained by the three different methods to solve the inverse problem \eqref{eq:inverse.problem.discretized} with $4\;\%$ added noise.}
\label{subfig:IP.solutions}
\end{figure}

%
%

\section*{Declarations}

\paragraph{Acknowledgments.}

We thank Giovanni Alberti and Gianluca Crippa for their useful comments and suggestions regarding Appendix \ref{sec:level_sets}.

\paragraph{Funding.}
The authors declare that no funds, grants, or other support were received during the preparation of this manuscript.

\paragraph{Data Availability}
The numerical results used to confirm the theory and illlustrate its usefulness were generated with 
Matlab codes not intended for public distribution. However, the corresponding author would certainly
make them available upon reasonable request.

\appendix

\section{Level sets of distance functions}\label{sec:level_sets}

In the following, for $p_1, p_2\in\R^d$, $\dist(p_1,p_2)$ denotes the Euclidean distance
\[
	\dist(p_1,p_2)=|p_1-p_2|
\]
between $p_1$ and $p_2$.
Here we prove the following theorem:
\begin{theorem}\label{thm:dist_set}
If $A\sbt\R^d$ is a $\Lam$-Lipschitz domain with bounded boundary, and $\del>0$ sufficiently small, then $A_\del$ given by
\begin{equation}
	A_\del =\curb{x\in A :\ \dist(x,\p A) > \del }
\end{equation}
is also $\Lam$-Lipschitz.
\end{theorem}

We say that a domain $A\sbt\R^d$ with bounded boundary $\p A$ is $\Lam$-Lipschitz, if near its boundary it locally coincides with the epigraph of a $\Lam$-Lipschitz function \cite{10.2307/2161378}.
As a preliminary result we first show in Theorem \ref{thm:dist_graph} of Section \ref{sec:Lip_graph} a similar result for the epigraph of a Lipschitz function.

\subsection{Distance functions for Lipschitz graphs}\label{sec:Lip_graph}

Let $\hat\cB\sbt\R^{d-1}$ be a ball of radius $R$, $f:\hat\cB\to\R$ $\Lam$-Lipschitz, $\what F$ the graph of $f$ in $\hat\cB$, $\cB\sbt\hat\cB$ a ball of radius $r<R$ concentric with $\hat \cB$, and
\begin{equation}\label{eq:G_del_set}
	G(\del) =\curb{ p=(x, y) :\ x\in\cB,\ y>f(x), \ \dist(p,\what F) =\del } .
\end{equation}
The setup is illustrated in the left frame of Figure \ref{fig:f_B_prop:d_sphere}.
Here we show that $G(\del)$ is the graph of a $\Lam$-Lipschitz function $g:\cB\to\R$.
\begin{figure}[t]
\centering
\includegraphics[height=4cm]{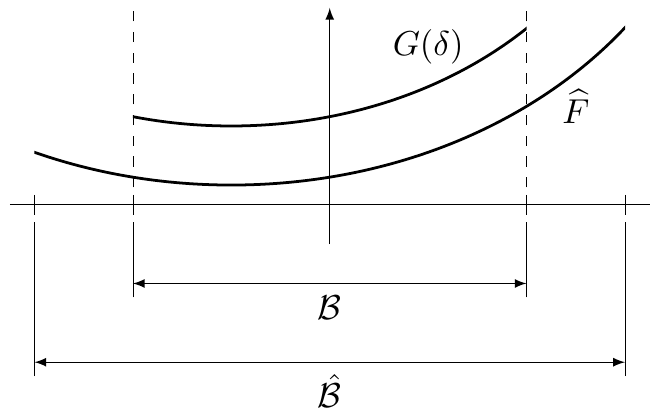}
\hfill
\includegraphics[height=5cm]{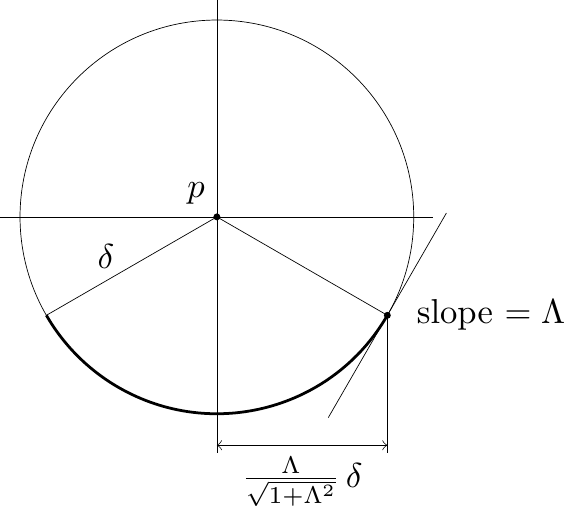}
\caption{Left: the graph $\what F$ of $f$ in $\hat\cB$ and $G(\del)$;
right: illustration of the setup of Proposition \ref{prop:d_sphere} in the plane.}
\label{fig:f_B_prop:d_sphere}
\end{figure}

For $p\in\R^d$, we let $\cC_p$ denote the open (two-sided) infinite cone,
\[
	\cC_p = p+\cC_0 ,
	\qquad
	\cC_0 = \bigcurb{(x,y)\in\R^{d-1}\times\R :\ |y| > \Lam |x|} .
\]
We shall use that a function $g:\cB\to\R$ is $\Lam$-Lipschitz if and only if for every point $p$ in its graph, $\operatorname{graph}(g)$, we have $\cC_p\cap\operatorname{graph}(g)=\emptyset$.

First we show that for every $x\in\cB$ and $y>f(x)$ sufficiently large, the distance of $(x,y)$ to $\what F$ is greater than $\del$.

\begin{proposition}\label{prop:d2graph_1}
If $x\in\cB$ and $y>f(x)+\Lam_0\del$, with $\Lam_0=\sqrt{1+\Lam^2}$, then
\begin{equation}
	\dist((x,y),\what F) > \del ;
\end{equation}
especially $(x,y)\notin G(\del)$.
\end{proposition}

\begin{proof}
Fix $x\in\cB$ and $h>h_0=\Lam_0\del$.
We show that $p=(x,f(x)+h)$ satisfies $\dist(p,\what F)>\del$.
If $\Lam=0$, then $f$ is constant and the conclusion is clear.
Suppose $\Lam>0$, and let $\hat x\in\hat\cB$, $\hat p=(\hat x,f(\hat x))$, and $\tau=|f(x) - f(\hat x)|/\Lam$.
Then,
\begin{equation}
\begin{aligned}
	\dist(p,\hat p)^2 & = | x-\hat x|^2 + (f(x)+h - f(\hat x))^2 \\
		& \ge \frac{1+\Lam^2}{\Lam^2}|f(x) - f(\hat x)|^2 - 2h|f(x) - f(\hat x)| +h^2 \\
		& =\rb{1+\Lam^2}\tau^2 -2h\Lam \tau+h^2 =:\psi(\tau) .
\end{aligned}
\end{equation}
Since the minimum of $\psi$ is achieved in
\begin{equation}
	\tau_* =\frac{h\Lam}{1+\Lam^2} ,
\end{equation}
we have
\begin{equation}
	\dist(p,\hat p)^2 \ge \psi(\tau_*) = h^2\rb{1-\frac{\Lam^2}{1+\Lam^2}}=\frac{h^2}{1+\Lam^2}
		> \frac{(1+\Lam^2)\del^2}{1+\Lam^2} = \del^2
\end{equation}
which yields the conclusion.
\end{proof}

As a result we have that for every $x\in\cB$, there exists $y>f(x)$ such that $(x,y)\in G(\del)$, and, in particular, we obtain an estimate of $y-f(x)$.

\begin{proposition}\label{prop:d2graph_ex}
For each $x\in\cB$, there exists $t\in[\del,\Lam_0\del]$, with $\Lam_0=\sqrt{1+\Lam^2}$, such that
\[
	(x,f(x)+t)\in G(\del) .
\]
\end{proposition}

\begin{proof}
Let $\rho(t)=\dist((x,f(x)+t),\what F)$, $p=(x,f(x))$, and $h_0=\Lam_0\del$.
Since
\[
	\dist(p,(x,f(x)+\del))=\del ,
\]
we have $\rho(\del)\le\del$.
In addition, by Proposition \ref{prop:d2graph_1}, $\rho(h)>\del$, for $h>h_0$.
Since $\rho$ is continuous, there exists $t\in[\del,h)$ such that
\begin{equation}
	\dist((x,f(x)+t),\what F) =\rho(t)=\del
\end{equation}
Because the above is true of every $h>h_0$, we have the conclusion.
\end{proof}

The following proposition puts restrictions on $f$ in a neighborhood of a point $x\in\cB$, provided $p=(x,y)\in G(\del)$.
The idea of the proof is illustrated in the right frame of Figure~\ref{fig:f_B_prop:d_sphere}.
%
%
%
\begin{proposition}\label{prop:d_sphere}
Let $p=(x,y)\in G(\del)$, and $\hat x\in\hat\cB$.
\begin{enumerate}
\item
If $|\hat x-x| \le \del$, then
\begin{equation}\label{eq:f_x_del}
	f(\hat x) \le y-\sqrt{\del^2-|\hat x -x|^2}\, .
\end{equation}


\item
If $|\hat x-x| > \del\Lam/\Lam_0$, where $\Lam_0=\sqrt{\Lam^2+1}$, then
\begin{equation}\label{eq:f_far}
	f(\hat x) \le \Lam |\hat x-x| + y-\Lam_0\del .
\end{equation}
\end{enumerate}
\end{proposition}

\begin{proof}
1.\ \
Assertion 1 is true because $f$ is continuous, $f(x)<y$, and $\dist(p,\what F)=\del$.


2.\ \
We show \eqref{eq:f_far} by contradiction.
Suppose, that $\hat x$ does not satisfy \eqref{eq:f_far}.
By considering the plane containing the points $p=(x,y)$, $(x,f(x))$ and $(\hat x, f(\hat x))$ (note that they are indeed not collinear), we reduce the problem to the 2-dimensional case, where we may assume without loss of generality that $\hat x > x$.
For $x_1 = x+\del\Lam/\Lam_0$ we have $|x_1-x|=x_1-x =\del\Lam/\Lam_0<\del$ and therefore by \eqref{eq:f_x_del} and $\Lam_0^2=\Lam^2+1$,
\begin{equation}
	f(x_1) \le y-\sqrt{\del^2-(x-x_1)^2} = y -\frac{\del}{\Lam_0} .
\end{equation}
Thus, using $\Lam_0^2=\Lam^2+1$ and $x_1=x+\del\Lam/\Lam_0$ we obtain
\begin{equation}
\begin{aligned}
	f(\hat x)-f(x_1) & > \bigrb{\Lam(\hat x-x)+y-\Lam_0\del}-\rb{y-\frac{\del}{\Lam_0}} \\
		& = \Lam \rb{\hat x-x} -\del\, \frac{\Lam_0^2-1}{\Lam_0}
			= \Lam\rb{\hat x-x_1} ,
\end{aligned}
\end{equation}
which contradicts $f$ being $\Lam$-Lipschitz.
\end{proof}

\begin{figure}[t]
\centering
\includegraphics[height=4.0cm]{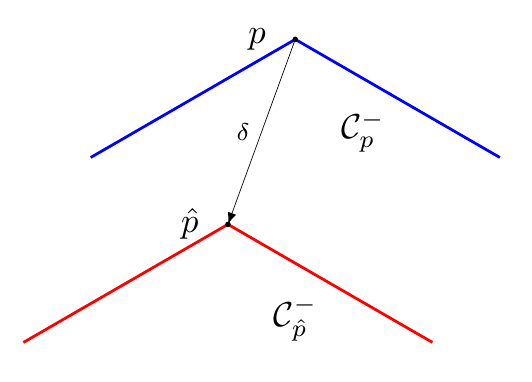}
\hfil
\includegraphics[height=4.0cm]{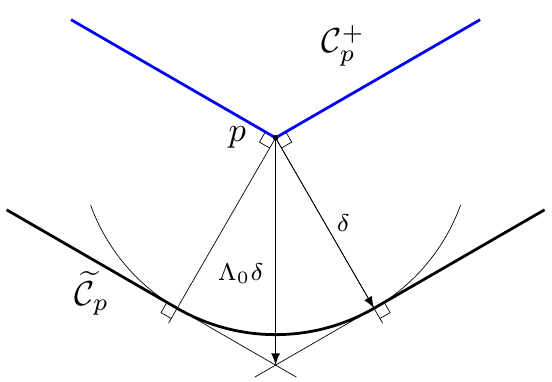}
\caption{Illustrations for the proof of Lemma \ref{lem:G_cone}; here $\cC_p^+$ and $\cC_p^-$ denote the upper and lower halves of the cone $\cC_p$, respectively.}
\label{fig:G_cone}
\end{figure}



As a result of Proposition \ref{prop:d_sphere} we have that if $p=(x,y)\in G(\del)$, then the graph $\what F$ of $f$ in $\hat\cB$ is a subset of
\begin{equation}\label{eq:smooth_cone}
\begin{aligned}
	\wt\cC_p
		& = \curb{(\hat x,\hat y)\in\hat\cB\times\R :\ |\hat x-x| \le \del \tfrac{\Lam}{\Lam_0},
				\ \hat y \le y -\sqrt{\del^2-|\hat x- x|^2}} \\
			&\quad \bigcup \curb{(\hat x,\hat y)\in\hat\cB\times\R :
				\ |\hat x-x| > \del \tfrac{\Lam}{\Lam_0}, \ \hat y \le \Lam|\hat x - x| +y-\Lam_0 \del} \\
		&=\curb{r\in\hat\cB\times\R :\ \dist\!\bigrb{r,\cC_p^+} \ge \del} ,
\end{aligned}
\end{equation}
illustrated in the right frame of Figure \ref{fig:G_cone}, where $\cC_p^+$ denotes the upper half of the cone $\cC_p$.
The second equality in \eqref{eq:smooth_cone} can be verified by using the radial symmetry of the sets about the vertical line $\hat x = x$ and then reducing the problem to the 2-dimensional case, similarly to the proof of Assertion 2 of Proposition \ref{prop:d_sphere}.
We use this observation to get the following.

\begin{lemma}\label{lem:G_cone}
If $p\in G(\del)$, then $\cC_p\cap G(\del)=\emptyset$.
\end{lemma}

\begin{proof}
We show separately the two propositions $\cC_p^\pm\cap G(\del)=\emptyset$, for the upper and lower parts $\cC_p^\pm$ of the cone $\cC_p$.

1.\ \
Consider the lower part $\cC_p^-$ of the cone $\cC_p$.
Since $p=(x,y)\in G(\del)$, there exists $\hat p=(\hat x,f(\hat x))\in\what F$ such that $\dist(p,\hat p)=\del$ and $y>f(\hat x)$ (by Proposition \ref{prop:d_sphere}).
Since $f$ is $\Lam$-Lipschitz and $\what F$ is the graph of $f$ in $\hat\cB$, we have $\cC_{\hat p}^-\cap \what F=\emptyset$, and in particular $\cC_{\hat p}^-$ lies below $\what F$.
However, the lower part $\cC_{p}^-$ of $\cC_{p}$ is given by $\cC_{p}^-=p-\hat p+\cC_{\hat p}^-$.
Since the length of $p-\hat p$ is $\del$, we have that every point $r$ in the interior of $\cC_{p}^-$ is at a distance of $\del$ from a point in the interior of $\cC_{\hat p}^-$, which yields $\dist(r,\what F)<\del$ and thus $r\notin G(\del)$.
Since $r\in\cC_p^-$ is arbitrary, we get $\cC_p^-\cap G(\del)=\emptyset$.

2.\ \
Now consider the upper part $\cC_p^+$ of the cone $\cC_p$, and let $r\in\cC_p^+$.
In this case, illustrated in the right frame of Figure~\ref{fig:G_cone}, it is clear that $\dist(r,\what F)>\del$, since $r\in\cC_p^+$ and $\what F\subset\wt\cC_p$ given by \eqref{eq:smooth_cone}.
\end{proof}

\begin{theorem}\label{thm:dist_graph}
For $\del>0$, the set $G(\del)$ is the graph of a $\Lam$-Lipschitz function $g:\cB\to\R$.
\end{theorem}

\begin{proof}
By Proposition \ref{prop:d2graph_ex} and Lemma \ref{lem:G_cone}, for each $x\in\cB$, there exists a unique $y$ such that $p=(x,y)\in G(\del)$.
This defines a function $g:\cB\to\R$ such that $G(\del)$ is its graph.
Moreover, by Lemma \ref{lem:G_cone}, for each $p\in G(\del)$, $G(\del)\cap\cC_p=\emptyset$, which yields that $g$ is $\Lam$-Lipschitz.
\end{proof}

\subsection{Distance functions for Lipschitz domains}


For $r>0$, let $B(r)$ denote the open ball in $\R^{d-1}$ of radius $r$ centered at the origin.

\begin{proof}[Proof of Theorem \ref{thm:dist_set}]
Since $A$ is $\Lam$-Lipschitz and $\p A$ is bounded, there is a finite set of pairs $(V_n,f_n)$, with $n=1,\ldots,N$, of bounded open right cylinders $V_n$ and functions $f_n$ of $d-1$ variables satisfying the following:
\begin{enumerate}
\item
$\{V_n\}_n$ is a finite open cover of $\p A$,

\item
the bases of $V_n$ are at a positive distance from $\p A$,

\item
$f_n$ is $\Lam$-Lipschitz, and $f_n(0)=0$,

\item
for each $n$, there exists a Cartesian coordinate system $(\xi,\eta)$, with $\xi\in\R^{d-1}$ and $\eta\in\R$, for which
\begin{equation}
	V_n = B(r_n)\times (-b_n,b_n) ,
\end{equation}
for some $r_n,b_n>0$, and
\begin{equation}
	A\cap \hat V_n =\curb{(\xi,\eta) :\ \xi\in B(2r_n),\ f_n(\xi)<\eta<b_n } ,
	\qquad
	\hat V_n=B(2r_n)\times (-b_n,b_n) .
\end{equation}
\end{enumerate}

Choose $\del>0$ such that
\begin{equation}
	\p A_\del \sbt \bigcup_{n=1}^N V_n
\end{equation}
and for all $n$, with respect to the $n$-th coordinate system $(\xi,\eta)$, the part of the boundary of $A_\del$ lying in $V_n$ coincides with the set $G_n(\del)=G(\del)$ given by \eqref{eq:G_del_set} with $f=f_n$, $\cB=B(r_n)$, $\hat\cB=B(2r_n)$.
By Theorem \ref{thm:dist_graph}, $G_n(\del)$ is the graph of a $\Lam$-Lipschitz function $g_n:B(r_n)\to\R$.
Thus, the boundary $\p A_\del$ of $A_\del$ is covered by a finite collection of open sets $V_n$, such that for each $n$ there exists a coordinate system $(\xi,\eta)$ in which $\p A_\del\cap V_n$ coincides with the graph of the $\Lam$-Lipschitz function $g_n$ and $A_\del\cap V_n$ lies above $g_n$, which yields the conclusion.
\end{proof}

\section{Estimates in thin sets}\label{sec:thin_domains}

We show the following theorem.

\begin{theorem}\label{thm:gen_est_U}
If $A\sbt\Om$ is a $\Lam$-Lipschitz domain, then there exists a constant $C>0$, such that for every sufficiently small $\del>0$ and every $v\in H^1(\Om)$,
\begin{equation}
	\|v\|_{L^2(U_\del)}^2
		\le C\rb{\del^{2}\|\nabla v\|_{L^2(U_\del)}^2
			+\del\|v\|_{H^1(A_\del)}^2} ,
\end{equation}
where
\begin{equation}
	U_\del = \curb{x\in A:\ \dist(x,\p A)<\del } ,
	\qquad
	A_\del = A\setminus \ol{U_\del} .
\end{equation}
\end{theorem}



We begin by citing some results of \cite{10.2307/2161378} regarding the flattening of Lipschitz graphs.
Let $V$ be a bounded domain such that $V\sbt \cB\times \R$, with $\cB\sbt\R^{d-1}$ an open ball, and let $f:\cB\to\R$ $\Lam$-Lipschitz.
We define $Y:V \longrightarrow Y(V)$ by
\begin{equation}
	Y(x) = (\hat x,x_d-f(\hat x))
	\qquad
	x=(\hat x,x_d)\in\cB\times \R .
\end{equation}
Note that the graph of $f$ is mapped by $Y$ to the flat surface $\cB\times\{0\}$.
It is easy to verify that
\begin{equation}
	\bigg|\det \frac{\p Y}{\p x} \bigg|=1 ,
\end{equation}
and that $Y$ is invertible and
\begin{equation}
	Y^{-1}(\hat y,y_d)=(\hat y,y_d+f(\hat y)) .
\end{equation}
We define
\begin{equation}
	T:H^1(V) \longrightarrow H^1(Y(V))
	\qquad
	Tu(y)=u\bigrb{Y^{-1}(y)} .
\end{equation}
The operator $T$ is well defined \cite{10.2307/2161378}, i.e., for every $u\in H^1(V)$, $Tu\in H^1(Y(V))$.
For any summable $g:V\to\R$, by the area formula we have
\begin{equation}
	\int_V g(x)d x =\int_{Y(V)} g(Y^{-1}(y)) dy .
\end{equation}
Therefore,
\begin{equation}\label{eq:T_L2_isometry}
	\|Tu\|_{L^2(Y(V))} = \|u\|_{L^2(V)} .
\end{equation}
We also have \cite{10.2307/2161378}
\begin{equation}\label{eq:T_H1_seminorm}
	\|\nabla(Tu)\|_{L^2(Y(V))} \le C \|\nabla u\|_{L^2(V)} ,
\end{equation}
where $C$ is independent of $u$ and therefore $T$ is continuous from $H^1(V)$ to $H^1(Y(V))$.
If
\[
	\G =\curb{(\hat x,f(\hat{x})) :\ \hat{x}\in\cB }\sbt\p V
\]
then there exists $C>0$ such that for every $u\in H^1(V)$
\begin{equation}\label{eq:trace_T}
	\|Tu\|_{L^2(Y(\G))} \le \|u\|_{L^2(\G)} \le C \|Tu\|_{L^2(Y(\G))} .
\end{equation}

Next we derive Poincaré-type inequalities for functions in cylinders bounded by Lipschitz graphs.
Specifically, we are interested in the behavior of the constants of the inequalities with respect to the height of the cylinder.

\begin{lemma}\label{lem:Poi_f-cyl}
Let $f:\cB\to\R$ be $\Lam$-Lipschitz and for $h>0$ let
\[
	\cC_h=\curb{(\hat x,x_d) :\ \hat x\in \cB,\ |x_d -f(\hat x)| < h } ,
	\qquad
	\G_h = \curb{(\hat x,f(\hat x)+h) :\ \hat x\in \cB }.
\]
There exists a constant $C>0$, such that for every $h>0$, and $v\in H^1(\cC_h)$,
\begin{equation}\label{eq:robin_est}
	C \|v\|_{L^2(\cC_h)}^2 \le h^2\|\nabla v\|_{L^2(\cC_h)}^2+h\, \|v\|_{L^2(\G_h)}^2 .
\end{equation}
\end{lemma}

\begin{proof}
Fix $h>0$ and let $\cC=\cC_h$ and $\G=\G_h$.
%
The estimate for $f\equiv 0$ follows easily from standard estimates for the smallest eigenvalue $\lam$ of the problem
\begin{align}
	& -\Del u = \lam u \qquad \text{in}\quad \cC \\
	& \p_n u=-h^{-1}u \qquad \text{on}\quad \G \\
	& \p_n u=0 \qquad \text{on}\quad \p\cC\setminus\G .
\end{align}
Suppose $f$ is $\Lam$-Lipschitz.
Then $Y(\cC)=\cB\times (-h,h)$, and $Y(\G)=\cB\times \{h\}$.
Since $Y(\cC)$ is a standard right cylinder, we get
\begin{equation}
	\| T v\|_{L^2(Y(\cC))}^2 \le C\rb{h^2\|\nabla (Tv)\|_{L^2(Y(\cC))}^2+h\|T v\|_{L^2(Y(\G))}^2} .
\end{equation}
Due to \eqref{eq:T_L2_isometry}, \eqref{eq:T_H1_seminorm} and \eqref{eq:trace_T} we get the conclusion.
\end{proof}

We now can prove Theorem \ref{thm:gen_est_U}

\begin{proof}[Proof of Theorem \ref{thm:gen_est_U}]
Let $v\in H^1(\Om)$.
Fix $x\in\p A$.
Since $A$ is bounded and $\Lam$-Lipschitz, there exists a cylinder $\cC$ and a $\Lam$-Lipschitz function $f$ of $d-1$ variables such that $f(0)=0$, the bases of $\cC$ are at a positive distance from $\p A$, and there exists a Cartesian coordinate system $(\xi,\eta)$, with $\xi\in\R^{d-1}$ and $\eta\in\R$, in which
\begin{equation}
	\cC = B(r)\times (-b,b) ,
\end{equation}
for $r,b>0$, $B(r)\in\R^{d-1}$ the ball of radius $r$ centered at zero and
\begin{equation}
	A\cap\cC = \curb{ (\xi,\eta) :\ \xi\in B(r),\ f(\xi) <\eta <b } .
\end{equation}

For $\ka>0$, let $V_\ka$ denote
\begin{equation}
	V_\ka =\curb{(\xi,\eta) :\ \xi\in B(r),\ 0< \eta -f(\xi) < \ka } .
\end{equation}
Choose $\del_0>0$ such that $\ka_0=2\del_0\sqrt{1+\Lam^2}< b$, and set $\ka=\Lam\del$, for $\del<\del_0$.
Then, by Proposition \ref{prop:d2graph_1} we have
\begin{equation}
	U_\del\cap V_\ka = U_\del\cap \cC
\end{equation}
and
\begin{equation}
	\wt{V} =\curb{(\xi,\eta) :\ \xi\in B(r),\ \ka <\eta -f(\xi) < \ka_0 } \sbt A_\del .
\end{equation}
Lemma \ref{lem:Poi_f-cyl} yields
\begin{equation}
	C \|v\|_{L^2(V_\ka)}^2 \le \ka^2 \|\nabla v\|_{L^2(V_\ka)}^2 +\ka \|v\|_{L^2(\G)}^2
\end{equation}
where
\[
	\G=\curb{(\xi,f(\xi)+\ka) :\ \xi\in B(r) } .
\]
Since $\ka$ is bounded at a positive distance below $\ka_0$, we have
\begin{equation}
	\|v\|_{L^2(\G)}^2 \le C\|v\|_{H^1(\wt{V})}^2 .
\end{equation}
Combining the above we obtain
\begin{equation}\label{eq:es_U_1}
	C_1 \|v\|_{L^2(V_{\ka})}^2
		\le \ka^{2}\|\nabla v\|_{L^2(V_{\ka})}^2 +\ka\|v\|_{H^1(\wt{V})}^2
\end{equation}

Since $V_{\ka}\sbt \cC\cap A$, we have
\begin{equation}
	\|\nabla v\|_{L^2(V_{\ka})}^2 \le 
			\|\nabla v\|_{L^2(\cC\cap A)}^2
		= \|\nabla v\|_{L^2(\cC\cap U_{\del})}^2
			+\|\nabla v\|_{L^2(\cC\cap A_\del)}^2
\end{equation}
Substituting this into \eqref{eq:es_U_1} and using $\wt{V}\sbt \cC\cap A_\del$ yields
\begin{equation}
	C\|v\|_{L^2(V_{\ka})}^2
		\le \del^{2}\|\nabla v\|_{L^2(U_\del)}^2
			+\del(1+\del)\|v\|_{H^1(\cC\cap A_\del)}^2 .
\end{equation}
Since $\p A$ is compact, we can cover it by a finite number of neighborhoods $\cC$, independent of $\del$ and thus obtain
\begin{equation}
	C\|v\|_{L^2(U_\del)}^2
		\le \del^{2}\|\nabla v\|_{L^2(U_\del)}^2
			+\del(1+\del)\|v\|_{H^1(A_\del)}^2
\end{equation}
which completes the proof
\end{proof}

\bibliographystyle{unsrt}		
\bibliography{ASI_bib, main-adaptive-eigenspace, NumericalExamples}


\end{document}